\documentclass[12pt,reqno]{amsart}
\usepackage{enumerate}
\usepackage{ifthen}
\usepackage{fullpage}
\usepackage{mathrsfs,amssymb,graphicx,verbatim}
\usepackage{paralist}
\usepackage[breaklinks,pdfstartview=FitH]{hyperref}
\newcommand{\mnote}[1]{}
\renewcommand{\hat}{\widehat}

\newcommand{\n}{\{1,\ldots,n\}}
\newcommand{\e}{\varepsilon}
\renewcommand{\c}{\mathscr{C}}
\newcommand{\R}{\mathbb R}
\newcommand{\1}{\mathbf 1}
\newcommand{\Rad}{\mathrm{\bf Rad}}
\newtheorem{theorem}{Theorem}[section]
\newtheorem{proposition}[theorem]{Proposition}
\newtheorem{lemma}[theorem]{Lemma}
\newtheorem{corollary}[theorem]{Corollary}

\newcommand{\eqdef}{\stackrel{\mathrm{def}}{=}}
\theoremstyle{remark}

\newtheorem{remark}[theorem]{Remark}
\theoremstyle{definition}
\newtheorem{definition}[theorem]{Definition}
\newcommand{\f}{\varphi}
\renewcommand{\le}{\leqslant}
\renewcommand{\ge}{\geqslant}

\newcommand{\diam}{\mathrm{diam}}
\newcommand{\zigzag}{{\text{\textcircled z}}}
\newcommand{\os}{\text{\textcircled s}}
\newcommand{\ob}{\text{\textcircled b}}

\newcommand{\A}{\mathscr A}
\newcommand{\C}{\mathbb C}
\renewcommand{\d}{\delta}

\DeclareMathOperator*{\E}{\mathbb{ E}}

\newcommand{\m}{\{1,\ldots,m\}}

\newcommand{\F}{\mathbb F}
\newcommand{\N}{\mathbb N}
\newcommand{\Z}{\mathbb Z}

\renewcommand{\setminus}{\smallsetminus}
\newcommand{\K}{\mathscr{K}}

\newcommand{\oz}{\zigzag}
\newcommand{\circr}{\text{\textcircled r}}

\newcommand{\bpconst}{\gamma_+}
\newcommand{\bp}{\gamma_{+}}

\begin{document}

 \title[Nonlinear spectral gaps]{Nonlinear spectral calculus and super-expanders}

\author[Manor
Mendel]{Manor Mendel}
%\thanks{MM was partially supported by ISF grant no. 221/07,
%BSF grant no. 2006009, and
%a gift from Cisco research center}
\address{Mathematics and Computer Science Department, Open University of Israel, 1 University Road, P.O. Box 808
Raanana 43107, Israel}
\email{mendelma@gmail.com}
\author[Assaf Naor]{Assaf Naor}
\address{Courant Institute, New York University, 251 Mercer Street, New York NY 10012, USA}
\email{naor@cims.nyu.edu}

\date{}

\begin{abstract}
Nonlinear spectral gaps with respect to uniformly convex normed spaces are shown to satisfy a spectral calculus inequality that establishes their decay along Ces\`aro averages. Nonlinear spectral gaps of graphs are also shown to behave sub-multiplicatively under zigzag products. These results yield a combinatorial construction of super-expanders, i.e., a sequence of
$3$-regular graphs
that does not admit a coarse embedding into any uniformly convex normed space.
\end{abstract}
%\newpage

\maketitle

\setcounter{tocdepth}{4}

{\footnotesize \tableofcontents}

%\newpage

\section{Introduction}

Let $A=(a_{ij})$ be an $n\times n$ symmetric stochastic matrix and
let $$1=\lambda_1(A)\ge \lambda_2(A)\ge \cdots\ge \lambda_n(A)\ge
-1$$ be its eigenvalues. The reciprocal of the spectral gap of $A$,
i.e., the quantity $\frac{1}{1-\lambda_2(A)}$, is the smallest  $\gamma\in (0,\infty]$
such that for every $x_1,\ldots,x_n\in \R$ we have
\begin{equation}\label{eq:R-poin}
\frac{1}{n^2}\sum_{i=1}^n\sum_{j=1}^n(x_i-x_j)^2\le
\frac{\gamma}{n}\sum_{i=1}^n\sum_{j=1}^n a_{ij} (x_i-x_j)^2.
\end{equation}
%\footnote{Indeed, to verify~\eqref{eq:R-poin} we may assume that
%$\sum_{i=1}^nx_i=0$, so that writing $x=(x_1,\ldots,x_n)\in \R^n$
%the difference between the right-hand side and left-hand side
%of~\eqref{eq:R-poin} becomes $\frac{2}{n}\left\langle
%\big((\gamma-1)I-\gamma A\big)x,x\right\rangle$, which is
%non-negative for all $x\in \R^n$ which is perpendicular to the
%all-ones vector if and only if $\gamma\ge
%\frac{1}{1-\lambda_2(A)}$.}.
By summing over the coordinates with respect to some orthonormal
basis, a restatement of~\eqref{eq:R-poin} is that
$\frac{1}{1-\lambda_2(A)}$ is the smallest $\gamma\in (0,\infty]$ such
that for all $x_1,\ldots,x_n\in L_2$ we have
\begin{equation}\label{eq:L_2-poin} \frac{1}{n^2}\sum_{i=1}^n\sum_{j=1}^n\|x_i-x_j\|_2^2\le
\frac{\gamma}{n}\sum_{i=1}^n\sum_{j=1}^n a_{ij} \|x_i-x_j\|_2^2.
\end{equation}

 It
is natural to generalize~\eqref{eq:L_2-poin} in several ways: one can
replace the exponent $2$ by some other exponent $p>0$ and, much more
substantially, one can replace the Euclidean geometry by some other
metric space $(X,d_X)$. Such generalizations are standard practice in
metric geometry. For the sake of
presentation, it is beneficial to take this generalization to even greater
extremes, as follows. Let $X$ be an arbitrary set and let $K: X\times X\to
[0,\infty)$ be a symmetric function. Such functions are sometimes called
{\em kernels} in the literature, and we shall adopt this terminology
here. Define the reciprocal spectral gap of $A$ with respect to
$K$, denoted $\gamma(A,K)$, to be the infimum over those $\gamma\in (0,\infty]$ such that for all $x_1,\ldots,x_n\in X$ we have
\begin{equation}\label{eq:kernel-poin} \frac{1}{n^2}\sum_{i=1}^n\sum_{j=1}^nK(x_i,x_j)\le
\frac{\gamma}{n}\sum_{i=1}^n\sum_{j=1}^n a_{ij}K(x_i,x_j).
\end{equation}

In what follows we will also call $\gamma(A,K)$ the Poincar\'e
constant of the matrix $A$ with respect to the kernel $K$. Readers are encouraged to
focus on the geometrically meaningful case when $K$ is a power of some metric on $X$, though
as will become clear presently, a surprising amount of ground can be
covered without any assumption on the kernel $K$.

For concreteness
we restate the above discussion: the standard gap in the linear
spectrum of $A$ corresponds to considering Poincar\'e constants with
respect to Euclidean spaces (i.e., kernels which are squares of
Euclidean metrics), but there is scope for a theory of nonlinear
spectral gaps when one considers inequalities such
as~\eqref{eq:kernel-poin} with respect to other geometries.
The
purpose of this paper is to make progress towards such a theory,
with emphasis on possible extensions of spectral calculus to nonlinear
(non-Euclidean) settings. We apply our results on calculus for nonlinear
spectral gaps to construct new strong types of expander graphs, and to resolve a question of V. Lafforgue~\cite{Lafforgue}. We
obtain a combinatorial construction of a remarkable type of
bounded degree graphs whose shortest path metric is incompatible with the geometry of any uniformly convex normed space in a very strong sense (i.e., coarse non-embeddability).
The existence of such graph families was first discovered by
Lafforgue via a tour de force algebraic construction~\cite{Lafforgue} . Our work indicates that there
is hope for a useful and rich theory of nonlinear spectral gaps, beyond the
sporadic (though often highly nontrivial) examples that have been previously studied in the
literature.

\subsection{Coarse non-embeddability}\label{sec:coarse}
A sequence of metric spaces $\{(X_n,d_{X_n})\}_{n=1}^\infty$ is said to embed coarsely
(with the same moduli) into a metric space $(Y,d_Y)$ if there exist two non-decreasing functions $
\alpha,\beta:[0,\infty)\to [0,\infty)
 $
 such that
 $
 \lim_{t\to\infty}\alpha(t)=\infty,
  $
  and there exist mappings $f_n:X_n\to Y$, such that for all $n\in\N$ and $x,y\in X_n$ we have
\begin{equation}\label{eq:coarse condition}
\alpha\left(d_{X_n}(x,y)\right)\le d_Y(f_n(x),f_n(y))\le \beta\left(d_{X_n}(x,y)\right).
\end{equation}

\eqref{eq:coarse condition} is a weak form of ``metric faithfulness" of the mappings $f_n$; a seemingly humble  requirement that can be restated informally as ``large distances map uniformly to large distances". Nevertheless, this weak notion of embedding (much weaker than, say, bi-Lipschitz embeddability) has beautiful applications in geometry and group theory; see~\cite{Gro93,Yu,CCJJV01,Roe03,GHW05}
and the references therein for examples of such applications.

Since coarse embeddability is a weak requirement, it is quite difficult to prove coarse non-embeddability. Very few methods to establish such a result are known, among which is the use of nonlinear spectral gaps, as pioneered by Gromov~\cite{Gromov-random-group} (other such methods include coarse notions of metric dimension~\cite{Gro93}, or the use of metric cotype~\cite{MN-cotype}. These methods do not seem to be applicable to the question that we study here). Gromov's argument is simple: fix $d\in \N$ and suppose that $X_n=(V_n,E_n)$ are connected $d$-regular graphs and that $d_{X_n}(\cdot,\cdot)$ is the shortest-path metric induced by $X_n$ on $V_n$. Suppose also that there exist $p,\gamma\in (0,\infty)$ such that for every $n\in \N$ and $f:V_n\to Y$ we have
\begin{equation}
\label{eq:graph poincare}
\frac{1}{|V_n|^2}\sum_{(u,v)\in V_n\times V_n}d_Y(f(u),f(v))^p
\le \frac{\gamma}{d|V_n|}\sum_{(x,y)\in E_n} d_Y(f(x),f(y))^p.
\end{equation}
%\ifsodaelse{\end{multline}}{\end{equation}}

A combination of~\eqref{eq:coarse condition} and~\eqref{eq:graph poincare} yields the bound
$$
\frac{1}{|V_n|^2}\sum_{(u,v)\in V_n\times V_n}\alpha\left(d_{X_n}(u,v)\right)^p\le \gamma\beta(1)^p.
 $$
 But, since $X_n$ is a bounded degree graph, at least half of the pairs of vertices $(u,v)\in V_n\times V_n$ satisfy $d_{X_n}(u,v)\ge c_d\log |V_n|$, where $c_d\in (0,\infty)$ depends on the degree $d$ but not on $n$. Thus $\alpha(c_d\log |V_n|)^p\le 2\gamma\beta(1)^p$, and in particular if $\lim_{n\to \infty}|V_n|=\infty$ then we get a contradiction to the assumption $\lim_{t\to\infty}\alpha(t)=\infty$. Observe in passing that this argument also shows that the metric space $(X_n,d_{X_n})$ has bi-Lipschitz distortion $\Omega(\log |V_n|)$ in $Y$; such an argument was first used by Linial, London and Rabinovich~\cite{LLR} (see also~\cite{Mat97}) to show that Bourgain's embedding theorem~\cite{Bourgain-embed} is asymptotically sharp.

Assumption~\eqref{eq:graph poincare} can be restated as saying that $\gamma(A_n,d_Y^p)\le \gamma$, where $A_n$ is the normalized adjacency matrix of $X_n$. This condition can be viewed to mean that the graphs $\{X_n\}_{n=1}^\infty$ are ``expanders" with respect to $(Y,d_Y)$.
Note that if $Y$ contains at least two points then~\eqref{eq:graph poincare} implies that $\{X_n\}_{n=1}^\infty$ are necessarily also expanders in the classical sense (see~\cite{HLW,Lub12} for more on classical expanders).

A key goal in the coarse non-embeddability question is therefore to construct such $\{X_n\}_{n=1}^\infty$ for which one can prove the inequality~\eqref{eq:graph poincare} for non-Hilbertian targets $Y$. This question has been previously investigated by several authors.  Matou\v{s}ek~\cite{Mat97} devised an extrapolation method for Poincar\'e inequalities
(see also the description of Matou\v{s}ek's argument  in~\cite{BLMN05})
which establishes the validity of~\eqref{eq:graph poincare} for every expander when $Y=L_p$. Works of Ozawa~\cite{Ozawa} and Pisier~\cite{pisier-79,pisier-2008} prove~\eqref{eq:graph poincare} for every expander if $Y$ is Banach space which satisfies certain geometric conditions (e.g. $Y$ can be taken to be a Banach lattice of finite cotype; see~\cite{LT79} for background on these notions). In~\cite{NS11,NR-2005} additional results of this type are obtained.

A normed space is called {\em super-reflexive} if it admits an equivalent norm which is uniformly convex. Recall that a normed space $(X,\|\cdot\|_X)$ is uniformly convex if for every $\e\in (0,1)$  there exists $\delta=\delta_X(\e)>0$ such that for any two vectors $x,y\in X$ with $\|x\|_X=\|y\|_X=1$ and $\|x-y\|_X\ge \e$ we have $\left\|\frac{x+y}{2}\right\|_X\le 1-\delta$. %(thus uniform convexity is a uniform version of {\em strict convexity}).
The question whether there exists a sequence of arbitrarily large regular graphs of bounded degree which do not admit a coarse embedding into any super-reflexive normed space was posed by Kasparov and Yu in~\cite{KY06}, and was solved in the remarkable work of V. Lafforgue~\cite{Lafforgue}  on the strengthened version of property $(T)$ for $SL_3(\mathbb F)$ when $\mathbb F$ is a non-Archimedian local field (see also~\cite{BFGM07,Laf10}). Thus, for concreteness, Lafforgue's graphs can be obtained as Cayley graphs of finite quotients of co-compact lattices in $SL_3(\mathbb Q_p)$, where $p$ is a prime and $\mathbb Q_p$ is the $p$-adic rationals. The potential validity of the same property for finite quotients of $SL_3(\Z)$ remains an intriguing open question~\cite{Lafforgue}.

Here we obtain a different solution of the Kasparov-Yu problem via a new approach that uses the {\em zigzag product} of
Reingold, Vadhan, and Wigderson~\cite{RVW}, as well as a variety of analytic and geometric arguments of independent interest.
%\mnote{Added the following sentence}
More specifically, we construct a family of 3-regular graphs that satisfies~\eqref{eq:graph poincare} for every super-reflexive Banach space $X$ (where $\gamma$ depends only on the geometry $X$); such graphs are called \emph{super-expanders}.

\begin{theorem}[Existence of super-expanders]\label{thm:existence intro}
There exists a sequence of $3$-regular graphs $\{G_n=(V_n,E_n)\}_{n=1}^\infty$ such that $\lim_{n\to \infty} |V_n|=\infty$ and  for every super-reflexive Banach space $(X,\|\cdot\|_X)$ we have
$$
\sup_{n\in \N}\gamma\left(A_{G_n},\|\cdot\|_X^2\right)<\infty,
$$
where $A_{G_n}$ is the normalized adjacency matrix of $G_n$.
\end{theorem}
As we explained earlier, the existence of super-expanders was previously proved by Lafforgue~\cite{Lafforgue}. Theorem~\ref{thm:existence intro} yields a second construction of such graphs (no other examples are currently known). Our proof of Theorem~\ref{thm:existence intro} is entirely different from Lafforgue's approach: it is based on a new systematic investigation of nonlinear spectral gaps and an elementary procedure which starts with a given small graph and iteratively increases its size so as to obtain the desired graph sequence. In fact, our study of nonlinear spectral gaps constitutes the main contribution of this work, and the new solution of the Kasparov-Yu problem should be viewed as an illustration of the applicability of our analytic and geometric results, which will be described in detail presently.

We state at the outset that it is a major open question whether every expander graph sequence satisfies~\eqref{eq:graph poincare} for every uniformly convex normed space $X$.
%In the full version we give a simple example that rules out the possibility that every two expander families coarsely
%embed in each other (which would have proved
%
It is also unknown whether there exist graph families of bounded degree and logarithmic girth that do not admit a coarse embedding into any super-reflexive normed space; this question is of particular interest in the context of the potential application to the Novikov conjecture that was proposed by Kasparov and Yu in~\cite{KY06}, since it would allow one to apply Gromov's random group construction~\cite{Gromov-random-group} with respect to actions on super-reflexive spaces.

Some geometric restriction on the target space $X$ must be imposed in order for it to admit a sequence of expanders. Indeed, the relation between nonlinear spectral gaps and coarse non-embeddability, in conjunction with the fact that every finite metric space embeds isometrically into $\ell_\infty$, shows that (for example) $X=\ell_\infty$ can never satisfy~\eqref{eq:graph poincare} for a
family of graphs of bounded  degree and unbounded cardinality. We conjecture that for a normed space $X$ the existence of such a graph family is equivalent to having finite cotype, i.e., that there exists $\e_0\in (0,\infty)$ and $n_0\in \N$ such that any embedding of $\ell_\infty^{n_0}$ into $X$ incurs bi-Lipschitz distortion at least $1+\e_0$; see e.g.~\cite{Mau03} for background on this notion.

Our approach can also be used (see Remark~\ref{rem:lafforgue} below) to show that there exist bounded degree graph sequences which do not admit a coarse embedding into any $K$-convex normed space. A normed space $X$ is $K$-convex\footnote{$K$-convexity is also equivalent to $X$ having Rademacher type strictly bigger than $1$, see~\cite{MS,Mau03}. The
$K$-convexity property is strictly weaker than super-reflexivity,
see~\cite{Jam74,JL75,Jam78,PX87}}. if there exists $\e_0>0$ and $n_0\in \N$ such that any embedding of $\ell_1^{n_0}$ into $X$ incurs distortion at least $1+\e_0$; see~\cite{Pisier-K-convex}.
 The question whether such graph sequences exist was asked by Lafforgue~\cite{Lafforgue}. Independently of our work,
Lafforgue~\cite{Laf09} succeeded to modify his argument so as to prove the desired coarse non-embeddability into $K$-convex spaces for his graph sequence as well.

\subsection{Absolute spectral gaps}\label{sec:absolute}The parameter $\gamma(A,K)$ will reappear later, but for several purposes we need to first study a variant of it which corresponds to the absolute spectral gap of a matrix. %(similar to the role of absolute spectral gaps in the work of Reingold-Vadhan-Wigderson~\cite{RVW}).
Define
\begin{equation*}\label{eq:def lambda scalar}
\lambda(A)\eqdef\max_{i\in \{2,\ldots,n\}} |\lambda_i(A)|,
\end{equation*}
 and call the quantity $1-\lambda(A)$ the absolute spectral gap of $A$. Similarly to~\eqref{eq:L_2-poin}, the reciprocal of the absolute spectral gap of $A$ is the smallest $\gamma_+\in (0,\infty]$ such
that for all $x_1,\ldots,x_n,y_1,\ldots,y_n\in L_2$ we have
\begin{equation}\label{eq:L_2-poin+} \frac{1}{n^2}\sum_{i=1}^n\sum_{j=1}^n\|x_i-y_j\|_2^2\le
\frac{\gamma_+}{n}\sum_{i=1}^n\sum_{j=1}^n a_{ij} \|x_i-y_j\|_2^2.
\end{equation}
%Since~\eqref{eq:L_2-poin+} might be less standard than~\eqref{eq:L_2-poin} we now briefly explain its validity.
Analogously to~\eqref{eq:kernel-poin}, given a kernel $K:X\times X\to [0,\infty)$ we can then define $\gamma_+(A,K)$ to be the the infimum over those $\gamma_+\in (0,\infty]$ such that for all $x_1,\ldots,x_n,y_1,\ldots,y_n\in X$ we have
\begin{equation}\label{eq:kernel-poin+} \frac{1}{n^2}\sum_{i=1}^n\sum_{j=1}^nK(x_i,y_j)\le
\frac{\gamma_+}{n}\sum_{i=1}^n\sum_{j=1}^n a_{ij}K(x_i,y_j).
\end{equation}
Note that clearly $\gamma_+(A,K)\ge \gamma(A,K)$. Additional useful relations between $\gamma(\cdot,\cdot)$ and $\gamma_+(\cdot,\cdot)$ are discussed in Section~\ref{sec:gamma gamma+}.

\subsection{A combinatorial approach to the existence of super-expanders}\label{sec:comb intro}
In what follows we will often deal with finite non-oriented regular graphs, which will always be allowed to have self loops and multiple edges (note that the shortest-path metric is not influenced by multiple edges or self loops). When discussing a graph $G=(V,E)$ it will always be understood that $V$ is a finite set and $E$ is a {\em multi-subset} of the ordered pairs $V\times V$, i.e., each ordered pair $(u,v)\in V\times V$ is allowed to appear in $E$ multiple times\footnote{Formally, one can alternatively think of $E$ as a subset of $(V\times V)\times \N$, with the understanding that for $(u,v)\in V\times V$, if we write $J=\{j\in \N:\ ((u,v),j)\in E\}$ then $\{(u,v)\}\times J$ are the $|J|$ ``copies" of $(u,v)$ that appear in $E$. However, it will not be necessary to use such formal notation in what follows.}. We also always impose the condition $(u,v)\in E\implies (v,u)\in E$, corresponding to the fact that $G$ is not oriented. For $(u,v)\in V\times V$ we denote by $E(u,v)=E(v,u)$ the number of times that $(u,v)$ appears in $E$. Thus, the graph $G$ is completely determined by the integer matrix $(E(u,v))_{(u,v)\in V\times V}$. The degree of $u\in V$ is $\deg_G(u)=\sum_{v\in V}E(u,v)$. Under this convention each self loop contributes $1$ to the degree of a vertex. For $d\in \N$, a graph $G=(V,E)$ is $d$-regular if $\deg_G(u)=d$ for every $u\in V$. The normalized adjacency matrix of a $d$-regular graph $G=(V,E)$, denoted $A_G$, is defined as usual by letting its entry at $(u,v)\in V\times V$ be equal to $E(u,v)/d$. When discussing Poincar\'e constants we will interchangeably identify $G$ with $A_G$. Thus, for examples, we write $\lambda(G)=\lambda(A_G)$ and  $\gamma_+(G,K)=\gamma_+(A_G,K)$.

The starting point of our work is an investigation of the behavior
of the quantity $\gamma_+(G,K)$ under certain graph products, the
most important of which (for our purposes) is the zigzag product of
Reingold, Vadhan and Wigderson~\cite{RVW}. We argue below that
such combinatorial constructions are well-adapted to controlling the
nonlinear quantity $\gamma_+(G,K)$. This crucial fact
allows us to use them in a perhaps unexpected geometric context.

\subsubsection{The iterative strategy}\label{sec:strategy}

Reingold, Vadhan and Wigderson~\cite{RVW} introduced the zigzag
product of graphs, and used it to produce a novel deterministic
construction of expanders. Fix $n_1,d_1,d_2\in \N$. Let $G_1$ be a
graph with $n_1$ vertices which is $d_1$-regular and let $G_2$ be a
graph with $d_1$ vertices which is $d_2$-regular. The zigzag
product $G_1\oz G_2$ is a graph with $n_1d_1$ vertices and degree
$d_2^2$, for which the
following fundamental theorem is proved in~\cite{RVW}.
\begin{theorem}[Reingold,
Vadhan and Wigderson]\label{thm:RVW intro} There exists
$f:[0,1]\times [0,1]\to [0,1]$ satisfying
\begin{equation}\label{eq:RVW property}
\forall\, t\in (0,1),\quad \limsup_{s\to 0} f(s,t)<1,
\end{equation}
such that for every $n_1,d_1,d_2\in \N$, if $G_1$ is a graph with
$n_1$ vertices which is $d_1$-regular and $G_2$ is a graph with
$d_2$ vertices which is $d_2$-regular then
\begin{equation}\label{eq:RVW bound}
\lambda(G_1\oz G_2)\le f(\lambda(G_1),\lambda(G_2)).
\end{equation}
\end{theorem}
The definition of $G_1\oz G_2$ is recalled in Section~\ref{sec:products}. For the purpose of expander
constructions one does not need to know anything about the zigzag
product other than that it has $n_1d_1$ vertices and degree $d_2^2$,
and that it satisfies Theorem~\ref{thm:RVW intro}.  Also, \cite{RVW} contains explicit algebraic expressions for functions  $f$  for which
Theorem~\ref{thm:RVW intro} holds true, but we do not need to quote them
here because they are irrelevant to the ensuing discussion.

 In order to proceed it would be instructive to briefly recall how Reingold, Vadhan and Wigderson used~\cite{RVW} Theorem~\ref{thm:RVW intro}   to construct expanders; see also the exposition in Section~9.2 of~\cite{HLW}.

Let $H$ be a regular graph with $n_0$ vertices and degree
$d_0$, such that $\lambda(H)<1$. Such a graph $H$ will be called a {\em base graph} in what
follows. From~\eqref{eq:RVW property} we deduce that there exist $\e,\delta\in (0,1)$ such that
\begin{equation}\label{eq:epsilon delta}
s\in (0,\delta)\implies f(s,\lambda(H))<1-\e.
\end{equation}
Fix $t_0\in \N$ satisfying
\begin{equation}\label{eq:power made small}
\max\left\{\lambda(H)^{2t_0},(1-\e)^{t_0}\right\}< \delta.
\end{equation}

 For a graph $G=(V,E)$ and for $t\in \N$, let $G^t$ be the graph in which an edge between $u,v\in V$ is drawn for every walk in $G$ of length $t$ whose endpoints are $u,v$. Thus $A_{G^t}=(A_G)^t$, and if $G$ is $d$-regular then $G^t$ is $d^t$-regular.
%\begin{equation}\label{eq:calculus identity}
%\lambda(G^t)=\lambda(G)^t.
%\end{equation}

Assume from now on that $n_0=d_0^{2t_0}$. Define $G_1=H^{2}$ and inductively
$$
G_{i+1}=G_i^{t_0}\oz H.
$$
Then for all $i\in \N$ the graph $G_i$ is well defined and has
$n_0^{i}=d_0^{2it_0}$ vertices and degree $d_0^2$. We claim that $\lambda(G_i)\le \max\{\lambda(H)^2,1-\e\}$ for all $i\in \N$. Indeed, there is nothing to prove for $i=1$, and if the desired bound is true for $i$ then~\eqref{eq:power made small} implies that $\lambda(G_i^{t_0})=\lambda(G_i)^{t_0}<\delta$, which by~\eqref{eq:RVW bound} and~\eqref{eq:epsilon delta}  implies that $\lambda(G_{i+1})\le f(\lambda(G_i^{t_0}),\lambda(H))<1-\e$.

Our strategy is to attempt to construct super-expanders via a similar iterative approach. It turns out that obtaining a non-Euclidean version of Theorem~\ref{thm:RVW intro} (which is the seemingly most substantial ingredient of the construction of Reingold, Vadhan and Wigderson) is not an obstacle here due to the following general result.

\begin{theorem}[Zigzag sub-multiplicativity]\label{thm:sub} Let $G_1=(V_1,E_1)$ be an $n_1$-vertex graph which is $d_1$-regular and let $G_2=(V_2,E_2)$ be a $d_1$-vertex graph which is $d_2$-regular. Then every kernel $K:X\times X\to [0,\infty)$ satisfies
\begin{equation}\label{eq:sub}
\gamma_+\left(G_1\oz G_2,K\right)\le \gamma_+(G_1,K)\cdot \gamma_+(G_2,K)^2.
\end{equation}
\end{theorem}
%In Section~\ref{sec:details-zigzag} we define the zigzag product $G_1\oz G_2$
%without the simplifying assumption on $G_1$ used here, and prove Theorem~\ref{thm:sub} unconditionally.

In the special case $X=\R$ and $K(x,y)=(x-y)^2$, Theorem~\ref{thm:sub} becomes
\begin{equation}\label{eq:RVW case}
\frac{1}{1-\lambda(G_1\oz G_2)}\le \frac{1}{1-\lambda(G_1)}\cdot \frac{1}{(1-\lambda(G_2))^2},
\end{equation}
implying Theorem~\ref{thm:RVW intro}. Note that the explicit bound on the function $f$ of Theorem~\ref{thm:RVW intro} that follows from~\eqref{eq:RVW case} coincides with the later bound of
Reingold, Trevisan and Vadhan~\cite{RTV06}. In~\cite{RVW} an
improved bound for $\lambda(G_1\oz G_2)$ is obtained which is better
than the bound of~\cite{RTV06} (and hence also~\eqref{eq:RVW case}),
though this improvement in lower-order terms has not been used (so far) in the
literature. Theorem~\ref{thm:sub} shows that the fact that the zigzag
product preserves spectral gaps  has nothing to do with the
underlying Euclidean geometry (or linear algebra) that was used
in~\cite{RVW,RTV06}: this is a truly nonlinear phenomenon which
holds in much greater generality, and simply amounts to an iteration
of the Poincar\'e inequality~\eqref{eq:kernel-poin+}.

Due to Theorem~\ref{thm:sub} there is hope to carry out an iterative construction based on the zigzag product in great generality. However, this cannot work for all kernels  since general kernels can fail to admit a sequence of bounded degree expanders. There are two major obstacles that need to be overcome. The first obstacle is the existence of a base graph, which is a substantial issue whose discussion is deferred to Section~\ref{sec:base intro}. The following subsection describes the main obstacle to our nonlinear zigzag strategy.

\subsubsection{The need for a calculus for nonlinear spectral gaps}\label{sec:calculus intro}
In the above description of the Reingold-Vadhan-Wigderson iteration we tacitly used the identity $\lambda(A^t)=\lambda(A)^t$ ($t\in \N$) in order to increase the spectral gap of $G_i$ in each step of the iteration. While this identity is a trivial corollary of spectral calculus, and was thus the ``trivial part" of the construction in~\cite{RVW}, there is no reason to expect that $\gamma_+(A^t,K)$ decreases similarly with $t$ for non-Euclidean kernels $K:X\times X\to [0,\infty)$. To better grasp what is happening here let us examine the asymptotic behavior of $\gamma_+(A^t,|\cdot|^2)$ as a function of $t$ (here and in what follows $|\cdot |$ denotes the absolute value on $\R$).
\begin{multline}\label{eq:euclidean decay}
\gamma_+\left(A^t,|\cdot|^2\right)=\frac{1}{1-\lambda(A^t)}=\frac{1}{1-\lambda(A)^t}
\\
=\frac{1}{1-\left(1-\frac{1}{\gamma_+\left(A,|\cdot|^2\right)}\right)^t}\asymp \max\left\{1,\frac{\gamma_+\left(A,|\cdot|^2\right)}{t}\right\},
\end{multline}
where above, and in what follows, $\asymp$ denotes equivalence up to universal multiplicative constants (we will also use the notation $\lesssim,\gtrsim$ to express the corresponding inequalities up to universal constants). \eqref{eq:euclidean decay} means that raising a matrix to a large power $t\in \N$ corresponds to decreasing its (real) Poincar\'e constant by a factor of $t$ as long as it is possible to do so.

For our strategy to work for other kernels $K:X\times X\to [0,\infty)$
we would like $K$ to satisfy a  ``spectral calculus" inequality of
this type, i.e., an inequality which ensures that, if
$\gamma_+(A,K)$ is large, then $\gamma_+(A^t,K)$ is much smaller
than $\gamma_+(A,K)$ for sufficiently large $t\in \N$. This is, in
fact, not the case in general: in Section~\ref{sec:no-decay} we
construct  a metric space $(X,d_X)$ such that for each $n\in \N$
there is a symmetric stochastic matrix $A_n$ such that
$\gamma_+(A_n,d_X^2)\ge n$ yet for every $t\in \N$ there is $n_0\in
\N$ such that for all $n\ge n_0$ we have $\gamma_+(A_n^t,d_X^2)\gtrsim \gamma_+(A_n,d_X^2)$. The question which metric spaces
satisfy the desired nonlinear spectral calculus inequality thus becomes a
subtle issue which we believe is of fundamental importance, beyond
the particular application that we present here. A large part of the
present paper is devoted to addressing this question. We obtain
rather satisfactory results which allow us to carry out a zigzag
type construction of super-expanders, though we are still quite far from a
complete understanding of the behavior of nonlinear spectral gaps
under graph powers for non-Euclidean geometries.

\subsubsection{Metric Markov cotype and spectral calculus}\label{sec:MMC}  We will introduce a criterion for a metric space $(X,d_X)$, which is a bi-Lipschitz invariant, and prove that it implies that for every $n,m\in\N$ and every $n\times n$ symmetric stochastic matrix $A$ the  Ces\`aro averages $\frac{1}{m}\sum_{t=0}^{m-1}A^t$ satisfy the following spectral calculus inequality.
\begin{equation}\label{eq:decay for gamma_+}
\gamma_+\left(\frac{1}{m}\sum_{t=0}^{m-1}A^t,d_X^2\right)\le C(X)\max\left\{1,\frac{\gamma_+\left(A,d_X^2\right)}{m^{\e(X)}}\right\},
\end{equation}
where $C(X),\e(X)\in (0,\infty)$ depend only on the geometry of $X$ but not on $m,n$ and the matrix $A$. The fact that we can only prove such an inequality for Ces\`aro averages rather than powers does not create any difficulty in the ensuing argument, since Ces\`aro averages are  compatible with iterative graph constructions based on the zigzag product.

Note that Ces\`aro averages have the following combinatorial interpretation in the case of graphs. Given an $n$-vertex $d$-regular graph $G=(V,E)$ let $\A_m(G)$ be the graph whose vertex set is $V$ and for every $t\in \{0,\ldots,m-1\}$ and $u,v\in V$ we draw $d^{m-1-t}$ edges joining $u,v$ for every walk in $G$ of length $t$ which starts at $u$ and terminates at $v$. With this definition $A_{\A_m(G)}=\frac{1}{m}\sum_{t=0}^{m-1}A_G^t$, and $\A_m(G)$ is $md^{m-1}$-regular. %\mnote{Added the following sentence}
We will slightly abuse this  notation by also using the  shorthand
\begin{equation}\label{eq:def cesaro notation}
\A_m(A)\eqdef\frac 1m \sum_{t=0}^{m-1}A^t,
\end{equation}
 when $A$ is an $n\times n$ matrix.

In the important paper~\cite{Ball} K. Ball introduced a {\em linear} property of
Banach spaces that he called Markov cotype $2$, and he indicated a
two-step definition that could be used to extend this notion to
general metric spaces. Motivated by Ball's ideas, we consider the following
variant of his definition.
\begin{definition}[Metric Markov cotype]\label{def:metric markov cotype} Fix $p,q\in (0,\infty)$. A
metric space $(X,d_X)$ has {\em metric Markov cotype $p$ with exponent $q$} if there exists $C\in (0,\infty)$ such that  for every $m,n\in
\N$, every $n\times n$ symmetric stochastic matrix $A=(a_{ij})$, and
every $x_1,\ldots,x_n\in X$, there exist $y_1,\ldots,y_n\in X$
satisfying
\begin{equation}\label{def:markov cotype}
\sum_{i=1}^n d_X(x_i,y_i)^q+m^{q/p}\sum_{i=1}^n\sum_{j=1}^n a_{ij}
d_X(y_i,y_j)^q\le C^q\sum_{i=1}^n\sum_{j=1}^n
\A_m(A)_{ij}d_X(x_i,x_j)^q.
\end{equation}
The infimum over those $C\in (0,\infty)$ for which~\eqref{def:markov cotype} holds true is denoted $C_p^{(q)}(X,d_X)$. When $q=p$ we drop the explicit mention of the exponent and simply say that if~\eqref{def:markov cotype} holds true with $q=p$ then $(X,d_X)$ has metric Markov cotype $p$.
\end{definition}

\begin{remark}\label{rem:geometric intuition cotype} We refer to~\cite[Sec.~4.1]{Nao12} for an  explanation of the background and geometric intuition that motivates the (admittedly cumbersome) terminology of Definition~\ref{def:metric markov cotype}. Briefly, the term ``cotype" indicates that this definition is intended to serve as a metric analog of the important Banach space property {\em Rademacher cotype} (see~\cite{Mau03}). Despite this fact, in the the forthcoming paper~\cite{MN12-ext} we show, using a clever idea of Kalton~\cite{Kal11}, that there exists a Banach space with Rademacher cotype $2$ that does not have metric Markov cotype $p$ for any $p\in (0,\infty)$. The term ``Markov" in Definition~\ref{def:metric markov cotype} refers to the fact that the notion of metric Markov cotype is intended to serve as a certain ``dual" to Ball's notion of {\em Markov type}~\cite{Ball}, which is a notion which is defined in terms of the geometric behavior of stationary reversible Markov chains whose state space is a finite subset of $X$.
\end{remark}

\begin{remark}\label{rem:relation to ball's cotype} Ball's original
definition~\cite{Ball} of metric Markov cotype is seemingly
different from Definition~\ref{def:metric markov cotype}, but
in~\cite{MN12-ext} we show that Definition~\ref{def:metric markov
cotype} is equivalent to Ball's definition. We introduced
Definition~\ref{def:metric markov cotype} since it directly implies
Theorem~\ref{thm:cotype implies calculus intro} below.
\end{remark}

The link between Definition~\ref{def:metric markov cotype}  and the desired spectral calculus inequality~\eqref{eq:decay for gamma_+} is contained in the following theorem, which is proved in Section~\ref{sec:cotype to calculus}.

\begin{theorem}[Metric Markov cotype implies nonlinear spectral calculus]\label{thm:cotype implies calculus intro}
Fix $p,C\in (0,\infty)$ and suppose that a metric space $(X,d_X)$ satisfies $$C_p^{(2)}(X,d_X)\le C.$$ Then for every $m,n\in \N$, every $n\times n$ symmetric stochastic matrix $A$ satisfies
$$
\gamma_+\left(\A_m(A),d_X^2\right)\le (45C)^2\max\left\{1,\frac{\gamma_+\left(A,d_X^2\right)}{m^{2/p}}\right\}.
$$
\end{theorem}

In Section~\ref{sec:martingale} we investigate the metric Markov cotype of super-reflexive Banach spaces, obtaining the following result, whose proof is inspired by Ball's insights in~\cite{Ball}.

\begin{theorem}[Metric Markov cotype for super-reflexive Banach spaces]\label{thm:markov cotype thm in intro}
Let $(X,\|\cdot\|_X)$ be  a super-reflexive Banach space. Then there exists $p=p(X)\in [2,\infty)$ such that $$C_p^{(2)}(X,\|\cdot\|_X)<\infty,$$ i.e., $(X,\|\cdot\|_X)$ has Metric Markov cotype $p$ with exponent $2$.
\end{theorem}

\begin{remark}\label{rem:CAT(0)} In our forthcoming paper~\cite{MN12-ext} we compute the metric Markov cotype of additional classes of metric spaces. In particular, we show that all $CAT(0)$ metric spaces (see~\cite{BH-book}), and hence also all complete simply connected Riemannian manifolds with nonnegative sectional curvature, have Metric Markov cotype $2$ with exponent $2$.
\end{remark}

By combining Theorem~\ref{thm:cotype implies calculus intro} and Theorem~\ref{thm:markov cotype thm in intro} we deduce the following result.
\begin{corollary}[Nonlinear spectral calculus for super-reflexive Banach spaces]\label{coro:decay for super-reflexive intro}
For every super-reflexive Banach space $(X,\|\cdot\|_X)$ there exist $\e(X),C(X)\in (0,\infty)$ such that for every $m,n\in \N$ and every $n\times n$ symmetric stochastic matrix $A$ we have
$$
\gamma_+\left(\A_m(A),\|\cdot\|_X^2\right)\le C(X)\max\left\{1,\frac{\gamma_+\left(A,\|\cdot\|_X^2\right)}{m^{\e(X)}}\right\}.
$$
\end{corollary}
\begin{remark}\label{rem:norm bound}
In Theorem~\ref{decay in uniformly convex} below we present a different approach to proving nonlinear spectral calculus inequalities in the setting of super-reflexive Banach spaces. This approach, which is based on bounding the norm of a certain linear operator, has the advantage that it establishes the decay of the Poincar\'e constant of the power $A^m$ rather than the Ces\`aro average $\A_m(A)$. While this result is of independent geometric interest, the form of the decay inequality that we are able to obtain has the disadvantage that we do not see how to use it to construct super-expanders. Moreover, we do not know how to obtain sub-multiplicativity estimates for such norm bounds under zigzag products and other graph products such as the tensor product and replacement product (see Section~\ref{sec:products intro} below). The approach based on metric Markov cotype  also has the advantage of being applicable to other classes of (non-Banach) metric spaces, in addition to its usefulness for the Lipschitz extension problem~\cite{Ball,MN12-ext}.
\end{remark}

\subsubsection{The base graph}\label{sec:base intro} In order to construct super-expanders using Theorem~\ref{thm:sub} and Corollary~\ref{coro:decay for super-reflexive intro} one must start the inductive procedure with an appropriate ``base graph". This is a nontrivial issue that raises analytic challenges which are interesting in their own right.

It is most natural to perform our construction of base graphs in the context of $K$-convex Banach spaces, which, as we recalled earlier, is a class of spaces that is strictly larger than the class of super-reflexive spaces. The result thus obtained, proved in Section~\ref{sec:base} using the preparatory work in Section~\ref{sec:decay heat} and part of Section~\ref{sec:UC}, reads as follows.

\begin{lemma}[Existence of base graphs for $K$-convex spaces]\label{lem:base in section}
There exists a strictly increasing sequence of integers  $\{m_n\}_{n=1}^\infty\subseteq \N$ satisfying
\begin{equation}
\forall\, n\in \N,\quad 2^{n/10}\le m_n\le 2^n,
\end{equation}
with the following properties. For every $\d\in (0,1]$ there is $n_0(\d)\in \N$ and a sequence of regular graphs $\{H_n(\d)\}_{n=n_0(\d)}^\infty$ such that
\begin{itemize}
\item $|V(H_n(\d))|=m_n$ for every integer $n\ge n_0(\d)$.
\item For every $n\in [n_0(\d),\infty)\cap\N$ the degree of $H_n(\d)$, denoted $d_n(\d)$, satisfies
\begin{equation}\label{eq:dndelta bound in section}
d_n(\d)\le e^{(\log m_n)^{1-\d}}.
\end{equation}
\item For every $K$-convex Banach space $(X,\|\cdot\|_X)$  we have $\gamma_+(H_n(\d),\|\cdot\|_X^2)<\infty$ for all $\d\in (0,1)$ and $n\in \N\cap [n_0(\d),\infty)$. Moreover, there exists $\d_0(X)\in (0,1)$ such that
\begin{equation}\label{eq:9^3}
\forall\, \d\in (0,\d_0(X)],\ \forall n\in [n_0(\d),\infty)\cap\N,\quad \gamma_+(H_n(\d),\|\cdot\|_X^2)\le 9^3.
\end{equation}
\end{itemize}
\end{lemma}
The bound $9^3$ in~\eqref{eq:9^3} is nothing more than an artifact of our proof and it does not play a special role in what follows: all that we will need for the purpose of constructing super-expanders is to ensure that
\begin{equation}\label{eq:weaker gamma bound comment}
\sup_{\d\in (0,\d_0(X)]}\sup_{n\in [n_0(\d),\infty)\cap\N} \gamma_+(H_n(\d),\|\cdot\|_X^2)<\infty,
\end{equation}
i.e., for our purposes the upper bound on $\gamma_+(H_n(\d),\|\cdot\|_X^2)$ can be allowed to depend on $X$. Moreover, in the ensuing arguments we can make do with a degree bound that is weaker than~\eqref{eq:dndelta bound in section}: all we need is that
\begin{equation}\label{eq:weaker degree bound comment}
\forall\, \d\in (0,1),\quad \lim_{n\to \infty} \frac{\log d_n(\d)}{\log m_n}=0.
\end{equation}
However, we do not see how to prove the weaker requirements~\eqref{eq:weaker gamma bound comment}, \eqref{eq:weaker degree bound comment} in a substantially simpler way than our proof of the stronger requirements~\eqref{eq:dndelta bound in section}, \eqref{eq:9^3}.

The starting point of our approach to construct base graphs is the ``hypercube quotient argument" of~\cite{KN06}, although in order to apply such ideas in our context we significantly modify this construction, and apply  deep methods of Pisier~\cite{Pisier-K-convex,pisier-2007}. A key analytic challenge that arises here is to bound the norm of the inverse of the hypercube Laplacian on the vector-valued {\em tail space}, i.e., the space of all functions taking values in a Banach space $X$ whose Fourier expansion is supported on Walsh functions corresponding to large sets. If $X$ is a Hilbert space then the desired estimate is an immediate consequence of orthogonality, but even when $X$ is an $L_p(\mu)$ space the corresponding inequalities are not known. P.-A. Meyer~\cite{Meyer} previously obtained $L_p$ bounds for the inverse of the Laplacian on the (real-valued) tail space, but such bounds are insufficient for our purposes. In order to overcome this difficulty, in Section~\ref{sec:heat} we obtain decay estimates for the heat semigroup on the tail space of functions taking values in a $K$-convex Banach space. We then use (in Section~\ref{sec:base}) the heat semigroup to construct a new (more complicated) hypercube quotient by a linear code which can serve as the base graph of Lemma~\ref{lem:base in section}.

The bounds on the norm of the heat semigroup on the vector valued tail space (and the corresponding bounds on the norm of the inverse of the Laplacian) that are proved in Section~\ref{sec:heat} are sufficient for the purpose of proving Lemma~\ref{lem:base in section}, but we conjecture that they are suboptimal. Section~\ref{sec:heat} contains analytic questions along these lines  whose positive solution would yield a simplification of our construction of the base graph (see Remark~\ref{rem:vanilla KN}).

With all the ingredients  in place (Theorem~\ref{thm:sub}, Corollary~\ref{coro:decay for super-reflexive intro}, Lemma~\ref{lem:base in section}), the actual iterative construction of super-expanders in performed in Section~\ref{sec:iterative}. Since we need to construct a single sequence of bounded degree graphs  that has a nonlinear spectral gap with respect to {\em all} super-reflexive Banach spaces, our implementation of the zigzag strategy is significantly more involved than the  zigzag iteration of Reingold, Vadhan and Wigderson (recall Section~\ref{sec:strategy}). This implementation itself may be of independent interest.

\subsubsection{Sub-multiplicativity theorems for graph products}\label{sec:products intro} Theorem~\ref{thm:sub} is a special case of a a larger family of sub-multiplicativity estimates for nonlinear spectral gaps with respect to certain graph products. The literature contains several combinatorial procedures to combine two graphs, and it turns out that such constructions are often highly compatible with nonlinear Poincar\'e inequalities. In Section~\ref{sec:products} we further investigate this theme.
%, in addition to analysing zigzag products as stated in Theorem~\ref{thm:sub}.

The main results of Section~\ref{sec:products} are collected in the following theorem (the relevant terminology is discussed immediately after its statement). Item~\eqref{item:zigzag} below is nothing more than a restatement of Theorem~\ref{thm:sub}.

\begin{theorem}\label{thm:products intro} Fix $m,n,n_1,d_1,d_2\in \N$. Suppose that $K:X\times X\to [0,\infty)$ is a kernel and $(Y,d_Y)$ is a metric space. Suppose also that $G_1=(V_1,E_1)$ is a $d_1$-regular graph with $n_1$ vertices and $G_2=(V_2,E_2)$ is a $d_2$-regular graph with $d_1$ vertices. Then,
\begin{enumerate}[(I)]
\item\label{item:product} If $A=(a_{ij})$ is an $m\times m$ symmetric stochastic matrix and $B=(b_{ij})$ is an $n\times n$ symmetric stochastic matrix then the {\bf tensor product} $A\otimes B$ satisfies
\begin{equation}\label{eq:tensor intro}
\gamma_+(A\otimes B,K)\le \gamma_+(A,K)\cdot \gamma_+(B,K).
\end{equation}
\item\label{item:zigzag} The {\bf zigzag product} $G_1\oz G_2$ satisfies
\begin{equation}\label{eq:zigzag item intro}
\gamma_+\left(G_1\oz G_2,K\right)\le \gamma_+(G_1,K)\cdot \gamma_+(G_2,K)^2.
\end{equation}
\item\label{item:deradomized squaring} The {\bf derandomized square} $G_1\os G_2$ satisfies
\begin{equation}\label{eq:os intro}
\gamma_+\left(G_1\os G_2,K\right)\le \gamma_+\left(G_1^2,K\right)\cdot \gamma_+(G_2,K).
\end{equation}
\item\label{item:replacement product} The {\bf replacement product} $G_1\circr G_2$ satisfies
\begin{equation}\label{eq:replacement intro}
\gamma_+\left(G_1\circr G_2,d_Y^2\right)\le 3(d_2+1)\cdot\gamma_+\left(G_1,d_Y^2\right)\cdot \gamma_+\left(G_2,d_Y^2\right)^2.
\end{equation}
\item\label{item:balanced replacement product} The {\bf balanced replacement product} $G_1\ob G_2$ satisfies
\begin{equation}\label{eq:balanced intro}
\gamma_+\left(G_1\ob G_2,d_Y^2\right)\le 6\cdot\gamma_+\left(G_1,d_Y^2\right)\cdot \gamma_+\left(G_2,d_Y^2\right)^2.
\end{equation}
\end{enumerate}
\end{theorem}

Since the $(mn)\times (mn)$ matrix $A\otimes B=(a_{ij}b_{k\ell})$ satisfies $\lambda(A\otimes B)=\max\{\lambda(A),\lambda(B)\}$, in the Euclidean case, i.e., $K:\R\times \R\to [0,\infty)$ is given by $K(x,y)=(x-y)^2$, the product in the right hand side of~\eqref{eq:tensor intro} can be replaced by a maximum. Lemma~\ref{lem:tesor sup uniformly convex} below contains a similar improvement of~\eqref{eq:tensor intro} under additional assumptions on the kernel $K$.

The definitions of the graph products $G_1\oz G_2$, $G_1\os G_2$, $G_1\circr G_2$, $G_1\ob G_2$ are recalled in Section~\ref{sec:products}. The replacement product $G_1\circr G_2$, which is a $(d_2+1)$-regular graph with $n_1d_1$ vertices, was introduced by Gromov in~\cite{Gromov-filling}, where he applied it iteratively to hypercubes of logarithmically decreasing size so as to obtain a constant degree graph which has sufficiently good expansion for his (geometric) application. In~\cite{Gromov-filling} Gromov bounded $\lambda(G_1\circr G_2)$ from above by an expression involving only $\lambda(G_1), \lambda(G_2), d_2$. Such a bound was also obtained by Reingold, Vadhan and Wigderson in~\cite{RVW}. We shall use~\eqref{eq:replacement intro}  in the proof of Theorem~\ref{thm:existence intro}.

The breakthrough of Reingold, Vadhan and Wigderson~\cite{RVW} introduced the zigzag product, which can be used to construct constant degree expanders; the fact that~\eqref{eq:zigzag item intro} holds true for general kernels $K$, while~\eqref{eq:replacement intro} assumes that $d_Y$ is a metric and incurs a multiplicative loss of $3(d_2+1)$ can be viewed as an indication why the zigzag product is a more basic operation than the replacement product.

The balanced replacement product $G_1\ob G_2$, which is a $2d_2$-regular graph with $n_1d_1$ vertices, was introduced by Reingold, Vadhan and Wigderson~\cite{RVW}, who bounded $\lambda(G_1\ob G_2)$ from above by an expression involving only $\lambda(G_1), \lambda(G_2)$.

The derandomized square $G_1\os G_2$, which is a $d_1d_2$-regular graph with $n_1$ vertices, was introduced by Rozenman and Vadhan in~\cite{RV05}, where they bounded $\lambda(G_1\os G_2)$ from above by an expression involving only $\lambda(G_1), \lambda(G_2)$. This operation is of a different nature: it aims to create a graph that has  spectral properties similar to the square $G_1^2$, but with significantly fewer edges. In~\cite{RVW,RV05} tensor products and derandomized squaring were used to improve the computational efficiency of zigzag constructions. The general bounds~\eqref{eq:tensor intro} and~\eqref{eq:os intro} can be used to improve the efficiency of our constructions in a similar manner, but we will not explicitly discuss computational efficiency issues in this paper (this, however, is relevant to our forthcoming paper~\cite{MN12-random}, where our construction is used for an algorithmic purpose).

%\subsection{The structure of this paper, subsequent work and open problems}\label{sec:subsequent}

\section{Preliminary results on nonlinear spectral gaps}\label{sec:prem}

The purpose of this section is to record some simple and elementary preliminary facts about nonlinear spectral gaps that will be used throughout this article. One can skip this section on first reading and refer back to it only when the facts presented here are used in the subsequent sections.

\subsection{The trivial bound for general graphs}\label{sec:trivial}
For $\kappa\in [0,\infty)$ a kernel $\rho:X\times X\to [0,\infty)$ is called a $2^\kappa$-quasi-semimetric if for every $x,y,z\in X$ we have
\begin{equation}\label{eq:quasimetric def}
\rho(x,y)\le 2^\kappa\left(\rho(x,z)+\rho(z,y)\right).
\end{equation}
The key examples of $2^\kappa$-quasi-semimetrics are of the form $\rho=d_X^p$, where $d_X:X\times X\to [0,\infty)$ is a semimetric and $p\in [1,\infty)$, in which case $\kappa=p-1$.

\begin{lemma}\label{lem:general graph} Fix $n,d\in \N$ and $\kappa\in [0,\infty)$.
 Let $G=(V,E)$ be a $d$-regular connected graph with  $n$ vertices. Then for every $2^\kappa$-quasi-semimetric $\rho:X\times X\to [0,\infty)$ we have
\begin{equation}\label{eq:trivial gamma bound}
\gamma(G,\rho)\le 2^{\kappa-1} d n^{\kappa+1}.
\end{equation}
If in addition $G$ is not a bipartite graph then
\begin{equation}\label{eq:trivial gamma+ bound}
\gamma_+(G,\rho)\le 2^{2\kappa}dn^{\kappa+1}.
\end{equation}
\end{lemma}

\begin{proof} For every $x,y\in V$ choose distinct $\{u_0^{x,y}=x,u_1^{x,y},\ldots,u_{m_{x,y}-1}^{x,y},u_{m_{x,y}}^{x,y}=y\}\subseteq  V$ such that $(u_i^{x,y},u_{i-1}^{x,y})\in E$ for every $i\in \{1,\ldots,m_{x,y}\}$, and $(u_i^{x,y},u_{i-1}^{x,y})\neq (u_j^{x,y},u_{j-1}^{x,y})$ for  distinct $i,j\in \{1,\ldots,m_{x,y}\}$. Fixing $f:V\to X$, a straightforward inductive application of~\eqref{eq:quasimetric def} yields
$$
\rho(f(x),f(y))\le (2m_{x,y})^\kappa\sum_{i=1}^{m_{x,y}} \rho\left(f\left(u_{i-1}^{x,y}\right),f\left(u_i^{x,y}\right)\right)\le (2n)^\kappa \sum_{i=1}^{m_{x,y}} \rho\left(f\left(u_{i-1}^{x,y}\right),f\left(u_i^{x,y}\right)\right).
$$
Thus
\begin{multline*}
\frac{1}{n^2}\sum_{(x,y)\in V\times V} \rho(f(x),f(y))\le \frac{(2n)^\kappa}{n^2}\sum_{(x,y)\in V\times V}\sum_{i=1}^{m_{x,y}} \rho\left(f\left(u_{i-1}^{x,y}\right),f\left(u_i^{x,y}\right)\right)\\\le \frac{(2n)^\kappa \binom{n}{2}}{n^2}\sum_{(a,b)\in E} \rho(f(a),f(b))
\le \frac{(2n)^\kappa nd}{2}\cdot \frac{1}{nd}\sum_{(a,b)\in E} \rho(f(a),f(b)).
\end{multline*}
This proves~\eqref{eq:trivial gamma bound}. To prove~\eqref{eq:trivial gamma+ bound} suppose that $G$ is connected but not bipartite. Then for every $x,y\in V$ there exists a path of {\em odd} length joining $x$ and $y$ whose total length is at most $2n$ and in which each edge  is repeated at most once (indeed, being non-bipartite, $G$ contains an odd cycle $c$; the desired path can be found by considering the shortest paths joining $x$ and $y$ with $c$). Let $\{w_0^{x,y}=x,w_1^{x,y},\ldots,w_{m-1}^{x,y},w_{2\ell_{x,y}+1}^{x,y}=y\}\subseteq  V$ be such a path. For every $f,g:V\to X$ we have
\begin{eqnarray*}
\sum_{(x,y)\in V\times V} \rho(f(x),g(y))&\le& \sum_{(x,y)\in V\times V}(4\ell_{x,y}+2)^\kappa\Biggl(\rho\left(f\left(w_{0}^{x,y}\right),g\left(w_{1}^{x,y}\right)\right)\\&&\quad+\sum_{i=1}^{\ell_{x,y}} \left(\rho\left(g\left(w_{2i-1}^{x,y}\right),f\left(w_{2i}^{x,y}\right)\right)+\rho\left(f\left(w_{2i}^{x,y}\right),g\left(w_{2i+1}^{x,y}\right)\right)\right)\Biggr)\\&\le& (4n)^\kappa\cdot n^2\sum_{(a,b)\in E} \rho(f(a),g(b)),
\end{eqnarray*}
implying~\eqref{eq:trivial gamma+ bound}.
\end{proof}

\begin{remark}\label{rem:cycle bounds}{
For $n\in \N$ let $C_n$ denote the $n$-cycle and let $C_n^\circ$ denote the $n$-cycle with self loops (thus $C_n^\circ$ is a $3$-regular graph). It follows from Lemma~\ref{lem:general graph} that $\gamma(C_n,\rho)\lesssim (2n)^{\kappa+1}$ and $\gamma_+(C_n^\circ,\rho)\lesssim (4n)^{\kappa+1}$ for every $2^\kappa$-quasi-semimetric. If $(X,d_X)$ is a metric space and $p\in [1,\infty)$ then one can refine the above arguments using the symmetry of the circle to get the improved bound
\begin{equation}\label{eq:cycle better p}
\gamma_+(C_n^\circ,d_X^p)\lesssim \frac{(n+1)^p}{p2^p}.
\end{equation}
We omit the proof of~\eqref{eq:cycle better p} since the improved dependence on $p$ is not used in the ensuing discussion.}
\end{remark}

\subsection{$\gamma$ versus $\gamma_+$}\label{sec:gamma gamma+}
By taking $f=g$ in the definition of $\gamma_+(\cdot,\cdot)$ one immediately sees that $\gamma(A,K)\le \gamma_+(A,K)$ for every kernel $K:X\times X\to [0,\infty)$ and every symmetric stochastic matrix $A$. Here we investigate additional relations between these quantities.

\begin{lemma}\label{lem:pass to 2-cover}
Fix $\kappa\in [0,\infty)$ and let $\rho:X\times X\to [0,\infty)$ be a $2^\kappa$-quasi-semimetric.  Then for every symmetric stochastic matrix $A$ we have
\begin{equation}\label{eq:2n}
 \frac2{2^{\kappa+1}+1} \gamma\left( \left ( \begin{smallmatrix}
  0 & A \\
   A & 0
   \end{smallmatrix} \right ),\rho\right)\le \gamma_+\left( A,\rho\right)\le 2\gamma\left(
   \left(\begin{smallmatrix}
  0 & A \\
   A & 0
   \end{smallmatrix} \right ),\rho\right).
   \end{equation}
\end{lemma}

\begin{proof}
 Fix $f,g:\n\to X$ and define $h:\{1,\ldots,2n\}\to X$ by
\begin{equation*}
h(i)\eqdef \left\{\begin{array}{ll}f(i)&\mathrm{if\ }i\in \n,\\
g(i-n)&\mathrm{if\ }i\in \{n+1,\ldots,2n\}.\end{array}\right.
\end{equation*}
Suppose that $A=(a_{ij})$ is an $n\times n$ symmetric stochastic matrix. Then
\begin{multline*}
\frac{1}{n^2}\sum_{i=1}^n\sum_{j=1}^n \rho(f(i),g(j))=\frac{1}{n^2}\sum_{i=1}^n\sum_{j=1}^n \rho(h(i),h(j+n))\le \frac{1}{2n^2}\sum_{i=1}^{2n}\sum_{j=1}^{2n} \rho(h(i),h(j))\\\le \frac{2\gamma\left(
   \left(\begin{smallmatrix}
  0 & A \\
   A & 0
   \end{smallmatrix} \right ),\rho\right)}{2n}\sum_{i=1}^{2n}\sum_{j=1}^{2n} \left(\begin{smallmatrix}
  0 & A \\
   A & 0
   \end{smallmatrix} \right )_{ij} \rho(h(i),h(j))=\frac{2\gamma\left(
   \left(\begin{smallmatrix}
  0 & A \\
   A & 0
   \end{smallmatrix} \right ),\rho\right)}{n}\sum_{i=1}^{n}\sum_{j=1}^{n} a_{ij} \rho(f(i),f(j)).
\end{multline*}
This proves the rightmost inequality in~\eqref{eq:2n}. Note that for this inequality the quasimetric inequality~\eqref{eq:quasimetric def} was not used, and therefore $\rho$ can be an arbitrarily kernel.

To prove the leftmost inequality in~\eqref{eq:2n} we argue as follows. Fix $h:\{1,\ldots,2n\}\to X$ and define $f,g:\n\to X$ by $f(i)=h(i)$ and $g(i)=h(i+n)$ for every $i\in \n$. Then
\begin{eqnarray}\label{eq:break 2 cover first}
\sum_{i=1}^n\sum_{j=1}^n \rho(h(i),h(j))&\le&\nonumber \frac{1}{n}\sum_{i=1}^n\sum_{j=1}^n\sum_{\ell=1}^n 2^\kappa\left(\rho(h(i),h(\ell+n))+\rho(h(j),h(\ell+n)\right))
\\&=&2^{\kappa+1}\sum_{i=1}^n\sum_{j=1}^n \rho(f(i),g(j)).
\end{eqnarray}
Similarly,
\begin{eqnarray}\label{eq:break 2 cover second}
\sum_{i=1}^n\sum_{j=1}^n \rho(h(i+n),h(j+n))&\le&\nonumber \frac{1}{n}\sum_{i=1}^n\sum_{j=1}^n\sum_{\ell=1}^n 2^\kappa\left(\rho(h(i+n),h(\ell))+\rho(h(j+n),h(\ell)\right))
\\&=&2^{\kappa+1}\sum_{i=1}^n\sum_{j=1}^n \rho(f(i),g(j)).
\end{eqnarray}
Hence,
\begin{eqnarray*}
&&\!\!\!\!\!\!\!\!\!\!\!\!\!\!\!\!\!\!\!\!\!\!\!\!\!\!\!\!\!\!\!\!\frac{1}{(2n)^2}\sum_{i=1}^{2n}\sum_{j=1}^{2n} \rho(h(i),h(j))\\&=&\frac{1}{(2n)^2} \sum_{i=1}^n\sum_{j=1}^n \rho(h(i),h(j))+\frac{1}{(2n)^2} \sum_{i=1}^n\sum_{j=1}^n \rho(h(i+n),h(j+n))\\&&  +\frac{1}{(2n)^2} \sum_{i=1}^n\sum_{j=1}^n \rho(h(i),h(j+n))+\frac{1}{(2n)^2} \sum_{i=1}^n\sum_{j=1}^n \rho(h(i+n),h(j))\\
&\stackrel{\eqref{eq:break 2 cover first}\wedge\eqref{eq:break 2 cover second}}{\le}&\frac{2^{\kappa+1}+1}{2n^2}\sum_{i=1}^n\sum_{j=1}^n \rho(f(i),g(j))\\
&\le& \frac{(2^{\kappa+1}+1)\gamma_+(A,\rho)}{2n}\sum_{i=1}^n\sum_{j=1}^n a_{ij}\rho(f(i),g(j))\\
&=&\frac{(2^{\kappa+1}+1)\gamma_+(A,\rho)}{2}\cdot \frac{1}{2n}\sum_{i=1}^{2n}\sum_{j=1}^{2n}\left(\begin{smallmatrix}
  0 & A \\
   A & 0
   \end{smallmatrix} \right )_{ij}\rho(h(i),h(j)),
\end{eqnarray*}
which is precisely the leftmost inequality in~\eqref{eq:2n}.
\end{proof}

\begin{lemma}\label{lem:2 cover cesaro} Fix $\kappa\in [0,\infty)$ and let $\rho:X\times X\to [0,\infty)$ be a $2^\kappa$-quasi-semimetric.  Then for every symmetric stochastic matrix $A$ we have
\begin{equation}\label{eq:commute tenzor}
\gamma\left( \left( \begin{smallmatrix}
  0 & \A_m(A) \\
   \A_m(A) & 0
   \end{smallmatrix} \right )
,\rho\right)\le \left(2^{\kappa+2}+1\right) \gamma\left( \A_m \left (\begin{smallmatrix}
  0 & A \\
   A & 0
   \end{smallmatrix} \right),\rho\right).
\end{equation}
\end{lemma}

\begin{proof} Suppose that $A=(a_{ij})$ is an $n\times n$ symmetric stochastic matrix. It suffices to show that for every $h:\{1,\ldots,2n\}\to X$ and every $m\in \N$ we have
\begin{equation}\label{eq:goal commute}
\sum_{i=1}^{2n}\sum_{j=1}^{2n} \A_m \left (\begin{smallmatrix}
  0 & A \\
   A & 0
   \end{smallmatrix} \right)_{ij}\rho(h(i),h(j))\le \left(2^{\kappa+2}+1\right)\sum_{i=1}^{2n}\sum_{j=1}^{2n} \left( \begin{smallmatrix}
  0 & \A_m(A) \\
   \A_m(A) & 0
   \end{smallmatrix} \right )_{ij}\rho(h(i),h(j)).
\end{equation}
For simplicity of notation write
$
B=(b_{ij})\eqdef \left(\begin{smallmatrix}
  0 & A \\
   A & 0
   \end{smallmatrix}\right).
$
Then
\begin{equation}\label{eq:decompose cesaro}
\A_m(B)=\frac{1}{m}I+\frac{1}{m}\sum_{s=1}^{\lfloor(m-1)/4\rfloor}B^{4s}+\frac{1}{m}\sum_{s=0}^{\lfloor(m-3)/4\rfloor}B^{2(2s+1)}+\frac{1}{m}\sum_{s=0}^{\lfloor (m-2)/2\rfloor} B^{2s+1}.
\end{equation}
Observe that
$$
t\in 2\N-1\implies \left(\begin{smallmatrix} 0 & A\\ A& 0 \end{smallmatrix}\right)^t= \left(\begin{smallmatrix} 0 & A^t\\ A^t& 0 \end{smallmatrix}\right).
$$
Hence,
\begin{equation}\label{eq:sum over odds}
\frac{1}{m}\sum_{s=0}^{\lfloor (m-2)/2\rfloor} B^{2s+1}= \left(\begin{smallmatrix} 0 & \frac{1}{m} \sum_{s=0}^{\lfloor (m-2)/2\rfloor}A^{2s+1}\\ \frac{1}{m} \sum_{s=0}^{\lfloor (m-2)/2\rfloor}A^{2s+1}& 0 \end{smallmatrix}\right).
\end{equation}
For every $s\in \N$, using the fact that $B^{2s-1}$ and $B^{2s+1}$ are symmetric and stochastic, we have
\begin{eqnarray}\label{eq:divisible by 4}
&&\nonumber\!\!\!\!\!\!\!\!\!\!\!\!\!\!\!\!\!\!\!\!\sum_{i=1}^{2n}\sum_{j=1}^{2n} \left(B^{4s}\right)_{ij}\rho(h(i),h(j))
%=\sum_{i=1}^{2n}\sum_{j=1}^{2n} \left(\sum_{\ell=1}^{2n}\left(B^{2s-1}\right)_{i\ell}\left(B^{2s+1}\right)_{\ell j}\right)\rho(h(i),h(j))
\\&\le& \nonumber \sum_{i=1}^{2n}\sum_{j=1}^{2n}\left(\sum_{\ell=1}^{2n}\left(B^{2s-1}\right)_{i\ell}\left(B^{2s+1}\right)_{\ell j}2^\kappa\left(\rho(h(i),h(\ell))+\rho(h(\ell),h(j))\right)\right)\\&=&
2^\kappa\sum_{a=1}^{2n}\sum_{b=1}^{2n} \left(\begin{smallmatrix} 0 & A^{2s-1}+A^{2s+1}\\ A^{2s-1}+A^{2s+1}& 0 \end{smallmatrix}\right)_{ab}\rho(h(a),h(b)).
\end{eqnarray}
Similarly, for every $s\in \N\cup\{0\}$,
\begin{equation}\label{eq:even and not divisible by 4}
\sum_{i=1}^{2n}\sum_{j=1}^{2n} \left(B^{2(2s+1)}\right)_{ij}\rho(h(i),h(j))\le
2^{\kappa+1}\sum_{a=1}^{2n}\sum_{b=1}^{2n} \left(\begin{smallmatrix} 0 & A^{2s+1}\\ A^{2s+1}& 0 \end{smallmatrix}\right)_{ab}\rho(h(a),h(b)).
\end{equation}
It follows from~\eqref{eq:decompose cesaro}, \eqref{eq:sum over odds},  \eqref{eq:divisible by 4} and~\eqref{eq:even and not divisible by 4} that
\begin{equation}\label{eq:forced odd}
\sum_{i=1}^{2n}\sum_{j=1}^{2n} \A_m(B)\rho(h(i),h(j))\le \sum_{i=1}^{2n}\sum_{j=1}^{2n} \left(\begin{smallmatrix} 0 & C\\ C& 0 \end{smallmatrix}\right)_{ij}\rho(h(i),h(j)),
\end{equation}
where
$$
C\eqdef \frac{1}{m}I+ \frac{2^\kappa}{m}\sum_{s=1}^{\lfloor(m-1)/4\rfloor}\left(A^{2s-1}+A^{2s+1}\right)+ \frac{2^{\kappa+1}}{m}\sum_{s=0}^{\lfloor(m-3)/4\rfloor}A^{2s+1}+\frac{1}{m} \sum_{s=0}^{\lfloor (m-2)/2\rfloor}A^{2s+1}.
$$
To deduce~\eqref{eq:goal commute} from~\eqref{eq:forced odd} it remains to observe that
\begin{equation*}
\forall\, i,j\in \{1,\ldots,n\},\quad C_{ij}\le \left(2^{\kappa+2}+1\right)\A_m(A)_{ij}.\tag*{\qedhere}
\end{equation*}
\end{proof}

The following two lemmas are intended to indicate that if one is only interested in the existence of super-expanders (rather than estimating the nonlinear spectral gap of a specific graph of interest) then the distinction between $\gamma(\cdot,\cdot)$ and $\gamma_+(\cdot,\cdot)$ is not very significant.

\begin{lemma}\label{lem:bipartite}
Fix $n,d\in \N$ and  let $G=(V,W,E)$ be a $d$-regular bipartite graph such that $|V|=|W|=n$. Then there exists a $2d$-regular graph $H=(V,F)$ for which every kernel $K:X\times X\to [0,\infty)$ satisfies
$
\gamma_+(H,K)\le 2\gamma(G,K).
$
\end{lemma}

\begin{proof}
Fix an arbitrary bijection $\sigma: V\to W$. The new edges $F$ on the vertex set $V$ are given by
\begin{equation*}\label{eq:def F bipartite}
\forall (u,v)\in V\times V,\quad F(u,v)\eqdef E(u,\sigma(v))+E(\sigma(u),v).
\end{equation*}
Thus $(V,F)$ is a $2d$-regular graph.

Given $f,g:V\to X$ define $\phi_1,\phi_2:V\cup W\to X$ by
\begin{equation*}\label{eq:def phis}
\phi_1(x)\eqdef \left\{\begin{array}{ll}f(x) &\mathrm{if\ } x\in V,\\
g\left(\sigma^{-1}(x)\right)& \mathrm{if\ } x\in W,\end{array}\right. \quad\mathrm{and}\quad \phi_2(x)\eqdef \left\{\begin{array}{ll}g(x) &\mathrm{if\ } x\in V,\\
f\left(\sigma^{-1}(x)\right)& \mathrm{if\ } x\in W.\end{array}\right.
\end{equation*}
Then,
\begin{align*}
&\frac{1}{n^2}\sum_{(u,v)\in V\times V}K(f(u),g(v))\\&\le \frac{1}{(2n)^2}\sum_{(x,y)\in (V\cup W)\times(V\cup W)}\left(K(\phi_1(x),\phi_1(y))+K(\phi_2(x),\phi_2(y))\right)\\
&\le \frac{\gamma(G,K)}{2nd}\sum_{(x,y)\in (V\times W)\cup(W\times V)} E(x,y) \left(K(\phi_1(x),\phi_1(y))+K(\phi_2(x),\phi_2(y))\right)\\&=\frac{\gamma(G,K)}{nd}\sum_{(u,v)\in V\times V}
\left(E(u,\sigma(v))+E(\sigma(u),v)\right)K(f(u),g(v))\\
&=\frac{2\gamma(G,K)}{n\cdot (2d)}\sum_{(u,v)\in F}K(f(u),g(v))\qedhere.
\end{align*}
\end{proof}

\begin{lemma}\label{lem:half size}
Fix $n,d\in \N$ and let $G=(V,E)$ be a $d$-regular graph with $|V|=2n$. Then there exists a $4d$-regular graph $G'=(V',E')$ with $|V'|=n$ such that for every $\kappa\in (0,\infty)$ and every $\rho:X\times X\to [0,\infty)$ which is a $2^\kappa$-quasi-semimetric we have $\gamma_+(G',\rho)\le 2^{\kappa+2}\gamma(G,\rho)$.
\end{lemma}

\begin{proof} Write $V=V'\cup V''$, where $V',V''\subseteq V$ are disjoint subsets of cardinality $n$, and fix an arbitrary bijection $\sigma: V'\to V''$. We first define a bipartite graph $H=(V',V'',F)$ by
 \begin{equation}\label{eq:sum of indicators}
 \forall (x,y)\in V'\times V'',\quad F(x,y)\eqdef E(x,y)+E\left(x,\sigma^{-1}(y)\right)+d\1_{\{y=\sigma(x)\}},
 \end{equation}
 where $F$ is extended to $V''\times V'$ by imposing symmetry.  This makes $H$ be a $2d$-regular bipartite graph. We shall now estimate $\gamma(H,\rho)$. For every $f:V\to X$ we have
\begin{multline}\label{eq:primes decomposition}
\frac{1}{(2n)^2}\sum_{(u,v)\in V\times V}\rho(f(u),f(v))\le \frac{\gamma(G,\rho)}{2nd}\left(\sum_{(u,v)\in (V'\times V'')\cup(V''\times V')}E(u,v)\rho(f(u),f(v))\right.\\\left.+\sum_{(u,v)\in V'\times V'}E(u,v)\rho(f(u),f(v))+\sum_{(u,v)\in V''\times V''}E(u,v)\rho(f(u),f(v))\right).
\end{multline}
Now, using the fact that $\rho$ is a $2^\kappa$-quasi-semimetric we have
\begin{multline}\label{eq:V' part}
\sum_{(u,v)\in V'\times V'}E(u,v)\rho(f(u),f(v))\le \sum_{(u,v)\in V'\times V'}2^\kappa E(u,v)\left(\rho(f(u),f(\sigma(v)))+\rho(f(\sigma(v)),f(v))\right)\\= 2^\kappa\sum_{(x,y)\in V'\times V''}E\left(x,\sigma^{-1}(y)\right)\rho(f(x),f(y))+2^\kappa d\sum_{z\in V'}\rho(f(\sigma(z)),f(z)).
\end{multline}
Similarly,
\begin{multline}\label{eq:V'' part}
\sum_{(u,v)\in V''\times V''}E(u,v)\rho(f(u),f(v))\\\le 2^\kappa\sum_{(x,y)\in V''\times V'}E(x,\sigma(y))\rho(f(x),f(y))+2^\kappa d\sum_{z\in V'}\rho(f(z),f(\sigma(z))).
\end{multline}
Recalling~\eqref{eq:sum of indicators}, we conclude from~\eqref{eq:primes decomposition},  \eqref{eq:V' part} and~\eqref{eq:V'' part} that
\begin{equation*}
\frac{1}{(2n)^2}\sum_{(u,v)\in (V'\cup V'')\times (V'\cup V'')}\rho(f(u),f(v))\le
\frac{2^{\kappa+1}\gamma(G,\rho)}{(2n)\cdot(2d)}\sum_{(x,y)\in F}\rho(f(x),f(y)).
\end{equation*}
Hence $\gamma(H,\rho)\le 2^{\kappa+1}\gamma(G,\rho)$. The desired assertion now follows from Lemma~\ref{lem:bipartite}.
\end{proof}

\subsection{Edge completion}\label{sec:edge completion}

In the ensuing arguments we will sometimes  add edges to a graph in order to ensure that it has certain desirable properties, but we will at the same time want to control the Poincar\'e constants of the resulting denser graph. The following very easy facts will be useful for this purpose.

\begin{lemma}\label{lem:graphs containing} Fix $n,d_1,d_2\in \N$.
Let $G_1=(V,E_1)$ and $G_2=(V,E_2)$ be two $n$-vertex graphs  on the same vertex set with $E_2\supseteq E_1$. Suppose that $G_1$ is $d_1$-regular and $G_2$ is $d_2$-regular. Then for every kernel $K:X\times X\to [0,\infty)$ we have
$$
\max\left\{\frac{\gamma(G_2,K)}{\gamma(G_1,K)},\frac{\gamma_+(G_2,K)}{\gamma_+(G_1,K)}\right\}\le \frac{d_2}{d_1}.
$$
\end{lemma}

\begin{proof}
One just has to note that for every $f,g:V\to X$ we have
\begin{equation*}
\frac{1}{nd_2}\sum_{(x,y)\in E_2}K(f(x),g(y))\ge \frac{1}{nd_2}\sum_{(x,y)\in E_1}K(f(x),g(y))=\frac{d_1}{d_2}\cdot\frac{1}{nd_1}\sum_{(x,y)\in E_1}K(f(x),g(y)).\tag*{\qedhere}
\end{equation*}
\end{proof}

\begin{definition}[Edge completion]\label{def:edge completion} Fix two integers $D\ge d\ge 2$. Let $G=(V,E)$ be a $d$-regular graph. The $D$-edge completion of $G$, denoted $\c_D(G)$, is defined as a graph on the same vertex set $V$, with edges $E(\c_D(G))\supseteq E$ defined as follows. Write $D=m d+r$, where $m\in \N$ and $r\in \{0,\ldots,d-1\}$. Then $E(\c_D(G))$ is obtained from $E$ by duplicating each edge $m$ times and adding $r$ self loops to each vertex in $V$, i.e.,
 \begin{equation}\label{eq:def edge completion}
\forall(x,y)\in V\times V,\quad E(\c_D(G))(x,y)\eqdef mE(x,y)+r\1_{\{x=y\}}.
 \end{equation}
 This definition makes $\c_D(G)$ be a $D$-regular graph.
\end{definition}

\begin{lemma}\label{lem:edge completion} Fix two integers $D\ge d\ge 2$ and let $G=(V,E)$ be a $d$-regular graph. Then for every kernel $K:X\times X\to [0,\infty)$ we have
\begin{equation}\label{eq:edge completition}
\max\left\{\frac{\gamma(\c_D(G),K)}{\gamma(G,K)},\frac{\gamma_+(\c_D(G),K)}{\gamma_+(G,K)}\right\}\le 2.
\end{equation}
\end{lemma}
\begin{proof}
Write $|V|=n$ and $D=m d+r$, where $m\in \N$ and $r\in \{0,\ldots,d-1\}$. For every $f,g:V\to X$ we have
\begin{multline*}
\frac{1}{nD}\sum_{(x,y)\in E\left(\c_D(G)\right)} K(f(x),g(y))\stackrel{\eqref{eq:def edge completion}}{=}\frac{1}{nd}\sum_{(x,y)\in V\times V}\frac{mdE(x,y)+rd\1_{\{x=y\}}}{md+r}K(f(x),g(y))\\\ge \frac{1}{nd}\sum_{(x,y)\in V\times V}\frac{m}{m+1}E(x,y)K(f(x),g(y))\ge \frac12\cdot \frac{1}{nd}\sum_{(x,y)\in E}K(f(x),g(y)).\tag*{\qedhere}
\end{multline*}
\end{proof}

\section{Metric Markov cotype implies nonlinear spectral calculus}\label{sec:cotype to calculus}

Our goal here is to prove Theorem~\ref{thm:cotype implies calculus intro}. We start with an analogous statement that treats the parameter $\gamma(\cdot,\cdot)$ rather than $\gamma_+(\cdot,\cdot)$.

\begin{lemma}[Metric Markov cotype implies the decay of $\gamma$]\label{lem:decay of gamma in section}
Fix $C,\e\in (0,\infty)$, $q\in [1,\infty)$, $m,n\in \N$ and an $n\times n$ symmetric stochastic matrix $A=(a_{ij})$. Suppose that $(X,d_X)$ is a metric space such  that for every $x_1,\ldots,x_n\in X$  there exist $y_1,\ldots,y_n\in X$ satisfying
\begin{equation}\label{eq:exponent q cotype}
\sum_{i=1}^n d_X(x_i,y_i)^q+m^\e\sum_{i=1}^n\sum_{j=1}^na_{ij} d_X(y_i,y_j)^q\le C^q \sum_{i=1}^n\sum_{j=1}^n \A_m(A)_{ij}d_X(x_i,x_j)^q.
\end{equation}
Then
\begin{equation}\label{eq:gamma decay in section}
\gamma\left(\A_m(A),d_X^q\right)\le (3C)^q \max\left\{1,\frac{\gamma\left(A,d_X^q\right)}{m^\e}\right\}.
\end{equation}
\end{lemma}

\begin{proof}
Write $B=(b_{ij})=\A_m(A)$. If $\gamma(B,d_X^q)\le (3C)^q$ then~\eqref{eq:gamma decay in section} holds true, so we may assume from now on that $\gamma(B,d_X^q)> (3C)^q$. Fix
\begin{equation}\label{gamma range assumption}
(3C)^q<\gamma<\gamma(B,d_X^q).
\end{equation}
 By the definition of $\gamma(B,d_X^q)$ there exist $x_1,\ldots,x_n\in X$ such that
\begin{equation}\label{eq:choose xi}
\frac{1}{n^2}\sum_{i=1}^n\sum_{j=1}^nd_X(x_i,x_j)^q>\frac{\gamma}{n}\sum_{i=1}^n\sum_{j=1}^nb_{ij}d_X(x_i,x_j)^q.
\end{equation}
Let $y_1,\ldots,y_n\in X$ satisfy~\eqref{eq:exponent q cotype}. By the triangle inequality, for every $i,j\in \n$ we have
\begin{equation}\label{eq:triangle xixjyiyj}
d_X(x_i,x_j)^q\le 3^{q-1}\left(d_X(x_i,y_i)^q+d_X(y_i,y_j)^q+d_X(y_j,x_j)^q\right).
\end{equation}
By averaging~\eqref{eq:triangle xixjyiyj} we get the following estimate.
\begin{eqnarray}\label{eq:before gamma for yi}
\nonumber\frac{1}{n^2}\sum_{i=1}^n\sum_{j=1}^nd_X(y_i,y_j)^q&\ge& \frac{1}{3^{q-1}n^2}\sum_{i=1}^n\sum_{j=1}^nd_X(x_i,x_j)^q-\frac{2}{n}\sum_{i=1}^n d_X(x_i,y_i)^q\\ \nonumber
&\stackrel{\eqref{eq:choose xi}}{>}& \frac{\gamma}{3^{q-1}n}\sum_{i=1}^n\sum_{j=1}^nb_{ij}d_X(x_i,x_j)^q-\frac{2}{n}\sum_{i=1}^n d_X(x_i,y_i)^q\\ \nonumber
&\stackrel{\eqref{eq:exponent q cotype}}{\ge}& \frac{3\gamma m^\e}{(3C)^qn}\sum_{i=1}^n\sum_{j=1}^na_{ij} d_X(y_i,y_j)^q+\left(\frac{3\gamma }{(3C)^qn}-\frac{2}{n}\right)\sum_{i=1}^n d_X(x_i,y_i)^q\\
&\stackrel{\eqref{gamma range assumption}}{\ge}& \frac{3\gamma m^\e}{(3C)^qn}\sum_{i=1}^n\sum_{j=1}^na_{ij} d_X(y_i,y_j)^q.
\end{eqnarray}
At the same time, by the definition of $\gamma(A,d_X^q)$ we have
\begin{equation}\label{eq:apply gamma for yi}
\frac{1}{n^2}\sum_{i=1}^n\sum_{j=1}^nd_X(y_i,y_j)^q\le \frac{\gamma(A,d_X^q)}{n} \sum_{i=1}^n\sum_{j=1}^na_{ij} d_X(y_i,y_j)^q.
\end{equation}
By contrasting~\eqref{eq:apply gamma for yi} with~\eqref{eq:before gamma for yi} and letting $\gamma\nearrow \gamma(B,d_X^p)$ we deduce that
\begin{equation*}
\gamma\left(\A_m(A),d_X^q\right)=\gamma(B,d_X^q)\le 3^{q-1}C^q\frac{\gamma(A,d_X^q)}{m^\e}.\qedhere
\end{equation*}
\end{proof}

The special case $q=2$ of the following theorem implies Theorem~\ref{thm:cotype implies calculus intro}.
\begin{theorem}[Metric Markov cotype implies the decay of $\gamma_+$]
Fix $C,\e\in (0,\infty)$, $q\in [1,\infty)$, $m,n\in \N$ and an $n\times n$ symmetric stochastic matrix $A=(a_{ij})$. Suppose that $(X,d_X)$ is a metric space such  that for every $x_1,\ldots,x_{2n}\in X$  there exist $y_1,\ldots,y_{2n}\in X$ satisfying
\begin{equation}\label{eq:exponent q cotype 2 cover}
\sum_{i=1}^{2n} d_X(x_i,y_i)^q+m^\e\sum_{i=1}^{2n}\sum_{j=1}^{2n}\left( \begin{smallmatrix}
  0 & A \\
   A & 0
   \end{smallmatrix} \right )_{ij} d_X(y_i,y_j)^q\le C^q \sum_{i=1}^{2n}\sum_{j=1}^{2n} \A_m\left( \begin{smallmatrix}
  0 & A \\
   A & 0
   \end{smallmatrix} \right )_{ij}d_X(x_i,x_j)^q.
\end{equation}
Then
\begin{equation}\label{eq:gamma+decayq}
\gamma_+\left(\A_m(A),d_X^q\right)\le (45C)^q\max\left\{1,\frac{\gamma_+\left(A,d_X^q\right)}{m^\e}\right\}.
\end{equation}
\end{theorem}

\begin{proof}
By Lemma~\ref{lem:pass to 2-cover} and Lemma~\ref{lem:2 cover cesaro} we have
\begin{equation}\label{eq:use 2 cover lemmas}
\gamma_+\left(\A_m(A),d_X^q\right)\stackrel{\eqref{eq:2n}}{\le} 2\gamma\left(\left( \begin{smallmatrix}
  0 & \A_m(A) \\
   \A_m(A) & 0
   \end{smallmatrix} \right ),d_X^q\right)\stackrel{\eqref{eq:commute tenzor}}{\le} 2\left(2^{q+1}+1\right)\gamma\left(\A_m\left( \begin{smallmatrix}
  0 & A \\
   A & 0
   \end{smallmatrix} \right ),d_X^q\right).
\end{equation}
At the same time, an application of Lemma~\ref{lem:decay of gamma in section} and Lemma~\ref{lem:pass to 2-cover} yields the estimate
\begin{multline}\label{eq:use gamma lemma}
\gamma\left(\A_m\left( \begin{smallmatrix}
  0 & A \\
   A & 0
   \end{smallmatrix} \right ),d_X^q\right)\le (3C)^q\max\left\{1,\frac{\gamma\left(\left( \begin{smallmatrix}
  0 & A \\
   A & 0
   \end{smallmatrix} \right ),d_X^q\right)}{m^\e}\right\}\\\stackrel{\eqref{eq:2n}}{\le} (3C)^q\max\left\{1,\frac{2^q+1}{2}\cdot\frac{\gamma_+\left(A,d_X^q\right)}{m^\e}\right\}.
\end{multline}
The desired estimate~\eqref{eq:gamma+decayq} is a consequence of~\eqref{eq:use 2 cover lemmas} and~\eqref{eq:use gamma lemma}.
\end{proof}

\section{An iterative construction of super-expanders}\label{sec:iterative}

Our goal here is to prove the existence of super-expanders as stated in Theorem~\ref{thm:existence intro}, assuming the validity of Lemma~\ref{lem:base in section}, Corollary~\ref{coro:decay for super-reflexive intro} and Theorem~\ref{thm:products intro}. These ingredients will then be proved in the subsequent sections.

In order to elucidate the ensuing construction, we phrase it in the setting of abstract kernels, though readers are encouraged to keep in mind that it will be used in the geometrically meaningful case of super-reflexive Banach spaces.

\begin{lemma}[Initial zigzag iteration]\label{lem:first kernel use} Fix $d,m,t\in \N$ satisfying
\begin{equation}\label{eq:t_0 assumption2}
td^{2(t-1)}\le m,
\end{equation}
and fix a $d$-regular graph $G_0=(V,E)$ with $|V|=m$. Then for every $j\in \N$ there exists a regular graph $F_j^t=(V_j^t,E_j^t)$ of degree $d^2$ and with $|V_j^t|=m^{j}$ such that the following holds true. If $K:X\times X\to [0,\infty)$ is a kernel such that $\gamma_+(G_0,K)<\infty$ then also $\gamma_+(F_j^t,K)<\infty$ for all $j\in \N$. Moreover, suppose that $C,\gamma\in [1,\infty)$ and $\e\in (0,1)$ satisfy
\begin{equation}\label{eq:t_0 assumption1}
t\ge \left(2C\gamma^2\right)^{1/\e},
\end{equation}
and that the kernel $K$ is such that every finite regular graph $G$ satisfies  the nonlinear spectral calculus inequality
\begin{equation}\label{eq:K calculus assumption in lemma}
 \gamma_+(\A_t(G),K)\le C\max\left\{1,\frac{\gamma_+(G,K)}{t^\e}\right\}.
\end{equation}
Suppose furthermore that
\begin{equation}\label{eq:G assumption}
\gamma_+(G_0,K)\le \gamma.
\end{equation}
Then
$$
\sup_{j\in \N} \gamma_+(F_j^t,K)\le 2C\gamma^2.
$$
%Suppose also that for every $\d\in (0,1]$ there exists a sequence of regular graphs $\{G_n(\d)\}_{n=1}^\infty$ with $\{V(G_n(\d))\}_{n=1}^\infty$ strictly increasing, and %if the degree of $G_n(\d)$ is denoted $\Delta_n(\d)$ then
%$$
%\lim_{n\to \infty}\frac{\log \Delta_n(\d)}{\log |V(G_n(\d))|}=0.
%$$
\end{lemma}

\begin{proof} Set $F_1^t\eqdef \c_{d^2}(G_0)$, where we recall the definition of the edge completion operation as discussed in Section~\ref{sec:edge completion}. Thus $F_1^t$ has $m$ vertices and degree $d^2$. Assume inductively that we defined $F_j^t$ to be a regular graph with $m^j$ vertices and degree $d^2$. Then the Ces\`aro average $\A_t(F_j^t)$ has $m^j$ vertices and degree $td^{2(t-1)}$ (recall the discussion preceding~\eqref{eq:def cesaro notation}). It follows from~\eqref{eq:t_0 assumption2} that the degree of $\A_t(F_j^t)$ is at most $m$, so we can form the edge completion $\c_m(\A_t(F_j^t))$, which has degree $m$, and we can therefore form the zigzag product
\begin{equation}\label{eq:def Fj+1t}
F_{j+1}^t\eqdef \left(\c_m(\A_t(F_j^t))\right)\oz G_0.
\end{equation}
Thus $F^t_{j+1}$ has $m^{j+1}$ vertices and degree $d^2$, completing the inductive construction. Using Theorem~\ref{thm:sub} and Lemma~\ref{lem:edge completion}, it follows inductively that if $K:X\times X\to [0,\infty)$ is a kernel such that $\gamma_+(G_0,K)<\infty$ then also $\gamma_+(F_j^t,K)<\infty$ for all $j\in \N$.

Assuming the validity of~\eqref{eq:G assumption}, by Lemma~\ref{lem:edge completion} we have
$$
\gamma_+(F^t_1,K)=\gamma_+\left(\c_{d^2}(G_0),K\right)\stackrel{\eqref{eq:edge completition}}{\le} 2\gamma_{+}(G_0,K)\stackrel{\eqref{eq:G assumption}}{\le }2\gamma.
$$
We claim that for every $j\in \N$,
\begin{equation}\label{eq:inductive assumption Fj}
\gamma_+(F_j^t,K)\le 2C\gamma^2.
\end{equation}
 Assuming the validity of~\eqref{eq:inductive assumption Fj} for some $j\in \N$, by Theorem~\ref{thm:sub} we have
\begin{multline*}
\gamma_+(F_{j+1}^t,K)\stackrel{\eqref{eq:sub}\wedge \eqref{eq:def Fj+1t}}{\le} \gamma_+\left(\c_m(\A_t(F_j^t))\right)\gamma_+(G_0,K)^2\stackrel{\eqref{eq:edge completition}\wedge\eqref{eq:G assumption}}{\le} 2\gamma_+\left(\A_t(F_j^t),K\right)\gamma^2\\
\stackrel{\eqref{eq:K calculus assumption in lemma}}{\le} 2C\gamma^2\max\left\{1,\frac{\gamma_+(F_j^t,K)}{t^\e}\right\}\stackrel{\eqref{eq:inductive assumption Fj}}{\le}
2C\gamma^2\max\left\{1,\frac{2C\gamma^2}{t^\e}\right\}\stackrel{\eqref{eq:t_0 assumption1}}{\le} 2C\gamma^2.\tag*{\qedhere}
\end{multline*}
\end{proof}

\begin{corollary}[Intermediate construction for super-reflexive Banach spaces]\label{cor:apply gen zigzag to base graphs}
For every $k\in \N$ there exist regular graphs $\{F_j(k)\}_{j=1}^\infty$ and integers $\{d_k\}_{k=1}^\infty, \{n_j(k)\}_{j,k\in \N}\subseteq \N$, where $\{n_j(k)\}_{j=1}^\infty$ is a strictly increasing sequence, such that $F_j(k)$ has degree $d_k$ and $n_j(k)$ vertices, and the following condition holds true. For every super-reflexive Banach space $(X,\|\cdot\|_X)$,
$$
\forall\, j,k\in \N,\quad \gamma_+\left(F_j(k),\|\cdot\|_X^2\right)<\infty,
$$
and moreover there exists $k(X)\in \N$ such that
$$
\sup_{\substack{j,k\in \N\\ k\ge k(X)}} \gamma_+\left(F_j(k),\|\cdot\|_X^2\right)\le k(X).
$$
\end{corollary}

\begin{proof}
We shall use here the notation of Lemma~\ref{lem:base in section}. For every $k\in \N$ choose an integer $n(k)\ge n_0(1/k)$ (recall that $n_0(1/k)$ was introduced in Lemma~\ref{lem:base in section}) such that
\begin{equation}\label{eq:n(k)}
ke^{2(k-1)\left(\log m_{n(k)}\right)^{1-\frac{1}{k}}}\le m_{n(k)}.
\end{equation}
By~\eqref{eq:dndelta bound in section}, it follows from~\eqref{eq:n(k)} that $d_{n(k)}(1/k)$, i.e., the degree of the graph $H_{n(k)}(1/k)$, satisfies
$$
kd_{n(k)}^{2(k-1)}\le m_{n(k)}=|V(H_{n(k)}(1/k))|,
$$
where here, and in what follows, $V(G)$ denotes the set of vertices of a graph $G$.  We can therefore apply Lemma~\ref{lem:first kernel use} with the parameters $t=k$, $d=d_{n(k)}(1/k)$, $m=m_{n(k)}$ and $G_0=H_{n(k)}(1/k)$. Letting $\{F_j(k)\}_{j=1}^\infty$ denote the resulting sequence of graphs, we define $$
d_k\eqdef \left(d_{n(k)}(1/k)\right)^2\quad \mathrm{and}\quad  n_j(k)\eqdef \left(m_{n(k)}\right)^j.$$

If $(X,\|\cdot\|_X)$ is a super-reflexive Banach space then it is in particular $K$-convex (see~\cite{Pisier-K-convex}). Recalling the parameter $\d_0(X)$ of Lemma~\ref{lem:base in section}, we have
\begin{equation*}\label{eq:first k condition}
k\ge \frac{1}{\d_0(X)}\implies \gamma_+\left(H_{n(k)}(1/k),\|\cdot\|_X^2\right)\le 9^3.
\end{equation*}
It also follows from Corollary~\ref{coro:decay for super-reflexive intro} that there exists $C(X)\in [1,\infty)$ and $\e(X)\in (0,1)$ for which every finite regular graph $G$ satisfies
\begin{equation}\label{eq:graph version of decay in section}
\forall\, t\in \N,\quad \gamma_+\left(\A_t(G),\|\cdot\|_X^2\right)\le C(X)\max\left\{1,\frac{\gamma_+\left(G,\|\cdot\|_X^2\right)}{t^{\e(X)}}\right\}.
\end{equation}
We may therefore apply Lemma~\ref{lem:first kernel use} with $C=C(X)$, $\e=\e(X)$ and $\gamma=9^3$ to deduce that if we define
$$
k(X)\eqdef \left\lceil \max\left\{\frac{1}{\d_0(X)},\left(2C(X)\cdot9^3\right)^{1/\e(X)}, 2C(X)\cdot9^6\right\} \right\rceil,
$$
then for every $j\in \N$,
\begin{equation*}
k\ge k(X)\implies \sup_{j\in \N}\gamma_+\left(F_j(k),\|\cdot\|_X^2\right)\le 2C(X)\cdot9^6\le k(X).\tag*{\qedhere}
\end{equation*}
\end{proof}

Corollary~\ref{cor:apply gen zigzag to base graphs} provides a sequence of expanders with respect to a {\em fixed} super-reflexive Banach space $(X,\|\cdot\|_X)$, but since the sequence of degrees $\{d_k\}_{k=1}^\infty$ may be unbounded (this is indeed the case in our construction), we still do not have one sequence of bounded degree regular graphs that are expanders with respect to {\em every} super-reflexive Banach space. This is achieved in the following crucial lemma.

\begin{lemma}[Main zigzag iteration]\label{lem:zigzag main iteration}
Let $\{d_k\}_{k=1}^\infty$ be a sequence of integers and for each $k\in \N$ let $\{n_j(k)\}_{j=1}^\infty$ be a strictly increasing sequence of integers. For every $j,k\in \N$ let $F_j(k)$ be a regular graph of degree $d_k$ with $n_j(k)$ vertices. Suppose that $\K$ is a family of kernels such that
\begin{equation}\label{eq:K finiteness assumption}
\forall\, K\in \K,\ \forall\, j,k\in \N,\quad \gamma_+(F_j(k),K)<\infty.
\end{equation}
Suppose also that the following two conditions hold true.
\begin{itemize}
\item For every $K\in \K$ there exists $k_1(K)\in \N$ such that
\begin{equation}\label{eq:k1}
\sup_{\substack{j,k\in \N\\k\ge k_1(K)}}\gamma_+(F_j(k),K)\le k_1(K).
\end{equation}
\item For every $K\in \K$ there exists $k_2(K)\in \N$ such that every regular graph $G$ satisfies the following spectral calculus inequality.
\begin{equation}\label{eq:k2}
\forall\, t\in \N,\quad \gamma_+\left(\A_t(G),K\right)\le k_2(K)\max\left\{1,\frac{\gamma_+(G,K)}{t^{1/k_2(K)}}\right\}.
\end{equation}
\end{itemize}
Then there exists $d\in \N$ and a sequence of $d$-regular graphs $\{H_i\}_{i=1}^\infty$ with
$$
\lim_{i\to \infty} |V(H_i)|=\infty
$$
and
\begin{equation}\label{eq:good for all K}
\forall\, K\in \K,\quad \sup_{j\in \N}\gamma_+(H_j,K)<\infty.
\end{equation}
\end{lemma}

\begin{proof}
In what follows, for every $k\in \N$ it will be convenient to introduce the notation
\begin{equation}\label{eq:def Mk}
M_k\eqdef \left(2k^3\right)^k.
\end{equation}
With this, define
\begin{equation}\label{eq:def j(k)}
j(k)\eqdef \min\left\{j\in \N:\ n_j(k)>2d_1^2+M_{k+1}d_{k+1}^{2(M_{k+1}-1)}\right\},
\end{equation}
and
\begin{equation}\label{eq:defW_k}
W_k\eqdef F_{j(k)}(k).
\end{equation}

We will next define for every $k\in \N$ an integer $\ell(k)\in \N\cup\{0\}$ and a sequence of regular graphs $W_k^0,W_k^1,\ldots,W_k^{\ell(k)}$, along with an auxiliary integer sequence $\{h_i(k)\}_{i=0}^{\ell(k)}\subseteq \N$. Set
\begin{equation}\label{eq:start W h}
W_k^0\eqdef W_k\quad\mathrm{and}\quad h_0(k)\eqdef k.
 \end{equation}
 Define $\ell(1)=0$.
 %Note that the degree of $W_1^0=F_{j(1)}(1)$ is $d_1$, which by~\eqref{eq:def j(k)} is less than $n_{j(1)}(1)$. We may therefore form the edge completion %$\c_{n_{j(1)}(1)}\left(W_1^0\right)$. Since $W_1$ has $n_{j(1)}(1)$ vertices, which is the same as the degree of $\c_{n_{j(1)}(1)}\left(W_1^0\right)$, we can define
%\begin{equation*}
%W_1^{\ell(1)}=W_1^1\eqdef \A_{2}\left(\c_{n_{j(1)}(1)}\left(W_1^0\right)\oz W_1\right).
%\end{equation*}
%Thus $W_1^1$ has degree $2d_1^{2}$.
 For every integer $k>1$ set
\begin{equation}\label{eq:def h1}
h_1(k)\eqdef \min\left\{h\in \N:\ n_{j(h)}(h)\ge d_{h_0(k)}\right\}.
\end{equation}
Observe that necessarily $h_1(k)<h_0(k)=k$. Indeed, if $h_1(k)\ge k$ then
$$
d_k\stackrel{\eqref{eq:def h1}}{>} n_{j(k-1)}(k-1)\stackrel{\eqref{eq:def j(k)}}{>} M_kd_{k}^{2\left(M_k-1\right)}\stackrel{\eqref{eq:def Mk}}{\ge} d_k,
$$
a contradiction. By the definition of $h_1(k)$ we know that $n_{j(h_1(k))}(h_1(k))\ge d_{h_0(k)}$, so we may form the edge completion $\c_{n_{j(h_1(k))}(h_1(k))}\left(W_k^0\right)$. Since the number of vertices of $W_{h_1(k)}$ is $n_{j(h_1(k))}(h_1(k))$, which is the same as the degree of $\c_{n_{j(h_1(k))}(h_1(k))}\left(W_k^0\right)$, we can define
$$
W_k^1\eqdef \A_{M_{h_1(k)}}\left(\c_{n_{j(h_1(k))}(h_1(k))}\left(W_k^0\right)\oz W_{h_1(k)}\right).
$$
The degree of $W_k^1$ equals $$M_{h_1(k)}d_{h_1(k)}^{2\left(M_{h_1(k)}-1\right)}.$$
%$$
%\Delta_k^1\eqdef (2h_1(k)+2)^{3h_1(k)+3}d_{h_1(k)+1}^{2\left((2h_1(k)+2)^{3h_1(k)+3}-1\right)}.
%$$

Assume inductively that $k,i>1$ and we have already defined the graph $W_k^{i-1}$ and the integer $h_{i-1}(k)$, such that the degree of $W_k^{i-1}$ equals \begin{equation}\label{eq:degree at i-1}
M_{h_{i-1}(k)}d_{h_{i-1}(k)}^{2\left(M_{h_{i-1}(k)}-1\right)}.
\end{equation}
%\begin{equation}\label{eq:def Deltaki}
%\Delta_k^{i-1}\eqdef (2h_{i-1}(k)+2)^{3h_{i-1}(k)+3}d_{h_{i-1}(k)+1}^{2\left((2h_{i-1}(k)+2)^{3h_{i-1}(k)+3}-1\right)}.
%\end{equation}

If $h_{i-1}(k)=1$ then conclude the construction, setting $\ell(k)=i-1$.
%, and define
%\begin{equation}\label{eq:last graph}
%W_k^{\ell(k)}=W_k^i\eqdef \A_{M_1}\left(\c_{n_{j(1)}(1)}\left(W_k^{i-1}\right)\oz W_1\right).
%\end{equation}
%Note that all the operations in~\eqref{eq:last graph} are well defined, since the assumption $h_{i-1}(k)=1$ combined with the fact that the degree of $W_k^{i-1}$, as %given in~\eqref{eq:degree at i-1}, equals $2d_1^2$, implies via~\eqref{eq:def j(k)} that the degree of $W_k^{i-1}$ is greater than $n_{j(1)}(1)$, so that the edge %completion $\c_{n_{j(1)}(1)}\left(W_k^{i-1}\right)$ is well defined. The zigzag product in~\eqref{eq:last graph} is now well defined since the number of vertices of %$W_1$ equals $n_{j(1)}(1)$. Observe that the degree of $W_k^{\ell(k)}$ equals $2d_1^2$.
If $h_{i-1}(k)>1$ then we proceed by defining
\begin{equation}\label{eq:def hi}
h_i(k)\eqdef \min\left\{h\in \N:\ n_{j(h)}(h)\ge M_{h_{i-1}(k)}d_{h_{i-1}(k)}^{2\left(M_{h_{i-1}(k)}-1\right)}\right\}.
\end{equation}
Observe that
\begin{equation}\label{eq:monotonicity h}
h_i(k)<h_{i-1}(k).
\end{equation}
Indeed, if $h_i(k)\ge h_{i-1}(k)$ then
$$
M_{h_{i-1}(k)}d_{h_{i-1}(k)}^{2\left(M_{h_{i-1}(k)}-1\right)}\stackrel{\eqref{eq:def hi}}{>} n_{j(h_{i-1}(k)-1)}(h_{i-1}(k)-1)\stackrel{\eqref{eq:def j(k)}}{>} 2d_1^2+M_{h_{i-1}(k)}d_{h_{i-1}(k)}^{2(M_{h_{i-1}(k)}-1)},
$$
a contradiction. Since the degree of $W_k^{i-1}$ is given in~\eqref{eq:degree at i-1}, which by~\eqref{eq:def hi} is at most  $n_{j(h_i(k))}(h_i(k))$, we may form the edge completion $\c_{n_{j(h_i(k))}(h_i(k))}\left(W_k^{i-1}\right)$. The degree of the resulting graph is $n_{j(h_i(k))}(h_i(k))$, which, by~\eqref{eq:defW_k}, equals the number of vertices of $W_{h_i(k)}$. We can therefore define
\begin{equation}\label{eq:main W recursion}
W_k^i\eqdef \A_{M_{h_i(k)}}\left(\c_{n_{j(h_i(k))}(h_i(k))}\left(W_k^{i-1}\right)\oz W_{h_i(k)}\right).
\end{equation}
The degree of $W_k^i$ equals $M_{h_i(k)}d_{h_i(k)}^{2\left(M_{h_i(k)}-1\right)}$, thus completing the inductive step.

Due to~\eqref{eq:monotonicity h} the above procedure must eventually terminate, and by definition $h_{\ell(k)}(k)=1$. Since $h_0(k)=k$, it follows that
\begin{equation}\label{eq:ellk trivial bound}
\forall\,k\in \N,\quad \ell(k)\le k.
\end{equation}
We define
$$
H_k\eqdef W_k^{\ell(k)}.
$$
The degree of $H_k$ equals $d\eqdef 2d_1^2$ for all $k\in \N$. Also, by construction we have
$$
|V(H_k)|=\left|V\left(W_k^{\ell(k)}\right)\right|\ge \left|V\left(W_k^{\ell(k)-1}\right)\right|\ge \ldots\ge  \left|V\left(W_k^0\right)\right|\stackrel{\eqref{eq:start W h}\wedge\eqref{eq:defW_k}}{=} n_{j(k)}(k)\stackrel{\eqref{eq:def j(k)}}{\ge} M_{k+1}.
$$
Thus $\lim_{k\to \infty} |V(H_k)|=\infty$. It remains to prove that for every kernel $K\in \K$ we have
\begin{equation}\label{eq:goal Hk good}
\sup_{k\in \N}\gamma_+(H_k,K)<\infty.
\end{equation}

To prove~\eqref{eq:goal Hk good} we start with the following crucial estimate, which holds for every $k\in \N$ and $i\in \{1,\ldots,\ell(k)\}$.
\begin{multline}\label{eq:for quote upper W}
\gamma_+\left(W_k^i,K\right)\stackrel{\eqref{eq:k2}\wedge\eqref{eq:main W recursion}}{\le} k_2(K)\max\left\{1,\frac{\gamma_+\left(\c_{n_{j(h_i(k))}(h_i(k))}\left(W_k^{i-1}\right)\oz W_{h_i(k)},K\right)}{M_{h_i(k)}^{1/k_2(K)}}\right\}\\
\stackrel{\eqref{eq:sub}\wedge\eqref{eq:edge completition}\wedge\eqref{eq:defW_k}}{\le} k_2(K)\max\left\{1,\frac{2\gamma_+\left(W_k^{i-1},K\right)\gamma_+\left(F_{j(h_i(k))}(h_i(k)),K\right)^2}{M_{h_i(k)}^{1/k_2(K)}}\right\}.
\end{multline}
In particular, it follows from~\eqref{eq:for quote upper W} that the following crude estimate holds true.
\begin{equation}\label{eq:crude use of calculus}
\gamma_+\left(W_k^i,K\right)\le
2k_2(K)\gamma_+\left(W_k^{i-1},K\right)\gamma_+\left(F_{j(h_i(k))}(h_i(k)),K\right)^2.
\end{equation}
A recursive application of~\eqref{eq:crude use of calculus} yields the estimate
\begin{equation*}
\gamma_+(H_k,K)=\gamma_+\left(W_k^{\ell(k)},K\right)\le(2k_2(K))^{\ell(k)}\gamma_+\left(W_k^0,K\right)\prod_{i=1}^{\ell(k)} \gamma_+\left(F_{j(h_i(k))}(h_i(k)),K\right)^2.
\end{equation*}
Due to the finiteness assumption~\eqref{eq:K finiteness assumption}, it follows that
\begin{equation}\label{eq:Hk finite}
\forall\, k\in \N,\quad \gamma_+(H_k,K)<\infty.
\end{equation}

In order to prove~\eqref{eq:goal Hk good} we will need to apply~\eqref{eq:for quote upper W} more carefully. To this end set
\begin{equation}\label{eq:def k3}
k_3(K)\eqdef \max\left\{k_1(K),k_2(K)\right\},
\end{equation}
and fix $k>k_3(K)$. We will now prove by induction on $i\in \{0,\ldots,\ell(k)\}$ that
\begin{equation}\label{eq:induction, ig hi big}
h_i(k)>k_3(K)\implies \gamma_+\left(W_k^i,K\right)\le k_3(K).
\end{equation}
If $i=0$ then $h_0(k)=k>k_3(K)\ge k_1(K)$, so by our assumption~\eqref{eq:k1},
$$
\gamma_+\left(W_k^0,K\right)\stackrel{\eqref{eq:start W h}\wedge\eqref{eq:defW_k}}{=}\gamma_+\left(F_{j(k)}(k),K\right)\stackrel{\eqref{eq:k1}}{\le} k_1(K)\le k_3(K).
$$
Assume inductively that $i\in \{1,\ldots,\ell(k)\}$ satisfies
\begin{equation}\label{eq:hi big inductive assumption}
h_i(k)>k_3(K).
\end{equation}
 By~\eqref{eq:monotonicity h} and the inductive hypothesis we therefore have
 \begin{equation}\label{eq:use inductive k3}
 \gamma_+\left(W_k^{i-1},K\right)\le k_3(K).
 \end{equation}
Hence,
\begin{multline*}
\gamma_+\left(W_k^i,K\right)
\stackrel{\eqref{eq:for quote upper W}\wedge\eqref{eq:hi big inductive assumption}\wedge\eqref{eq:k1}\wedge \eqref{eq:use inductive k3}}{\le} k_2(K)\max\left\{1,\frac{2k_3(K)k_1(K)^2}{M_{h_i(k)}^{1/k_2(K)}}\right\}\nonumber\\
\stackrel{\eqref{eq:def k3}\wedge\eqref{eq:hi big inductive assumption}}{\le} k_3(K)\max\left\{1,\frac{2k_3(K)^3}{M_{k_3(K)}^{1/k_3(K)}}\right\}\stackrel{\eqref{eq:def Mk}}{=} k_3(K).\nonumber
\end{multline*}
This completes the inductive proof of~\eqref{eq:induction, ig hi big}.

Define
\begin{equation}\label{eq:def max i0}
i_0(k)\eqdef \max\left\{i\in\{0,\ldots,\ell(k)-1\}:\ h_i(k)>k_3(K) \right\}.
\end{equation}
Note that since $h_0(k)=k$, the maximum in~\eqref{eq:def max i0} is well defined. By~\eqref{eq:induction, ig hi big} we have
\begin{equation}\label{eq:bound at k3}
\gamma_+\left(W_k^{i_0(k)},K\right)\le k_3(K).
\end{equation}
A recursive application of~\eqref{eq:crude use of calculus}, combined with~\eqref{eq:bound at k3}, yields the estimate
\begin{equation}\label{eq:product estimate1}
\gamma_+(H_k,K)\le k_3(K)\prod_{i=i_0(k)+1}^{\ell(k)}\left(2k_2(K)\gamma_+\left(F_{j(h_i(k))}(h_i(k)),K\right)^2\right).
\end{equation}
By~\eqref{eq:def max i0}, for every $i\in \{i_0(k)+1,\ldots,\ell(k)\}$ we have $h_i(k)\le k_3(K)$. Due to the strict monotonicity appearing in~\eqref{eq:monotonicity h}, it follows that the number of terms in the product appearing in~\eqref{eq:product estimate1} is at most $k_3(K)$, and therefore
\begin{equation}\label{eq:our large k bound}
\gamma_+(H_k,K)\le k_3(K)\left(2k_2(K)\right)^{k_3(K)}\prod_{r=1}^{k_3(K)}\gamma_+\left(F_{j(r)}(r),K\right)^2.
\end{equation}
We have proved that~\eqref{eq:our large k bound} holds true for every integer $k>k_3(K)$. Note that the upper bound in~\eqref{eq:our large k bound} is independent of $k$, so in combination with~\eqref{eq:Hk finite} this completes the proof of~\eqref{eq:goal Hk good}.
\end{proof}

\begin{proof}[Proof of Theorem~\ref{thm:existence intro}]
Lemma~\ref{lem:zigzag main iteration} applies when $\K$ consists of all $K:X\times X\to [0,\infty)$ of the form  $K(x,y)=\|x-y\|_X^2$, where $(X,\|\cdot\|_X)$ ranges over all super-reflexive Banach spaces. Indeed, hypotheses~\eqref{eq:K finiteness assumption} and ~\eqref{eq:k1} of Lemma~\ref{lem:zigzag main iteration} are nothing more than the assertions of Corollary~\ref{cor:apply gen zigzag to base graphs}. Hypothesis~\eqref{eq:k2} of Lemma~\ref{lem:zigzag main iteration} holds true as well since, by Corollary~\ref{coro:decay for super-reflexive intro}, every super-reflexive Banach space $(X,\|\cdot\|_X)$ satisfies~\eqref{eq:graph version of decay in section}, so we may take $$
k_2(X)\eqdef \max\left\{C(X),\frac{1}{\e(X)}\right\}.
$$

Let $d\in \N$ and $\{H_i\}_{i=1}^\infty$ be the output of Lemma~\ref{lem:zigzag main iteration}. Recalling the notation of Remark~\ref{rem:cycle bounds}, $C_d^\circ$ denotes the cycle of length $d$ with self loops, and $C_9$ denotes the cycle of length $9$ without self loops. For each $i\in \N$, since $H_i$ is $d$-regular, we may form the zigzag product $H_i\oz C_d^\circ$, which is a $9$-regular graph with $d|V(H_i)|$ vertices. We can therefore consider the graph
$$
H_i^*\eqdef \left(H_i\oz C_d^\circ\right)\circr C_9.
$$
Thus $\{H_i^*\}_{i=1}^\infty$ are $3$-regular graphs with $\lim_{i\to \infty} |V(H_i^*)|=\infty$. By Theorem~\ref{thm:sub} and part~\eqref{item:replacement product} of Theorem~\ref{thm:products intro}, for every super-reflexive Banach space $(X,\|\cdot\|_X)$ we have
$$
\gamma_+\left(H_i^*,\|\cdot\|_X^2\right)\le 9\gamma_+\left(H_i,\|\cdot\|_X^2\right)\gamma_+\left(C_d^\circ,\|\cdot\|_X^2\right)^2\gamma_+\left(C_9,\|\cdot\|_X^2\right)^2.
$$
By Lemma~\ref{lem:general graph} we have $\gamma_+\left(C_d^\circ,\|\cdot\|_X^2\right)\le 12d^2$ and  $\gamma_+\left(C_9,\|\cdot\|_X^2\right)\le 648$ (since $C_9$ is not bipartite). Therefore $\gamma_+\left(H_i^*,\|\cdot\|_X^2\right)\lesssim d^4 \gamma_+\left(H_i,\|\cdot\|_X^2\right)$, so due to~\eqref{eq:good for all K} the graphs $\{H_i^*\}_{i=1}^\infty$ satisfy the conclusion of Theorem~\ref{thm:existence intro}.
\end{proof}

\begin{remark}\label{rem:lafforgue}
V. Lafforgue asked~\cite{Lafforgue} whether there exists a sequence
of bounded degree graphs $\{G_k\}_{k=1}^\infty$ that does not admit
a coarse embedding (with the same moduli) into any $K$-convex Banach
space. A positive answer to this question follows from our methods. Independently of our work, Lafforgue~\cite{Laf09} managed to solve this problem as well, so we only sketch the argument.  An inspection of Lafforgue's proof in~\cite{Lafforgue}
shows that his method produces regular graphs $\{H_j(k)\}_{j,k\in \N}$ such
that for each $k\in \N$ the graphs $\{H_j(k)\}_{j\in \N}$ have
degree $d_k$, their cardinalities are unbounded, and for every
$K$-convex Banach space $(X,\|\cdot\|_X)$ there is some $k\in \N$
for which $\sup_{j\in \N}\gamma_+(H_j(k),\|\cdot\|_X^2)<\infty$. The
problem is that the degrees $\{d_k\}_{k\in \N}$ are unbounded, but
this can be overcome as above by applying the zigzag product with a
cycle with self loops. Indeed, define $G_j(k)=H_j(k)\oz
C_{d_k}^\circ$. Then $G_j(k)$ is $9$-regular, and as argued in the
proof of Theorem~\ref{thm:existence intro}, we still have
$\sup_{j\in \N}\gamma_+(G_j(k),\|\cdot\|_X^2)<\infty$. To get a
single sequence of graphs that does not admit a coarse embedding
 into any $K$-convex Banach space, fix a bijection $\psi=(a,b):\N\to \N\times \N$, and define
$G_m=G_{a(m)}(b(m))$. The graphs $G_m$ all have degree $9$. If $X$
is $K$-convex then choose $k\in \N$ as above. If we let $m_j\in \N$
be such that $\psi(m_j)=(j,k)$ then we have shown that the graphs
$\{G_{m_j}\}_{j=1}^\infty$ are arbitrarily large, have bounded
degree, and  satisfy $\sup_{j\in
\N}\gamma_+(G_{m_j},\|\cdot\|_X^2)<\infty$. The argument that was
presented in Section~\ref{sec:coarse} implies that
$\{G_m\}_{m=1}^\infty$ do not embed coarsely into $X$.
\end{remark}

\section{The heat semigroup on the tail space}\label{sec:heat}
This section contains estimates that will be crucially used in the proof of Lemma~\ref{lem:base in section}, in addition to geometric results and open questions of independent interest. We start the discussion by recalling some basic definitions, and setting some (mostly standard) notation on vector-valued Fourier analysis. Let $(X,\|\cdot\|_X)$ be a Banach space. We assume throughout that $X$ is a Banach space over the complex scalars, though, by a standard complexification argument, our results hold also for
Banach spaces over $\R$.

Given a measure space $(\Omega,\mu)$ and $p\in [1,\infty)$, we
denote as usual by $L_p(\mu,X)$ the space of all measurable
$f:\Omega \to X$ satisfying
$$
\|f\|_{L_p(\mu,X)}\eqdef \left(\int_{\Omega}\|f\|_X^pd\mu\right)^{1/p}<\infty.
$$
When $X=\C$ we use the standard notation $L_p(\mu)=L_p(\mu, \C)$. When $\Omega$ is a finite set we denote by $L_p(\Omega,X)$ the space $L_p(\mu,X)$, where $\mu$ is the normalized counting measure on $\Omega$.

For $n\in \N$ and $A\subseteq \{1,\ldots,n\}$, the Walsh function $W_A:\F_2^n\to \{-1,1\}$ is defined by
$$
W_A(x)\eqdef(-1)^{\sum_{j\in A}x_j}.
$$
Any $f:\F_2^n\to X$ has the expansion
$$
f=\sum_{A\subseteq \{1,\ldots, n\}} \hat{f}(A)W_A,
$$
where
$$
\hat{f}(A)\eqdef \frac{1}{2^n}\sum_{x\in \F_2^n} f(x)W_A(x)\in X.
$$
For $\varphi:\F_2^n\to \C$ and $f:\F_2^n\to X$, the convolution $\varphi*f:\F_2^n\to X$ is defined as usual by
$$
\varphi*f(x)\eqdef\frac{1}{2^n}\sum_{w\in \F_2^n} \varphi(x-w)f(w)=\sum_{A\subseteq \{1,\ldots,n\}} \hat{\varphi}(A)\hat{f}(A)W_A(x).
$$

For $k\in \{1,\ldots,n\}$ and $p\in [1,\infty]$ we let $L_p^{\ge k}(\F_2^n,X)$ denote the subspace of $L_p(\F_2^n,X)$ consisting of those $f:\F_2^n\to X$ that satisfy $\hat{f}(A)=0$ for all $A\subseteq \{1,\ldots,n\}$ with $|A|<k$.

Let $e_1,\ldots,e_n$ be the standard basis of $\F_2^n$. For $j\in \{1,\ldots,n \}$ define $\partial_jf:\F_2^n\to X$ by
$$
\partial_jf(x)\eqdef \frac{f(x)-f(x+e_j)}{2}.
$$
Thus
$$
\partial_jf=\sum_{\substack{A\subseteq \{1,\ldots,n\}\\ j\in A}} \hat{f}(A)W_A,
$$
and
$$
\Delta f\eqdef \sum_{j=1}^n \partial_j f =\sum_{A\subseteq \{1,\ldots,n\}} |A|\hat{f}(A) W_A.
$$
For every $z\in \C$ we then have 
\begin{equation}\label{eq:riesz}
e^{z\Delta} f=\sum_{A\subseteq \{1,\ldots,n\}} e^{z|A|}\hat{f}(A) W_A=R_z*f,
\end{equation}
where
\begin{equation}\label{eq:def:riesz}
R_z(x)\eqdef \prod_{j=1}^n \left(1+e^z(-1)^{x_j}\right)=\left(1-e^z\right)^{\|x\|_1}\left(1+e^z\right)^{n-\|x\|_1},
\end{equation}
and we identify $\F_2^n$ with $\{0,1\}^n\subseteq \R^n$. Hence, for every $x\in \F_2^n$ we have
\begin{equation} \label{eq:primal semigroup}
e^{z\Delta} f(x) = \sum_{w\in\mathbb F_2^n}  \left( \frac{1-e^{z}}2 \right)^{\|x-w\|_1}
\left( \frac{1+e^{z}}2 \right)^{n-\|x-w\|_1}
f(w).
\end{equation}
In particular,
\begin{equation}\label{eq:matrix form heat}
\forall\, x,y\in \F_2^n,\quad (e^{z\Delta}\delta_x)(y)= \left(
\frac{1-e^{z}}2 \right)^{\|x-y\|_1} \left( \frac{1+e^{z}}2
\right)^{n-\|x-y\|_1},
\end{equation}
where $\delta_x(w)\eqdef \1_{\{x=w\}}$ is the Kronecker delta.

Given $n\in \N$ and $f:\F_2^n\to X$, the Rademacher projection~\cite{MP-type-cotype} of $f$ is defined by
$$
\Rad(f)\eqdef \sum_{j=1}^n \hat{f}(\{j\})W_{\{j\}}.
$$
The $K$-convexity constant of $X$ is defined~\cite{MP-type-cotype} by
$$
K(X)\eqdef \sup_{n\in \N} \left\|\Rad\right\|_{L_2(\F_2^n,X)\to L_2(\F_2^n,X)}.
$$
If $K(X)<\infty$ then $X$ is said to be $K$-convex. Pisier's deep $K$-convexity theorem~\cite{Pisier-K-convex} asserts that $X$ is $K$-convex if and only if it does not contain copies of $\{\ell_1^n\}_{n=1}^\infty$ with distortion arbitrarily close to $1$, i.e., for all $n\in \N$ we have
$$
\inf_{ T\in \mathscr{L}(\ell_1^n,X)}\|T\|_{\ell_1^n\to X}\cdot\|T^{-1}\|_{T(\ell_1^n)\to \ell_1^n}=1,
$$
where $\mathscr{L}(\ell_1^n,X)$ denotes the space of linear operators $T:\ell_1^n\to X$ (and we use the convention $\|T^{-1}\|_{T(\ell_1^n)\to \ell_1^n}=\infty$ if $T$ is not injective).

Our main result in this section is the following theorem.

\begin{theorem}[Decay of the heat semigroup on the tail space]\label{thm:AB}
For every $K,p\in (1,\infty)$ there are $A(K,p)\in (0,1)$ and $B(K,p), C(K,p)\in (2,\infty)$ such that for every $K$-convex Banach $(X,\|\cdot\|_X)$ with $K(X)\le K$, every $k,n\in \N$ and every $t\in (0,\infty)$,
\begin{equation}\label{eq:AB}
\left\|e^{-t\Delta}\right\|_{L_p^{\ge k}(\F_2^n,X)\to L_p^{\ge k}(\F_2^n,X)}\le C(K,p)e^{-A(K,p)k\min\left\{t,t^{B(K,p)}\right\}}.
\end{equation}
%where $A=A(K(X),p), B=B(K(X),p), C=C(K(X),p)$.
\end{theorem}

The fact that Theorem~\ref{thm:AB} assumes that $X$ is $K$-convex is not an artifact of our proof: we have, in fact, the following converse statement.

\begin{theorem}\label{thm:K-convec char}
Let $X$ be a Banach space $(X,\|\cdot\|_X)$ for which exist $k\in \N$, $p\in (1,\infty)$ and $t\in (0,\infty)$ such that
\begin{equation}\label{eq:<1}
\sup_{n\in \N} \left\|e^{-t\Delta}\right\|_{L_p^{\ge k}(\F_2^n,X)\to L_p^{\ge k}(\F_2^n,X)}<1.
\end{equation}
Then $X$ is $K$-convex.
\end{theorem}

\begin{remark}\label{rem:conj<1} We conjecture that any $K$-convex Banach space satisfies~\eqref{eq:<1}  for every $k\in \N$, $p\in (1,\infty)$ and $t\in (0,\infty)$. Theorem~\ref{thm:AB} implies~\eqref{eq:<1} if $k$ or $t$ are large enough, but, due to the factor $C(K,p)$ in~\eqref{eq:AB}, it does not imply~\eqref{eq:<1} in its entirety. The factor $C(K,p)$ in~\eqref{eq:AB} does not have impact on the application of Theorem~\ref{thm:AB} that we present here; see Section~\ref{sec:base}.
\end{remark}

\subsection{Warmup: the tail space of scalar valued functions}\label{sec:scalar}

Before passing to the proofs of Theorem~\ref{thm:AB} and Theorem~\ref{thm:K-convec char}, we address separately the classical scalar case $X=\C$, since it already exhibits interesting open questions. The problem was studied by P.-A. Meyer~\cite{Meyer} who proved Lemma~\ref{lem:meyer} below. We include its proof here since it is not stated explicitly in this way in~\cite{Meyer}, and moreover Meyer studies this problem with $\F_2^n$ replaced by $\R^n$ equipped with the standard Gaussian measure (the proof in the discrete setting does not require anything new. We warn the reader that the proof in~\cite{Meyer} contains an inaccurate duality argument).

\begin{lemma}[P.-A. Meyer]\label{lem:meyer} For every $p\in [2,\infty)$
there exists $c_p\in (0,\infty)$ such that for every $k\in \N$,
every tail space function $f\in L_p^{\ge k}(\F_2^n)$ and every
time $t\in (0,\infty)$,
\begin{equation}\label{e:meyer}
\left\|e^{-t\Delta} f\right\|_{L_p(\F_2^n)}\le
e^{-c_pk\min\{t,t^2\}}\left\|f\right\|_{L_p(\F_2^n)}.
\end{equation}
Hence,
\begin{equation}\label{eq:meyer laplacian}
\left\|\Delta f\right\|_{L_p(\F_2^n)}\gtrsim c_p\sqrt{k}\cdot \|f\|_{L_p(\F_2^n)}.
\end{equation}
\end{lemma}

\begin{proof}
The estimate~\eqref{eq:meyer laplacian} follows immediately
from~\eqref{e:meyer} as follows.
\begin{multline*}
\|f\|_{L_p(\F_2^n)}=\left\|\int_0^\infty e^{-t\Delta}\Delta f dt\right\|_{L_p(\F_2^n)}\le \int_0^\infty \left\|e^{-t\Delta}\Delta f\right\|_{L_p(\F_2^n)}dt\\ \stackrel{\eqref{e:meyer}}{\le}
\left(\int_0^1 e^{-c_pkt^2}dt+\int_1^\infty
e^{-c_pkt}dt\right)\|\Delta f\|_{L_p(\F_2^n)}\lesssim \frac{\|\Delta
f\|_{L_p(\F_2^n)}}{c_p\sqrt{k}}.
\end{multline*}

To prove~\eqref{e:meyer}, we may assume that $\|f\|_{L_p(\F_2^n)}=1$. Since
$p\ge 2$, it follows that
\begin{equation}\label{eq:spectral}
\left\|e^{-t\Delta} f\right\|_{L_2(\F_2^n)}\le e^{-kt}\|f\|_{L_2(\F_2^n)}\le e^{-kt}\|f\|_{L_p(\F_2^n)}= e^{-kt}.
\end{equation}
By classical hypercontractivity estimates~\cite{Bon70,Beck75}, if we
define
\begin{equation}\label{eq:def q}
q\eqdef 1+e^{2t}(p-1).
\end{equation}
then
\begin{equation}\label{eq:use beckner}
\left\|e^{-t\Delta}f\right\|_{L_q(\F_2^n)}\le \|f\|_{L_p(\F_2^n)}=1. \end{equation}
 Since
$p\in [2,q]$ we may consider $\theta\in [0,1]$ given by
\begin{equation}\label{eq:def theta1}
\frac{1}{p}=\frac{\theta}{2}+\frac{1-\theta}{q}.
\end{equation}
Now,
\begin{multline}\label{eq:hypercontractive}
\left\|e^{-t\Delta}f\right\|_{L_p(\F_2^n)}\le
\left\|e^{-t\Delta}f\right\|_{L_2(\F_2^n)}^\theta\cdot\left\|e^{-t\Delta}f\right\|_{L_q(\F_2^n)}^{1-\theta}
\\\stackrel{\eqref{eq:spectral}
\wedge\eqref{eq:use beckner}}{\le}
e^{-kt\theta}\stackrel{\eqref{eq:def q}\wedge\eqref{eq:def
theta1}}{=}\exp\left(-\frac{2(p-1)kt(e^{2t}-1)}{p\left(e^{2t}(p-1)-1\right)}\right).
\end{multline}
By choosing $c_p$ appropriately, the desired
estimate~\eqref{e:meyer} is a
consequence~\eqref{eq:hypercontractive}.
\end{proof}

\begin{remark}\label{rem:open meyer}
For the purpose of the geometric applications that are contained in the
present paper we need to understand the vector-valued analogue of
Lemma~\ref{lem:meyer}, i.e., Theorem~\ref{thm:AB}. Nevertheless, The following interesting questions
seem to be open for scalar-valued functions.
\begin{enumerate}
\item Can one
prove Lemma~\ref{lem:meyer} also when $p\in (1,2)$? Note that while
$\Delta$ and $e^{-t\Delta}$ are self-adjoint operators, one needs to
understand the dual norm on $L_p^{\ge k}(\F_2^n,\R)^*$ in order to
use duality here.
\item What is the correct asymptotic dependence on $k$
in~\eqref{eq:meyer laplacian}? Specifically, can~\eqref{eq:meyer
laplacian} be improved to
\begin{equation}\label{eq:lower linear}
\|\Delta f\|_p\gtrsim_p k\|f\|_p?
\end{equation}
\item As a potential way to prove~\eqref{eq:lower linear},  can one improve~\eqref{e:meyer} to
\begin{equation}\label{eq:conjectured linear heat}
f\in L_p^{\ge k}(\F_2^n)\implies \forall t\in (0,\infty),\quad \left\|e^{-t\Delta}f\right\|_{L_p(\F_2^n)}\le e^{-c_p k t}\|f\|_{L_p(\F_2^n)}?
\end{equation}
\end{enumerate}
As some evidence for~\eqref{eq:conjectured linear heat}, P. Cattiaux  proved (private
communication) the case $k=1$, $p=4$ of~\eqref{eq:conjectured linear heat} when the heat
semigroup on $\F_2^n$ is replaced by the Ornstein-Uhlenbeck semigroup on
$\R^n$. Specifically, let $\gamma_n$ be the standard Gaussian
measure on $\R^n$ and consider the Ornstein-Uhlenbeck operator
$L=\Delta - x\cdot \nabla$. Cattiaux proved that there exists a
universal constant $c\in (0,\infty)$ such that for every $f\in
L_4(\gamma_n,\R)$ and every $t\in (0,\infty)$,
\begin{equation}\label{eq:catt}
\int_{\R^n} f d\gamma_n=0\implies
\left\|e^{-tL}f\right\|_{L_4(\gamma_n)}\le
e^{-ct}\|f\|_{L_4(\gamma_n)}. \end{equation}
We shall now present a sketch of Cattiaux's proof of~\eqref{eq:catt}. By differentiating at $t=0$, integrating by parts, and using the semigroup property, one sees that~\eqref{eq:catt} is equivalent to the following assertion.
\begin{equation}\label{eq:catt differentiated}
\int_{\R^n} f d\gamma_n=0\implies \int_{\R^n} f^4d\gamma_n\lesssim \int_{\R^n}f^2\|\nabla f\|_2^2d\gamma_n.
\end{equation}
The Gaussian Poincar\'e inequality (see~\cite{BU83,Led01}) applied
to $f^2$ implies that
$$
\int_{\R^n} f^4d\gamma_n-\left(\int_{\R^n}f^2d\gamma_n\right)^2\lesssim \int_{\R^n} f^2\left\|\nabla f\right\|_2^2d\gamma_n.
$$
The desired inequality~\eqref{eq:catt differentiated} would therefore follow from
\begin{equation}\label{eq:catt desired new}
\int_{\R^n} f d\gamma_n=0\implies\left(\int_{\R^n}f^2d\gamma_n\right)^2\lesssim \int_{\R^n} f^2\left\|\nabla f\right\|_2^2d\gamma_n.
\end{equation}
Fix $M\in (0,\infty)$ that will be determined later. Define $\phi_M:\R\to \R$ by
\begin{equation}\label{eq:def phiM}
\phi_M(x)\eqdef \left\{\begin{array}{ll} 0&\mathrm{if\ } |x|\le M,\\
2(x-M)&\mathrm{if\ } x\in [M,2M],\\
2(x+M)&\mathrm{if\ } x\in [-2M,-M],\\
x&\mathrm{if\ } |x|\ge 2M.
\end{array}\right.
\end{equation}
Since $|\phi'|\le 2$, an application of the Gaussian Poincar\'e inequality to $\phi\circ f$ yields the estimate
\begin{equation}\label{eq:poincare composed}
\int_{\R^n} (\phi\circ f)^2d\gamma_n-\left(\int_{\R^n}\phi\circ f d\gamma_n\right)^2\stackrel{\eqref{eq:def phiM}}{\lesssim} \int_{\{|f|\ge M\}} \left\|\nabla f\right\|_2^2d\gamma_n.
\end{equation}
Now,
\begin{equation}\label{eq:lower composed}
\int_{\R^n} (\phi\circ f)^2d\gamma_n\stackrel{\eqref{eq:def phiM}}{\ge} \int_{\{|f|\ge 2M\}}f^2d\gamma_n\ge \int_{\R^n} f^2d\gamma_n-4M^2.
\end{equation}
Also,
\begin{equation}\label{eq:upper on M set}
\int_{\{|f|\ge M\}} \left\|\nabla f\right\|_2^2d\gamma_n\le \frac{1}{M^2}\int_{\R^n} f^2\|\nabla f\|_2^2d\gamma_n.
\end{equation}
If in addition $\int_{\R^n}fd\gamma_n=0$ then
\begin{equation}\label{eq:use expectation 0}
\left|\int_{\R^n}\phi\circ fd\gamma_n\right|=\left|\int_{\R^n}(\phi\circ f-f)d\gamma_n\right|\stackrel{\eqref{eq:def phiM}}{=}\left|\int_{\{|f|\le 2M\}}(\phi\circ f-f)d\gamma_n\right|
\le 4M.
\end{equation}
Hence, by \eqref{eq:poincare composed}, \eqref{eq:lower composed}, \eqref{eq:upper on M set} and~\eqref{eq:use expectation 0},
\begin{equation}\label{eq:to optimize M}
\int_{\R^n} f d\gamma_n=0\implies \int_{\R^n}f^2d\gamma_n\lesssim M^2+\frac{1}{M^2}\int_{\R^n} f^2\|\nabla f\|_2^2d\gamma_n.
\end{equation}
The optimal choice of $M$ in~\eqref{eq:to optimize M} is
$$
M=\left(\int_{\R^n} f^2\|\nabla f\|_2^2d\gamma_n\right)^{1/4},
$$
yielding the desired inequality~\eqref{eq:catt desired new}. It would be interesting to generalize the above argument so as to extend~\eqref{eq:catt} to the
setting of functions in all the Hermite tail spaces $\{L_p^{\ge
k}(\gamma_n,\R)\}_{k\in \N}$ (i.e., functions whose Hermite coefficients of
degree less than $k$ vanish).
\end{remark}

\subsection{Proof of Theorem~\ref{thm:AB}}\label{sec:decay heat}
For every $m\in \{1,\ldots,n\}$ consider the level-$m$ Rademacher projection given by
\[ \Rad_m (f)\eqdef \sum_{\substack{A\subseteq \{1,\ldots,n\}\\ |A|=m}} \hat f(A) W_A .\]
Thus $\Rad_1=\Rad$ and for every $z\in \C$ we have
$$
e^{z\Delta}=\sum_{m=0}^n e^{zm}\Rad_m.
$$
We shall
use the following deep theorem of Pisier~\cite{Pisier-K-convex}.

\begin{theorem}[Pisier]\label{thm:Pis K-conv}
For  every $K,p\in(1,\infty)$ there exist $\phi=\phi(K,p)\in (0,\pi/4)$ and $M=M(K,p)\in (2,\infty)$
such that for every Banach space $X$ satisfying $K(X)\le K$, $n\in \N$ and $z\in \C$, we have
\begin{equation}\label{eq:holomorphic extension sector}
 |\arg z|\le \phi\implies \left\|e^{-z\Delta}\right\|_{L_p(\F_2^n,X)\to L_p(\F_2^n,X)}\le M.
\end{equation}
\end{theorem}
One can give explicit bounds on $M,\phi$ in terms of $p$ and $K$; see \cite{Mau03}. We will
require the following standard corollary of Theorem~\ref{thm:Pis
K-conv}. Define
$$
a=\frac{\pi}{\tan \phi},
$$ so that all the points in the
open segment joining $a-i\pi$ and $a+i\pi$ have argument at most
$\phi$. Then
\begin{equation}\label{eq:higher rad bound}
\left\|\Rad_m\right\|_{L_p(\F_2^n,X)\to L_p(\F_2^n,X)}\le M e^{am}.
\end{equation}
Indeed,
\begin{equation*}
\frac{1}{2\pi} \int_{-\pi}^\pi e^{imt} e^{-(a+it)\Delta} dt = \frac{1}{2\pi}
\int_{-\pi}^\pi e^{imt} \sum_{k=0}^n e^{-(a+it)k} \Rad_k dt=e^{-ma}\Rad_m.
\end{equation*}
Now~\eqref{eq:higher rad bound} is deduced by convexity as follows.
\begin{equation*}
 \left\|\Rad_m\right\|_{L_p(\F_2^n,X)\to L_p(\F_2^n,X)}
 \le \frac{e^{ma}}{2\pi} \int_{-\pi}^\pi
 \left\|e^{-(a+it)\Delta}\right\|_{L_p(\F_2^n,X)\to L_p(\F_2^n,X)}
dt \le Me^{ma} .  \end{equation*}

It follows that
\begin{equation}\label{eq:big real}
\Re z\ge 2a\implies \left\|e^{-z\Delta}\right\|_{L_p^{\ge k}(\F_2^n,X)\to L_p^{\ge k}(\F_2^n,X)}\le \frac{M}{1-e^{-a}}e^{-k\Re z/2}\le \frac{M}{1-e^{-a}}e^{-ka}.
\end{equation}
Indeed,
 \begin{multline*}
 \left\|e^{-z\Delta}\right\|_{L_p^{\ge k}(\F_2^n,X)\to L_p^{\ge k}(\F_2^n,X)}
 =
 \left\|\sum_{m=k}^n e^{-zm} \Rad_m \right\|_{L_p^{\ge k}(\F_2^n,X)\to L_p^{\ge k}(\F_2^n,X)} \\
 \stackrel{\eqref{eq:higher rad bound}}{\le}
 \sum_{m=k}^n e^{-m\Re z}Me^{am}\le M\sum_{m=k}^n e^{-m\Re z/2} = \frac{M}{1-e^{-\Re z/2}} e^{-k\Re z/2}\le \frac{M}{1-e^{-a}}e^{-k\Re z/2}.
  \end{multline*}

The ensuing argument is a quantitative variant of the proof of the
main theorem of Pisier in~\cite{pisier-2007}. Let
$$
r\eqdef 2\sqrt{a^2+\pi^2},
$$ and define
$$
V\eqdef\left\{z\in \mathbb C:\ |z|\le r\ \wedge\ |\arg z|\le \phi\right\}.
$$
The set $V\subseteq \mathbb C$ is depicted in Figure~\ref{fig:triangle}.

 \begin{figure}[ht]
 \begin{center}
   \includegraphics{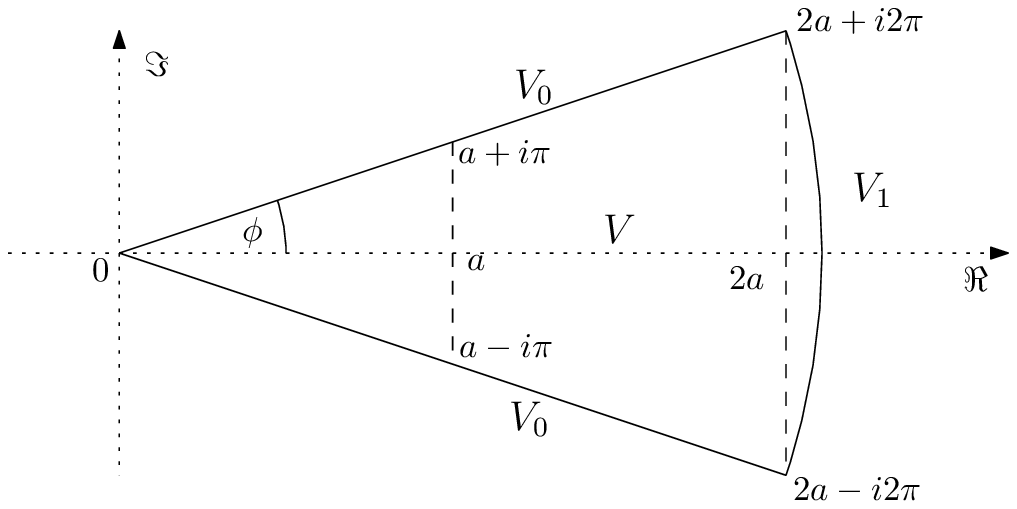}
 \end{center}
 \caption{The sector $V\subseteq \C$.}
 \label{fig:triangle}
 \end{figure}
Denote
 $$
 V_0\eqdef\left\{ x\pm i x\tan \phi:\; x\in[0,2a) \right\},
 $$ and
 $$V_1\eqdef\left\{re^{i\theta}:\; |\theta|\le \phi\right\},
 $$ so that we have the disjoint union $\partial V=V_0\cup V_1$.

Fix $t\in(0,2a)$. Let $\mu_t$ be the harmonic measure corresponding to $V$ and $t$, i.e., $\mu_t$ is the Borel probability measure  on $\partial V$
such that for every bounded analytic function $f:V\to \C$ we have
\begin{equation}\label{eq:harmonic}
f(t) =\int_{\partial V} f(z) d\mu_t(z) .
\end{equation}
We refer to~\cite{GM08} for more information on this topic and the
ensuing discussion. For concreteness, it suffices to recall here that for every Borel set
$E\subseteq
\partial V$ the number $\mu_t(E)$ is the probability that the standard $2$-dimensional Brownian motion
starting at $t$ exits $V$ at $E$. Equivalently, by conformal
invariance, $\mu_t$ is the push-forward of the normalized Lebesgue
measure on the unit circle $S^1$ under the Riemann mapping from the
unit disk to $V$ which takes the origin to $t$.

Denote $$
\theta_t\eqdef\mu_t(V_1),
$$ and write
\begin{equation}\label{eq:break mu_t}
\mu_t=(1-\theta_t)\mu_t^0+\theta_t
\mu_t^1,
\end{equation}
where $\mu_t^0,\mu_t^1$ are probability measures on $V_0,V_1$, respectively. We will use the following bound on $\theta_t$, whose proof is standard.
\begin{lemma}\label{lem:theta bound} For every $t\in (0,2a)$ we have
\begin{equation}\label{eq:theta lower}
\theta_t\ge \frac12 \left(\frac{t}{r}\right)^{\frac{\pi}{2\phi}}.
\end{equation}
\end{lemma}
\begin{proof}
This is  an exercise in conformal invariance. Let $\mathbb D=\{z\in \C:\ |z|\le 1\}$ denote the unit disk centered at the origin, and let $\mathbb D_+$ denote the intersection of $\mathbb D$ with the right half plane $\{z\in \mathbb C:\ \Re z\ge 0\}$. The mapping $h_1:V\to \mathbb D_+$ given by
$$
h_1(z)\eqdef \left(\frac{z}{r}\right)^{\frac{\pi}{2\phi}}
$$
is a conformal equivalence between $V$ and $\mathbb D_+$. Let $\mathbb Q_+=\{x+iy:\ x,y\in [0,\infty)\}$ denote the positive quadrant. The M\"obius transformation $h_2:\mathbb D_+\to \mathbb Q_+$ given by
$$
h_2(z)\eqdef -i\cdot\frac{z+i}{z-i}
$$
is a conformal equivalence between $\mathbb D_+$ and $\mathbb Q_+$. The mapping $h_3(z)\eqdef z^2$ is a conformal equivalence between $\mathbb Q_+$ and the upper half-plane $\mathbb H_+=\{z\in \C:\ \Im(z)\ge 0\}$. Finally, the M\"obius transformation
 $$
 h_4(z)\eqdef \frac{z-i}{z+i}
  $$
  is a conformal equivalence between $\mathbb H_+$ and $\mathbb D$. By composing these mappings, we obtain the following conformal equivalence between $V$ and $\mathbb D$.
$$
F(z)\eqdef (h_4\circ h_3\circ h_2\circ h_1)(z)=\frac{-\left(\left(\frac{z}{r}\right)^{\frac{\pi}{2\phi}}+i\right)^2
-i\left(\left(\frac{z}{r}\right)^{\frac{\pi}{2\phi}}-i\right)^2}
{-\left(\left(\frac{z}{r}\right)^{\frac{\pi}{2\phi}}+i\right)^2+i
\left(\left(\frac{z}{r}\right)^{\frac{\pi}{2\phi}}-i\right)^2}.
$$
Therefore, the mapping $G:V\to \mathbb D$ given by
$$
G(z)\eqdef \frac{F(z)-F(t)}{1-\overline{F(t)}F(z)}
$$
is a conformal equivalence between $V$ and $\mathbb D$ with $G(t)=0$.

By conformal invariance, $\theta_t$ is the length of the arc $G(V_1)\subseteq \partial \mathbb D=S^1$, divided by $2\pi$. Writing $s=h_1(t)=(t/r)^{\pi/(2\phi)}\in (0,1)$, we have
$$
G(2a+i2\pi)=\frac{-4s(s^2-1)-i\left((s^2-1)^2-4s^2\right)}{(s^2+1)^2},
$$
and
$$
G(2a-i2\pi)=\frac{4s(s^2-1)-i\left((s^2-1)^2-4s^2\right)}{(s^2+1)^2}.
$$
It follows that if $s\ge \sqrt{2}-1$ then $\theta_t\ge \frac12$, and if $s<\sqrt{2}-1$ then
\begin{equation}\label{eq:elementary arcsin}
\theta_t=\frac{1}{\pi}\arcsin\left(\frac{4s(1-s^2)}{(s^2+1)^2}\right)\ge \frac{s}{2},
\end{equation}
where the rightmost inequality in~\eqref{eq:elementary arcsin} follows from elementary calculus.
\end{proof}

\begin{lemma}\label{lem:strip}
For every $\e\in (0,1)$ there exists a bounded analytic function  $\Psi_\e^t:V\to \mathbb C$ satisfying
\begin{itemize}
\item
$\Psi_\e^t(t)=1$,
\item  $|\Psi_\e^t(z)|= \e$ for every $z\in V_0$,
\item $|\Psi_\e^t(z)|= \frac{1}{\e^{(1-\theta_t)/\theta_t}}$ for every $z\in V_1$.
\end{itemize}
\end{lemma}

\begin{proof} The proof is the same as the proof of Claim 2 in~\cite{pisier-2007}. We sketch it briefly for the sake of completeness. Consider the strip $S=\{z\in \mathbb C:\ \Re(z)\in [0,1]\}$ and for $j\in \{0,1\}$ let $S_j=\{z\in \mathbb C:\ \Re(z)=j\}$. As explained in~\cite[Claim~1]{pisier-2007}, there exists a conformal equivalence $h:V\to S$ such that $h(t)=\theta_t$,  $h(V_0)=S_0$ and $h(V_1)=S_1$. Now define $\Psi_\e(z)\eqdef \e^{1-\frac{h(z)}{\theta_t}}.$
\end{proof}

\begin{proof}[Proof of Theorem~\ref{thm:AB}] Take $t\in (0,\infty)$. If $t\ge 2a$ then by~\eqref{eq:big real} we have
\begin{equation}\label{eq:bigger than 2a}
\left\|e^{-t\Delta}\right\|_{L_p^{\ge k}(\F_2^n,X)\to L_p^{\ge k}(\F_2^n,X)}\le \frac{M}{1-e^{-a}}e^{-kt/2}.
\end{equation}
Suppose therefore that $t\in (0,2a)$. Fix $\e\in (0,1)$ that will be determined later, and let $\Psi_\e^t$ be the function from Lemma~\ref{lem:strip}. Then
\begin{multline}\label{eq:harmonic identity}
e^{-t\Delta} =\Psi_\e^t(t)e^{-t\Delta}\stackrel{\eqref{eq:harmonic}}{=} \int_{\partial V} \Psi_\e^t(z) e^{-z\Delta} d \mu_t(z)\\
\stackrel{\eqref{eq:break mu_t}}{=} (1-\theta_t) \int_{V_0}\Psi_\e^t(z) e^{-z\Delta} d\mu^0_t(z) +
\theta_t \int_{V_1} \Psi_\e^t(z) e^{-z\Delta} d\mu^1_t(z).
\end{multline}
Hence, using~\eqref{eq:harmonic identity} in combination with Lemma~\ref{lem:strip}, Theorem~\ref{thm:Pis K-conv} and~\eqref{eq:big real}, we deduce that
\begin{multline}\label{eq:before eps choice}
\left\|e^{-t\Delta}\right\|_{L_p^{\ge k}(\F_2^n,X)\to L_p^{\ge k}(\F_2^n,X)}\le (1-\theta_t)\e M+\frac{\theta_t}{\e^{(1-\theta_t)/\theta_t}}\cdot \frac{Me^{-ka}}{1-e^{-a}}\\
\stackrel{\eqref{eq:theta lower}}{\le} \e M +\frac{M e^{-ka}}{1-e^{-a}}\cdot \frac{1}{\e^{2(r/t)^{\pi/(2\phi)}-1}}.
\end{multline}
We now choose
$$
\e=\exp\left(-\frac12\left(\frac{t}{r}\right)^{\frac{\pi}{2\phi}}ka\right),
$$
in which case~\eqref{eq:before eps choice} completes the proof of Theorem~\ref{thm:AB}, with $B(K,p)=\frac{\pi}{2\phi}$.
\end{proof}

\subsection{Proof of Theorem~\ref{thm:K-convec char}}\label{sec:K-convex char}

The elementary computation contained in Lemma~\ref{lem:normalized delta} below will be useful in ensuing considerations.
\begin{lemma}\label{lem:normalized delta}
Define $f_n:\F_2^n\to L_1(\F_2^n)$ by
\begin{equation}\label{eq:def normalized delta}
f_n(x)(y)=2^n\1_{\{x=y\}}-1.
\end{equation}
Then $f_n\in L_p^{\ge 1}(\F_2^n,L_1(\F_2^n))$, yet for every $t\in (0,\infty)$  we have
\begin{equation}\label{eq:limit ratio}
\lim_{n\to \infty} \frac{\left\|e^{-t\Delta}f_n\right\|_{L_p(\F_2^n,L_1(\F_2^n))}}{\|f_n\|_{L_p(\F_2^n,L_1(\F_2^n))}}=1,
\end{equation}
where the limit in~\eqref{eq:limit ratio} is uniform in $p\in [1,\infty)$.
\end{lemma}
\begin{proof} By definition $\sum_{x\in \F_2^n}f_n(x)=0$, i.e., $f_n\in L_p^{\ge 1}(\F_2^n,L_1(\F_2^n))$. Observe that
\begin{equation}\label{eq:norm of delta}
\|f_n\|_{L_p(\F_2^n,L_1(\F_2^n))}=2\left(1-\frac{1}{2^n}\right),
\end{equation}
and note also that for every $x,y\in \F_2^n$ we have
\begin{equation}\label{eq:f_n riesz}
f_n(x)(y)= \prod_{i=1}^n \left(1+(-1)^{x_i+y_i}\right)-1=\sum_{\substack{A\subseteq \{1,\ldots,n\}\\A\neq\emptyset}} W_A(x)W_A(y).
\end{equation}
It follows from~\eqref{eq:primal semigroup} that for every $x\in \F_2^n$ we have
\begin{multline*}
\left\|e^{-t\Delta}f_n\right\|_{L_1(\F_2^n)}=\frac{1}{2^n}\sum_{y\in \F_2^n} \left|\sum_{w\in\mathbb F_2^n}  \left( \frac{1-e^{t}}2 \right)^{\|x-w\|_1}
\left( \frac{1+e^{t}}2 \right)^{n-\|x-w\|_1}
f_n(w)(y)\right|\\\stackrel{\eqref{eq:def normalized delta}}{=}
\frac{1}{2^n}\sum_{y\in \F_2^n}\left|2^n\left( \frac{1-e^{t}}2 \right)^{\|x-y\|_1}
\left( \frac{1+e^{t}}2 \right)^{n-\|x-y\|_1}-1\right|.
\end{multline*}
Hence,
\begin{equation}\label{eq:norm of evolute}
\left\|e^{-t\Delta}f_n\right\|_{L_p(\F_2^n,L_1(\F_2^n))}=\sum_{m=0}^n\binom{n}{m}\left|\left(\frac{1-e^{-t}}{2}
\right)^m\left(\frac{1+e^{-t}}{2}\right)^{n-m}-\frac{1}{2^n}\right|.
\end{equation}

Let $U_1,\ldots,U_n$ be i.i.d. random variables such that $\Pr[U_1=0]=\Pr[U_1=1]=\frac12$. By the Central Limit Theorem,
\begin{multline}\label{eq:CLT1}
1=\lim_{n\to \infty}\Pr\left[\sum_{j=1}^n U_j\in \left(\frac12 n-n^{2/3},\frac12 n+n^{2/3}\right)\right]\\=\lim_{n\to \infty} \sum_{m\in \left(\frac12 n-n^{2/3},\frac12 n+n^{2/3}\right)\cap \N} \binom{n}{m}\frac{1}{2^n}.
\end{multline}
Similarly, if $V_1,\ldots,V_n$ are i.i.d. random variables such that $\Pr[V_1=1]=(1-e^{-t})/2$ and $\Pr[V_1=0]=(1+e^{-t})/2$, then by the Central Limit Theorem,
\begin{multline}\label{eq:CLT2}
1=\lim_{n\to \infty}\Pr\left[\sum_{j=1}^n V_j\in \left(\frac{1-e^{-t}}{2} n-n^{2/3},\frac{1-e^{-t}}{2}n+n^{2/3}\right)\right]\\
=\lim_{n\to \infty} \sum_{m\in \left(\frac{1-e^{-t}}{2}  n-n^{2/3},\frac{1-e^{-t}}{2}  n+n^{2/3}\right)\cap \N} \binom{n}{m}\left(\frac{1-e^{-t}}{2}\right)^m\left(\frac{1+e^{-t}}{2} \right)^{n-m}.
\end{multline}
Fix $\e\in (0,1)$. It follows from~\eqref{eq:CLT1}, \eqref{eq:CLT2} that for $n$ large enough we have
\begin{equation}\label{eq:CLTeps1}
\sum_{m\in \left(\frac12 n-n^{2/3},\frac12 n+n^{2/3}\right)\cap \N} \binom{n}{m}\frac{1}{2^n}\ge 1-\frac{\e}{2},
\end{equation}
and
\begin{equation}\label{eq:CLTeps2}
\sum_{m\in \left(\frac{1-e^{-t}}{2}  n-n^{2/3},\frac{1-e^{-t}}{2}  n+n^{2/3}\right)\cap \N} \binom{n}{m}\left(\frac{1-e^{-t}}{2}\right)^m\left(\frac{1+e^{-t}}{2} \right)^{n-m}\ge 1-\frac{\e}{2}.
\end{equation}
Moreover, by choosing $n$ to be large enough we can ensure that
\begin{equation}\label{eq:disjoint}
\left(\frac12 n-n^{2/3},\frac12 n+n^{2/3}\right)\cap \left(\frac{1-e^{-t}}{2}  n-n^{2/3},\frac{1-e^{-t}}{2}  n+n^{2/3}\right)=\emptyset.
\end{equation}
Since
\begin{multline*}
m\in \left(\frac12 n-n^{2/3},\frac12 n+n^{2/3}\right)\implies \\ \left(\frac{1-e^{-t}}{2}\right)^m\left(\frac{1+e^{-t}}{2}\right)^{n-m}<\frac{1}{2^n}
\left(1-e^{-2t}\right)^{n/2}\left(\frac{1+e^{-t}}{1-e^{-t}}\right)^{n^{2/3}},
\end{multline*}
if $n$ is large enough then
\begin{equation}\label{eq:m in 1/2 range}
m\in \left(\frac12 n-n^{2/3},\frac12 n+n^{2/3}\right)\implies \left(\frac{1-e^{-t}}{2}\right)^m\left(\frac{1+e^{-t}}{2}\right)^{n-m}<\frac{\e}{2^{n+1}}.
\end{equation}
Moreover, because $t>0$ we have $h((1-e^{-t})/2)>\frac12$, where $h(s)\eqdef s^s(1-s)^{1-s}$ for $s\in [0,1]$. Noting that
\begin{multline*}
m\in \left(\frac{1-e^{-t}}{2}  n-n^{2/3},\frac{1-e^{-t}}{2}  n+n^{2/3}\right)\implies  \\
\left(\frac{1-e^{-t}}{2}\right)^m\left(\frac{1+e^{-t}}{2}\right)^{n-m}>
\left(h\left(\frac{1-e^{-t}}{2}\right)\right)^n
\left(\frac{1-e^{-t}}{1+e^{-t}}\right)^{n^{2/3}},
\end{multline*}
we see that if $n$ is large enough then
\begin{equation}\label{eq:m in other range}
m\in \left(\frac{1-e^{-t}}{2}  n-n^{2/3},\frac{1-e^{-t}}{2}  n+n^{2/3}\right)\implies \frac{1}{2^n}<\frac{\e}{2} \left(\frac{1-e^{-t}}{2}\right)^m\left(\frac{1+e^{-t}}{2}\right)^{n-m}.
\end{equation}
Consequently, if we choose $n$ so as to ensure the validity of~\eqref{eq:CLTeps1}, \eqref{eq:CLTeps2}, \eqref{eq:disjoint}, \eqref{eq:m in 1/2 range}, \eqref{eq:m in other range}, then recalling~\eqref{eq:norm of delta} we see that
\begin{equation*}
\left\|e^{-t\Delta}f_n\right\|_{L_p(\F_2^n,L_1(\F_2^n))}\ge 2\left(1-\frac{\e}{2}\right)^2\stackrel{\eqref{eq:norm of delta}}{\ge} (1-\e)\|f_n\|_{L_p(\F_2^n,L_1(\F_2^n))}.\tag*{\qedhere}
\end{equation*}
\end{proof}

\begin{proof}[Proof of Theorem~\ref{thm:K-convec char}]
Suppose that there exists $\delta\in (0,1)$, $k\in \N$, $p\in (1,\infty)$ and $t\in (0,\infty)$ such that
\begin{equation}\label{eq:<1-delta}
\forall\, n\in \N,\quad \left\|e^{-t\Delta}\right\|_{L_p^{\ge k}(\F_2^n,X)\to L_p^{\ge k}(\F_2^n,X)}<1-\delta.
\end{equation}

For $n\in \N$, identify $\F_2^{kn}$ with the $k$-fold product of $\F_2^n$. Define $F:\F_2^{kn}\to L_1(\F_2^{kn})$ by
\begin{equation}\label{eq:def F example}
F(x^1,\ldots,x^k)(y^1,\ldots,y^k)\eqdef \prod_{i=1}^k f_n(x^i)(y^i),
\end{equation}
where $f_n\in L_p^{\ge 1}(\F_2^n,L_1(\F_2^n))$ is given in~\eqref{eq:def normalized delta}. Then $F\in L_p^{\ge k}(\F_2^{kn},L_1(\F_2^{kn}))$. For every injective linear operator $T:L_1(\F_2^{kn})\to X$ we have $T\circ F\in L_p^{\ge k}(\F_2^{kn},X)$, and therefore
\begin{multline}\label{eq:arrow linit}
1-\delta\stackrel{\eqref{eq:<1-delta}}{>}\frac{\left\|e^{-t\Delta}(T\circ F)\right\|_{L_p(\F_2^{kn},X)}}{\|T\circ F\|_{L_p(\F_2^{kn},X)}}\ge \frac{1}{\|T\|\cdot \|T^{-1}\|}\cdot\frac{\left\|e^{-t\Delta}F\right\|_{L_p(\F_2^{kn},L_1(\F_2^{kn}))}}
{\|F\|_{L_p(\F_2^{kn},L_1(\F_2^{kn}))}}\\
\stackrel{\eqref{eq:def F example}}{=} \frac{1}{\|T\|\cdot \|T^{-1}\|}\cdot \left(\frac{\left\|e^{-t\Delta}f_n\right\|_{L_p(\F_2^n,L_1(\F_2^n))}}{\|f_n\|_{L_p(\F_2^n,L_1(\F_2^n))}}\right)^k
\xrightarrow[ n\to\infty]{} \frac{1}{\|T\|\cdot \|T^{-1}\|},
\end{multline}
where  in the last step of~\eqref{eq:arrow linit} we used Lemma~\ref{lem:normalized delta}.
It follows that
$$
\sup_{m\in \N}\inf_{S\in \mathscr{L}(\ell_1^m,X)}\|S\|\cdot\|S^{-1}\|\ge \frac{1}{1-\delta}.
$$
By Pisier's $K$-convexity theorem~\cite{Pisier-K-convex}  we conclude that $X$ must be $K$-convex.
\end{proof}

\subsection{Inverting the Laplacian on the vector-valued tail space}\label{sec:inverting}
Here we discuss lower bounds on the restriction of $\Delta$ to the tail space.  Such bounds can potentially yield a simplification of our construction of the base graph; see Remarks~\ref{rem:better laplace bound?} and~\ref{rem:vanilla KN} below.

\begin{theorem}\label{thm:lower laplace}
For every $K,p\in (1,\infty)$ there exist $\delta=\delta(K,p),c=c(K,p)\in (0,1)$ such that if $X$ is a $K$-convex Banach space with $K(X)\le K$ then for every $n\in \N$ and $k\in \{1,\ldots,n\}$,
\begin{equation}\label{eq:lower laplacian delta}
f\in L_p^{\ge k}\left(\F_2^n,X\right)\implies \left\|\Delta f\right\|_{L_p(\F_2^n,X)}\ge ck^\delta\cdot\|f\|_{L_p(\F_2^n,X)}.
\end{equation}
\end{theorem}

\begin{proof}
The estimate~\eqref{eq:lower laplacian delta} is deduced from Theorem~\ref{thm:AB} as follows. If $f\in L_p^{\ge k}\left(\F_2^n,X\right)$ then
\begin{multline*}
\|f\|_{L_p(\F_2^n,X)}=\left\|\int_0^\infty e^{-t\Delta}\Delta f dt\right\|_{L_p(\F_2^n,X)}\stackrel{\eqref{eq:AB}}{\le}
C\left(\int_0^1 e^{-Akt^B}dt+\int_1^\infty e^{-Akt}dt\right)\|\Delta f\|_{L_p(\F_2^n,X)}\\
\le C\left(\frac{\Gamma(1/B)}{(Ak)^{1/B}}+\frac{e^{-Ak}}{Ak}\right)\|\Delta f\|_{L_p(\F_2^n,X)}\lesssim \frac{CB}{A^{1/B}}\cdot \frac{1}{k^{1/B}} \|\Delta f\|_{L_p(\F_2^n,X)}.\tag*{\qedhere}
\end{multline*}
\end{proof}

We also have the following converse to Theorem~\ref{thm:lower laplace}.
\begin{theorem}
If $X$ is a Banach space such that for some $p,K\in (0,\infty)$ and $k\in \N$ we have \begin{equation}\label{eq:second K-convex char}
\lim_{n\to \infty} \inf_{\substack{f\in L_p^{\ge k}(\F_2^n,X)\\f\neq 0}}\frac{\|\Delta f\|_{L_p(\F_2^n,X)}}{\|f\|_{L_p(\F_2^n,X)}}>0,
\end{equation}
then $X$ is $K$-convex.
\end{theorem}

\begin{proof} For $f\in L_p(\F_2^n,X)$ define
$$
\Delta^{-1}f =\sum_{\substack{A\subseteq \{1,\ldots,n\}\\A\neq \emptyset}} \frac{1}{|A|}\hat{f}(A)W_A.
$$
In~\cite[Thm.~5]{NS02} it was shown that if $X$ is not $K$-convex then
$$
\sup_{n\in \N} \left\|\Delta^{-1}\right\|_{L_p(\F_2^n,X)\to L_p(\F_2^n,X)}=\infty.
$$
Here we need to extend this statement to the assertion contained in~\eqref{eq:goal lower inverse laplace} below, which should hold true for every Banach space $X$ that is not $K$-convex and every $k\in \N$.
\begin{equation}\label{eq:goal lower inverse laplace}
\sup_{n\in \N} \left\|\Delta^{-1}\right\|_{L_p^{\ge k}(\F_2^n,X)\to L_p^{\ge k}(\F_2^n,X)}=\infty.
\end{equation}

Arguing as in the proof of Theorem~\ref{thm:K-convec char}, by Pisier's $K$-convexity theorem~\cite{Pisier-K-convex} it will suffice to prove that for $n\ge 2$, if $F:\F_2^{kn}\to L_1(F_2^{kn})$ is given as in~\eqref{eq:def F example} then
\begin{equation}\label{eq:goal log}
\frac{\left\|\Delta^{-1}F\right\|_{L_p(\F_2^{kn},L_1(\F_2^{kn}))}}{\|F\|_{L_p(\F_2^{kn},L_1(\F_2^{kn}))}}\gtrsim \frac{\log n}{8^k}.
\end{equation}
Note that by~\eqref{eq:norm of delta},
\begin{equation}\label{eq:norm F}
\|F\|_{L_p(\F_2^{kn},L_1(\F_2^{kn}))}=2^k\left(1-\frac{1}{2^n}\right)^k\le 2^k.
\end{equation}
By~\eqref{eq:f_n riesz} and \eqref{eq:def F example}, for every $(x^1,\ldots,x^k),(y^1,\ldots,y^k)\in\F_2^{kn}$ and every $t\in (0,\infty)$,
\begin{multline}\label{eq:prod identity}
e^{-t\Delta}F(x^1,\ldots,x^k)(y^1,\ldots,y^k)
=\prod_{i=1}^k\left(\prod_{j=1}^n\left(1+e^{-t}(-1)^{x_j^i+y^i_j}\right)-1\right)\\
=\prod_{i=1}^k \left(\left(1-e^{-t}\right)^{\|x^i-y^i\|_1}\left(1+e^{-t}\right)^{n-\|x^i-y^i\|_1}-1\right).
\end{multline}
For every $x\in \F_2^n$ denote
$$
\Omega_x\eqdef \left\{y\in \F_2^n:\ \|y-x\|_1\le \frac{n}{2}\right\}.
$$
Then
\begin{equation}\label{eq:Omega size}
\forall\, x\in \F_2^n,\quad |\Omega_x|\ge 2^{n-1},
\end{equation}
 and by~\eqref{eq:prod identity} we have
\begin{equation}\label{eq:lower on good prod}
(y^1,\ldots,y^k)\in \prod_{i=1}^k\Omega_{x^i}\implies \left|e^{-t\Delta}F(x^1,\ldots,x^k)(y^1,\ldots,y^k)\right|
\ge \left(1-\left(1-e^{-2t}\right)^{n/2}\right)^k.
\end{equation}
Now,
\begin{align}\label{eq:use on good}
\nonumber\left\|\Delta^{-1}F(x^1,\ldots,x^k)\right\|_{L_1(\F_2^{kn})}&=\left\|\int_0^\infty
e^{-t\Delta}F(x^1,\ldots,x^k)dt\right\|_{L_1(\F_2^{kn})}\\\nonumber&\ge \frac{1}{2^{kn}}
\sum_{(y^1,\ldots,y^k)\in\prod_{i=1}^k\Omega_{x^i}}\left|\int_0^\infty
e^{-t\Delta}F(x^1,\ldots,x^k)(y^1,\ldots,y^k)dt\right|\\&\ge \frac{1}{2^k}\int_0^\infty \left(1-\left(1-e^{-2t}\right)^{n/2}\right)^kdt,
\end{align}
where in~\eqref{eq:use on good} we used~\eqref{eq:Omega size} and~\eqref{eq:lower on good prod}. Finally,
\begin{equation}\label{eq:log range}
\left\|\Delta^{-1}F\right\|_{L_p(\F_2^{kn},L_1(\F_2^{kn}))}\stackrel{\eqref{eq:use on good}}{\ge} \frac{1}{2^k}\int_0^{\frac12\log n}\left(1-\left(1-e^{-2t}\right)^{n/2}\right)^kdt\gtrsim \frac{\log n}{4^k}.
\end{equation}
The desired estimate~\eqref{eq:goal log} now follows from~\eqref{eq:norm F} and~\eqref{eq:log range}.
\end{proof}

\begin{remark}\label{rem:better laplace bound?}
The following natural problem presents itself. Can one improve Theorem~\ref{thm:lower laplace} so as to have $\delta=1$, i.e., to obtain the bound
\begin{equation}\label{eq:Laplace linear hope}
f\in L_p^{\ge k}\left(\F_2^n,X\right)\implies \left\|\Delta f\right\|_{L_p(\F_2^n,X)}\ge c(K,p)k\cdot\|f\|_{L_p(\F_2^n,X)}?
\end{equation}
As discussed in Remark~\ref{rem:open meyer}, this seems to be unknown even when $X=\R$. If~\eqref{eq:Laplace linear hope} were true then it would significantly simplify our construction of the base graph, since in Section~\ref{sec:base} we would be able to use the ``vanilla" hypercube quotients of~\cite{KN06} instead of the quotients of the discretized heat semigroup as in Lemma~\ref{lem:base code}; see Remark~\ref{rem:vanilla KN} below for more information on this potential simplification.
\end{remark}

\section{Nonlinear spectral gaps in uniformly convex normed
spaces}\label{sec:UC}

Let $(X,\|\cdot\|_X)$ be a normed space. For $n\in \N$ and $p\in
[1,\infty)$ we let $L_p^n(X)$ denote the space of functions
$f:\{1,\ldots,n\}\to X$, equipped with the norm
$$
\|f\|_{L_p^n(X)} =\left(\frac{1}{n}\sum_{i=1}^n
\|f(i)\|_X^p\right)^{1/p}.
$$
Thus, using the notation introduced in the beginning of Section~\ref{sec:heat},
$L_p^n(X)=L_p^n(\{1,\ldots,n\},X)$. We shall also use the notation
$$
L_p^n(X)_0\eqdef \left\{f\in L_p^n(X):\ \sum_{i=1}^n f(i)=0\right\}.
$$

Given an $n\times n$ symmetric stochastic matrix $A=(a_{ij})$ we
denote by $A\otimes I^n_X$ the operator from $L_p^n(X)$ to
$L_p^n(X)$ given by
$$
\left(A\otimes I^n_X\right)f(i)=\sum_{j=1}^n a_{ij}f(j).
$$
Note that since $A$ is symmetric and stochastic the operator
$A\otimes I^n_X$ preserves the subspace $L_p^n(X)_0$, that is
$\left(A\otimes I^n_X\right)(L_p^n(X)_0)\subseteq L_p^n(X)_0$. Define
\begin{equation}\label{eq:def lambdap}
\lambda^{(p)}_X(A)\eqdef \left\|A\otimes I^n_X\right\|_{L_p^n(X)_0\to
L_p^n(X)_0}.
\end{equation}
Note that, since $A$ is doubly stochastic, $\lambda_X^p(A)\le 1$. It
is immediate to check that
$$
\lambda_{\R}^{(2)}(A)=\lambda_{L_2}^{(2)}(A)=\lambda(A)=\max_{i\in
\{2,\ldots,n\}}|\lambda_i(A)|.
$$
Thus $\lambda^{(p)}_X(A)$ should be viewed as a non-Euclidean (though still linear) variant of
the absolute spectral gap of $A$. The following lemma substantiates
this analogy by establishing a relation between $\lambda_X^{(p)}(A)$
and $\bp(A,\|\cdot\|_X^p)$.
\begin{lemma}\label{lem:norm to poincare}
For every normed space $(X,\|\cdot\|_X)$, every $p\ge 1$ and every
$n\times n$ symmetric stochastic matrix $A$, we have
\begin{eqnarray}\label{eq:linear to nonlinear}
\bp(A,\|\cdot\|_X^p)\le
\left(1+\frac{4}{1-\lambda_X^{(p)}(A)}\right)^p.
\end{eqnarray}
\end{lemma}

\begin{proof}
Write $\lambda =\lambda_X^{(p)}(A)$. We may assume that $\lambda<1$,
since otherwise there is nothing to prove. Fix
$f,g:\{1,\ldots,n\}\to X$ and denote
$$
\overline f\eqdef \frac{1}{n}\sum_{i=1}^n f(i)\quad
\mathrm{and}\quad  \overline g\eqdef \frac{1}{n}\sum_{i=1}^n g(i).
$$
Thus
$$
f_0\eqdef f-\overline f\in L_p^n(X)_0\quad \mathrm{and}\quad
g_0\eqdef g-\overline g\in L_p^n(X)_0. $$
Therefore
\begin{equation}\label{eq:use norm bounds}
\left\|(A\otimes I_X^n)f_0\right\|_{L_p^n(X)}\le \lambda
\left\|f_0\right\|_{L_p^n(X)}\quad \mathrm{and}\quad
\left\|(A\otimes I_X^n)g_0\right\|_{L_p^n(X)}\le \lambda
\|g_0\|_{L_p^n(X)}.
\end{equation}
Let $B$ be the $(2n)\times (2n)$ symmetric stochastic matrix given
by
\begin{equation}\label{eq:defB}
 B=   \begin{pmatrix}
  0 & A \\
   A & 0
   \end{pmatrix}.
   \end{equation}
Letting $h=f_0\oplus g_0\in L_p^{2n}(X)$ be given by
$$
h(i)\eqdef \left\{\begin{array}{ll} f_0(i) &\mathrm{if\ } i\in
\{1,\ldots,n\}, \\ g_0(i-n) &\mathrm{if\ } i\in \{n+1,\ldots,2n\},
\end{array}\right.
$$
 we
see that
\begin{align}\label{eq:subtract identity}
&\!\!\!\!\!\!\!\!\!\!\!\!\!\!\!(1-\lambda)\left\|h\right\|_{L_p^{2n}(X)}
=\|h\|_{L_p^{2n}(X)}-\left(\frac{\lambda^p\left\|
f_0\right\|_{L_p^n(X)}^p+  \lambda^p\left\|
g_0\right\|_{L_p^n(X)}^p}{2}\right)^{1/p}\nonumber\\&\stackrel{\eqref{eq:use norm bounds}}\le
\|h\|_{L_p^{2n}(X)}-\left(\frac12\left\|(A\otimes
I_X^n)f_0\right\|_{L_p^n(X)}^p+ \frac12\left\|(A\otimes
I_X^n)g_0\right\|_{L_p^n(X)}^p\right)^{1/p}\nonumber\\
&\stackrel{\eqref{eq:defB}}{=} \|h\|_{L_p^{2n}(X)}-\left\|(B\otimes
I_X^{2n})h\right\|_{L_p^{2n}(X)}\nonumber\\
&\le \left\|\left(I_{L_p^{2n}(X)}-B\otimes
I_X^{2n}\right)h\right\|_{L_p^{2n}(X)}\nonumber\\
&= \left(\frac{1}{2n}\sum_{i=1}^n\left\|\sum_{j=1}^n
a_{ij}\left(f_0(i)-g_0(j)\right)\right\|_X^p+\frac{1}{2n}\sum_{i=1}^n\left\|\sum_{j=1}^n
a_{ij}\left(g_0(i)-f_0(j)\right)\right\|_X^p\right)^{1/p}\nonumber\\
&\le  \left(\frac{1}{n}\sum_{i=1}^n\sum_{j=1}^n
a_{ij}\|f_0(i)-g_0(j)\|_X^p\right)^{1/p}\nonumber\\
&\le \left\|\overline f-\overline
g\right\|_X+\left(\frac{1}{n}\sum_{i=1}^n\sum_{j=1}^n
a_{ij}\|f(i)-g(j)\|_X^p\right)^{1/p}.
\end{align}
Note that
\begin{multline}\label{eq:distance between expectations}
\left\|\overline f-\overline
g\right\|_X=\left\|\frac{1}{n}\sum_{i=1}^n\sum_{j=1}^n
a_{ij}(f(i)-g(j))\right\|_X\\\le \frac{1}{n}\sum_{i=1}^n\sum_{j=1}^n
a_{ij} \left\|f(i)-g(j)\right\|_X\le
\left(\frac{1}{n}\sum_{i=1}^n\sum_{j=1}^n
a_{ij}\|f(i)-g(j)\|_X^p\right)^{1/p}.
\end{multline}
Combining~\eqref{eq:subtract identity} and~\eqref{eq:distance
between expectations} we see that
\begin{equation}\label{eq:combined}
\left(\frac{1}{2n}\sum_{i=1}^n\left(\|f_0(i)\|_X^p+\|g_0(i)\|_{X}^p\right)\right)^{1/p}\le
\frac{2}{1-\lambda}\left(\frac{1}{n}\sum_{i=1}^n\sum_{j=1}^n
a_{ij}\|f(i)-g(j)\|_X^p\right)^{1/p}. \end{equation} But,
\begin{eqnarray*}
&&\!\!\!\!\!\!\!\!\!\!\!\!\!\!\!\!\!\!\!\!\!\!\!\!\!\!\!\!\!\!\!\!\!\!\!\!\!\!\left(\frac{1}{n^2}\sum_{i=1}^n\sum_{j=1}^n
\|f(i)-g(j)\|_X^p\right)^{1/p}\le \left\|\overline f-\overline
g\right\|_X+ \left(\frac{1}{n^2}\sum_{i=1}^n\sum_{j=1}^n
\|f_0(i)-g_0(j)\|_X^p\right)^{1/p}\\&\le& \left\|\overline
f-\overline g\right\|_X+
\left(\frac{1}{n}\sum_{i=1}^n2^{p-1}\left(\|f_0(i)\|_X^p+\|g_0(i)\|_{X}^p\right)\right)^{1/p}\\
&\stackrel{\eqref{eq:distance between
expectations}\wedge\eqref{eq:combined}}{\le}&
\left(1+\frac{4}{1-\lambda}\right)\left(\frac{1}{n}\sum_{i=1}^n\sum_{j=1}^n
a_{ij}\|f(i)-g(j)\|_X^p\right)^{1/p},
\end{eqnarray*}
which implies the desired estimate~\eqref{eq:linear to nonlinear}.
\end{proof}

\subsection{Norm bounds need not imply nonlinear spectral gaps}\label{sec:no converse in L_1}

One cannot bound $\bp(A,\|\cdot\|_X^p)$ in terms of $\lambda_X^{(p)}(A)$
for a  general Banach space $X$, as shown in the following example.

\begin{lemma}\label{lem:gamma+ but no norm}
For every $n\in \N$ there exists a $2^n\times 2^n$ symmetric stochastic matrix $A_n$ such that for every $p\in [1,\infty)$,
\begin{equation}\label{eq:heat has gamma+}
\sup_{n\in \N} \gamma_+\left(A_n,\|\cdot\|_{L_1}^p\right)<\infty,
\end{equation}
yet
\begin{equation}\label{eq:no norm bound}
\lim_{n\to \infty} \lambda_{L_1}^{(p)}(A_n)=1.
\end{equation}
\end{lemma}

\begin{proof}
We use here the results and notation of Section~\ref{sec:heat}. For every $t\in (0,\infty)$, the operator $e^{-t\Delta}$ is an averaging operator, since by~\eqref{eq:riesz} it corresponds to convolution with the Riesz kernel given in~\eqref{eq:def:riesz}. Hence the $\F_2^n\times \F_2^n$ matrix $A_n$ whose entry at $(x,y)\in \F_2^n\times \F_2^n$ is
$$
(e^{-t\Delta}\delta_x)(y)\stackrel{\eqref{eq:primal semigroup}}{=}\left(\frac{1-e^{-t}}{2}\right)^{\|x-y\|_1}\left(\frac{1+e^{-t}}{2}\right)^{n-\|x-y\|_1}
$$
is symmetric and stochastic. Lemma~\ref{lem:normalized delta} implies the validity of~\eqref{eq:no norm bound}, so it remains to establish~\eqref{eq:heat has gamma+}.

By Lemma~\ref{lem:meyer} there exists $c_p\in (0,\infty)$ such that
$$
\lambda_{L_{2p}}^{(2p)}(A_n)\le e^{-c_p\min\{t,t^2\}}.
$$
It therefore follows from Lemma~\ref{lem:norm to poincare} that
$$
\gamma_+\left(A_n,\|\cdot\|_{L_{2p}}^{2p}\right)\le \left(\frac{5-e^{-c_p\min\{t,t^2\}}}{1-e^{-c_p\min\{t,t^2\}}}\right)^p\eqdef C_p(t)<\infty.
$$
Since $L_2$ embeds isometrically into $L_{2p}$ (see e.g.~\cite{Woj91}), it follows that $\gamma_+\left(A_n,\|\cdot\|_{L_{2}}^{2p}\right)\le C_p(t)$. It is a standard fact that $L_1$ equipped with the metric $d(f,g)=\sqrt{\|f-g\|_1}$ admits an isometric embedding into $L_2$ (for one of several possible simple proofs of this, see~\cite[Sec.~3]{Nao10}). It follows that $\gamma_+\left(A_n,\|\cdot\|_{L_{1}}^{p}\right)=\gamma_+\left(A_n,d^{2p}\right)\le C_p(t)$.
\end{proof}

\subsection{A partial converse to Lemma~\ref{lem:norm to poincare} in uniformly convex spaces}\label{sec:converese} Despite the validity of Lemma~\ref{lem:gamma+ but no norm}, Lemma~\ref{poincare to norm in uniformly convex} below is a partial converse to Lemma~\ref{lem:norm to poincare} that holds true if $X$ is uniformly convex. We start this section with a review of uniform convexity and smoothness; the material below will also be used in Section~\ref{sec:martingale}.

Let $(X,\|\cdot\|_X)$ be a normed space. The {\em modulus of uniform
convexity\/} of $X$ is defined for $\e\in [0,2]$ as
\begin{eqnarray}\label{def:convexity}
\delta_X(\e)\eqdef\inf\left\{ 1-\frac{\|x+y\|_X}{2}:\ x,y\in X,\
\|x\|_X=\|y\|_X=1,\ \|x-y\|_X=\e\right\}.
\end{eqnarray}
$X$ is said to be {\em uniformly convex} if $\delta_X(\e)>0$ for all
$\e\in (0,2]$. Furthermore, $X$ is said to have modulus of convexity
of power type $p$ if there exists a constant $c\in (0,\infty)$ such
that $\delta_X(\e)\ge c\,\e^p$ for all $\e\in [0,2]$. It is
straightforward to check that in this case necessarily $p\ge 2$. By
Proposition 7 in~\cite{BCL} (see also~\cite{Fiegel76}), $X$ has
modulus of convexity of power type $p$ if and only if there exists a
constant $K\in [1,\infty)$ such that for every $x,y\in X$
\begin{eqnarray}\label{eq:two point convex}
\|x\|_X^p+\frac{1}{K^p}\|y\|_X^p\le
\frac{\|x+y\|_X^p+\|x-y\|_X^p}{2}.
\end{eqnarray}
The infimum over those $K$ for which \eqref{eq:two point convex}
holds is called the $p$-convexity constant of $X$, and is denoted
$K_p(X)$.

The {\em modulus of uniform smoothness} of $X$ is defined for
$\tau\in (0,\infty)$ as
\begin{equation}\label{eq:def smoothness}
\rho_X(\tau)\eqdef \sup\left\{\frac{\|x+\tau y\|_X+\|x-\tau
y\|_X}{2}-1:\ x,y\in X,\ \|x\|_X=\|y\|_X=1\right\}.
\end{equation}
$X$ is said to be {\em uniformly smooth} if $\lim_{\tau\to
0}\rho_X(\tau)/\tau=0$. Furthermore, $X$ is said to have modulus of
smoothness of power type $p$ if there exists a constant $C\in
(0,\infty)$ such that $\rho_X(\tau)\le C\tau^p$ for all $\tau\in
(0,\infty)$. It is straightforward to check that in this case
necessarily $p\in [1,2]$. It follows from~\cite{BCL} that $X$ has
modulus of smoothness of power type $p$ if and only if there exists
a constant $S\in [1,\infty)$ such that for every $x,y\in X$
\begin{equation}\label{eq:two point smooth}
\frac{\|x+y\|_X^p+\|x-y\|_X^p}{2}\le \|x\|_X^p+S^p\|y\|_X^p.
\end{equation}
The infimum over those $S$ for which \eqref{eq:two point smooth}
holds is called the $p$-smoothness constant of $X$, and is denoted
$S_p(X)$.

The moduli appearing in~\eqref{def:convexity} and~\eqref{eq:def
smoothness} relate to each other via the following classical duality
formula of Lindenstrauss~\cite{Lin63}.
$$
\rho_{X^*}(\tau)=\sup\left\{\frac{\tau\e}{2}-\delta_X(\e):\ \e\in
[0,2]\right\}.
$$
Correspondingly, it was shown in~\cite[Lem.~5]{BCL} that the  best
constants in~\eqref{eq:two point convex} and~\eqref{eq:two point
smooth} have the following duality relation.
\begin{equation}\label{eq:duality KS}
K_p(X)=S_{p/(p-1)}(X^*).
\end{equation}

Observe that if $q\ge p$ then for all $x,y\in X$ we have
$$\left(\frac{\|x+y\|_X^p+\|x-y\|_X^p}{2}\right)^{1/p}\le
\left(\frac{\|x+y\|_X^q+\|x-y\|_X^q}{2}\right)^{1/q},
$$
and
$$
\left(\|x\|_X^q+\frac{1}{K^q}\|y\|_X^q\right)^{1/q}\le
\left(\|x\|_X^p+\frac{1}{K^p}\|y\|_X^p\right)^{1/p}.
$$
Hence,
\begin{equation}\label{eq:monotonicity K}
q\ge p\implies K_q(X)\le K_p(X).
\end{equation}
Similarly we have (though we will not use this fact later),
$$
q\le p\implies S_q(X)\le S_p(X).
$$

The following lemma can be deduced from a combination of results
in~\cite{FP74,Fiegel76} and~\cite{BCL} (without the explicit
dependence on $p,q$). A simple proof of the case $p=2$ of it is also
contained in~\cite{Nao11}; we include the natural adaptation of the
argument  to general $p\in (1,2]$ for the sake of
completeness.

\begin{lemma}\label{lem:L_q smoothness} For every $p\in (1,2]$,
$q\in [p,\infty)$, every Banach space $(X,\|\cdot\|_X$ and every measure space $(\Omega,\mu)$, we have
$$
S_p\left(L_q(\mu,X)\right)\le (5pq)^{1/p}S_p(X).
$$
\end{lemma}
\begin{proof} Fix $S>S_p(X)$. We will show that for every $x,y\in X$ we have
\begin{equation}\label{eq:AR}
\frac{\|x+y\|_X^q+\|x-y\|_X^q}{2}\le
\left(\|x\|_X^p+5pqS^p\|y\|_X^p\right)^{q/p}.
\end{equation}

Assuming the validity of~\eqref{eq:AR} for the moment, we complete
the proof of Lemma~\ref{lem:L_q smoothness} as follows. If $f,g\in
L_q(\mu,X)$ then
\begin{eqnarray*}
\frac{\|f+g\|_{L_q(\mu,X)}^p+\|f-g\|_{L_q(\mu,X)}^p}{2}&\le&
\left(\frac{\|f+g\|_{L_q(\mu,X)}^q+\|f-g\|_{L_q(\mu,X)}^q}{2}\right)^{p/q}\\
&=&\left(\int_\Omega\frac{\|f+g\|_X^q+\|f-g\|_X^q}{2}d\mu\right)^{p/q}\\
&\stackrel{\eqref{eq:AR}}{\le}&
\Big\|\|f\|_X^p+5pqS^p\|g\|_X^p\Big\|_{L_{q/p}(\mu)}\\
&\le& \|f\|_{L_q(\mu,X)}^p+5pqS^p\|g\|_{L_q(\mu,X)}^p.
\end{eqnarray*}
This proves that $S_p(L_q(\mu,X))^p\le 5pqS_p(X)^p$, as desired.

It remains to prove~\eqref{eq:AR}. Since  $\|y\|_X^p+5pqS^p\|x\|_X^p \le
\|x\|_X^p+5pqS^p\|y\|_X^p$ if $\|x\|_X\le \|y\|_X$, it suffices to
prove~\eqref{eq:AR} under the additional assumption $\|y\|_X\le
\|x\|_X$. After normalization we may further assume that $\|x\|_X=1$
and $\|y\|_X\le 1$.

Note that
\begin{equation}\label{eq:y triangle}
\big|\|x+y\|_X^p-\|x-y\|_X^p\big|\le
\left(1+\|y\|_X\right)^p-\left(1-\|y\|_X\right)^p\le 2p\|y\|_X.
\end{equation}
We claim that for every $\alpha\in [1,\infty)$ and $\beta\in [-1,1]$
we have
\begin{equation}\label{eq:norm quadratic}
\left(\frac{(1+\beta)^\alpha+(1-\beta)^\alpha}{2}\right)^{1/\alpha}\le
1+ 2\alpha\beta^2.
\end{equation}
Indeed, by symmetry it suffices to prove~\eqref{eq:norm quadratic}
when $\beta\in [0,1]$. The left hand side of~\eqref{eq:norm
quadratic} is at most $\max\{1+\beta,1-\beta\}=1+\beta$, which
implies~\eqref{eq:norm quadratic} when $\beta\ge 1/(2\alpha)$. We
may therefore assume  that $\beta\in [0,1/(2\alpha)]$, in which case
the crude bound $(1+\beta)^\alpha+(1-\beta)^\alpha\le 2+4\alpha^2\beta^2$ follows
from Taylor's expansion, implying~\eqref{eq:norm quadratic} in this
case as well.

Set
\begin{equation}\label{eq:b beta}
b\eqdef \frac{\|x+y\|_X^p+\|x-y\|_X^p}{2}\quad \mathrm{and}\quad
\beta\eqdef \frac{\|x+y\|_X^p-\|x-y\|_X^p}{\|x+y\|_X^p+\|x-y\|_X^p},
\end{equation}
and define
\begin{equation}\label{eq:def theta beta}
\theta\eqdef
\left(\frac{(1+\beta)^{q/p}+(1-\beta)^{q/p}}{2}\right)^{p/q}-1\in
[0,1].
\end{equation}
Observe that by convexity $b\ge 1$, and therefore
\begin{equation}\label{eq:bound thetha upper}
\theta\stackrel{\eqref{eq:norm quadratic}}{\le}
2\frac{q}{p}\beta^2\stackrel{\eqref{eq:y triangle}\wedge \eqref{eq:b beta}}{\le}
\frac{2q}{p}\left(\frac{2p\|y\|_X}{2b}\right)^2\le 2pq\|y\|_X^p,
\end{equation}
where we used the fact that $p\in [1,2]$ and $\|y\|_X\le 1$. Now,
\begin{multline*}
\frac{\|x+y\|_X^q+\|x-y\|_X^q}{2} \stackrel{\eqref{eq:b beta}\wedge
\eqref{eq:def theta
beta}}{=}(b(1+\theta))^{q/p}\\\stackrel{\eqref{eq:two point
smooth}}{\le}
\left(\left(1+S^p\|y\|_X^p\right)(1+\theta)\right)^{q/p}\stackrel{\eqref{eq:bound
thetha upper}}{\le} \left(1+5pqS^p\|y\|_X^p\right)^{q/p}.\tag*{\qedhere}
\end{multline*}
\end{proof}

By~\eqref{eq:duality KS}, Lemma~\ref{lem:L_q smoothness} implies the
following dual statement.
\begin{corollary}\label{cor:L_p(X)}
For every $p\in [2,\infty)$, $q\in (1,p]$, every Banach space $(X,\|\cdot\|_X)$ and every measure space
$(\Omega,\mu)$, we have
$$
K_p(L_q(\mu,X))\le \left(\frac{5pq}{(p-1)(q-1)}\right)^{1-1/p}K_p(X).
$$
\end{corollary}

The following lemma is stated and proved in~\cite{Ball} when $p=2$.

\begin{lemma}\label{lem:from Ball}
Let $X$ be a normed space and $U$ a random vector in $X$ with $\E
\left[\|U\|_X^p\right]<\infty$. Then
$$
\left\|\E
\left[U\right]\right\|_X^p+\frac{1}{(2^{p-1}-1)K_p(X)^p}\E\left[\left\|U-\E[
U]\right\|_X^p\right]\le \E\left[\|U\|_X^p\right].
$$
\end{lemma}

\begin{proof}
We repeat here the $p>2$ variant of the argument from~\cite{Ball}
for the sake of completeness. Let $(\Omega,\Pr)$ be the probability
space on which $U$ is defined. Denote
\begin{equation}\label{eq:def theta}
\theta\eqdef\inf\left\{\frac{\E \left[\|V\|_X^p\right]-\left\|\E
[V]\right\|_X^p}{\E\left[\left\|V-\E [V]\right\|_X^p\right]}:\ V\in
L_p(\Omega,X)\ \wedge\ \E\left[\|V-\E [V]\|_X^p\right]>0\right\}.
\end{equation}
Then $\theta\ge 0$. Our goal is to show that
\begin{equation}\label{eq:goal lower theta}
\theta\ge \frac{1}{(2^{p-1}-1)K_p(X)^p}.
\end{equation}
Fix $\phi>\theta$. Then there exists a random vector $V_0\in
L_p(\Omega,X)$ for which
\begin{equation}\label{eq:reverse theta}
\phi \E\left[\|V_0-\E [V_0]\|_X^p\right]> \E
\left[\|V_0\|_X^p\right]-\|\E [V_0]\|_X^p.
\end{equation}
Fix $K>K_p(X)$. Apply the inequality~\eqref{eq:two point convex}
 to the vectors $$ x=\frac12 V_0+\frac12 \E [V_0]\quad
\mathrm{and} \quad y=\frac12 V_0-\frac12 \E [V_0], $$ to get the
point-wise estimate
\begin{equation}\label{eq:point-wise}
2\left\|\frac12 V_0+\frac12 \E
[V_0]\right\|_X^p+\frac{2}{K^p}\left\|\frac12 V_0-\frac12 \E
[V_0]\right\|_X^p\le \|V_0\|_X^p+\|\E [V_0]\|_X^p.
\end{equation}
Hence
\begin{eqnarray*}
&&\!\!\!\!\!\!\!\!\!\!\!\!\!\!\!\!\!\phi \E\left[\|V_0-\E[
V_0]\|_X^p\right]\stackrel{\eqref{eq:reverse theta}}{>} \E\left[
\|V_0\|_X^p\right]-\|\E [V_0]\|_X^p\\
&\stackrel{\eqref{eq:point-wise}}{\ge}&
2\left(\E\left[\left\|\frac12 V_0+\frac12 \E[
V_0]\right\|_X^p\right]-\left\|\E \left[\frac12V_0+\frac12\E[
V_0]\right]\right\|_X^p\right)+ \frac{2}{K^p}\E\left[\left\|\frac12
V_0-\frac12 \E [V_0]\right\|_X^p\right]\\
&\stackrel{\eqref{eq:def
theta}}{\ge}&2\theta\E\left[\left\|\left(\frac12V_0+\frac12\E[
V_0]\right)-\E \left[\frac12V_0+\frac12\E[
V_0]\right]\right\|_X^p\right]+\frac{2}{K^p}\E\left[\left\|\frac12
V_0-\frac12 \E[
V_0]\right\|_X^p\right]\\
&=&
\left(\frac{\theta}{2^{p-1}}+\frac{1}{2^{p-1}K^p}\right)\E\left[\|V_0-\E
[V_0]\|_X^p\right].
\end{eqnarray*}
Thus
\begin{equation}\label{eq:phi}
\phi\ge \frac{\theta}{2^{p-1}}+\frac{1}{2^{p-1}K^p}.
\end{equation}
Since~\eqref{eq:phi} holds for all $\phi>\theta$ and $K>K_p(X)$, the
desired lower bound~\eqref{eq:goal lower theta} follows.
\end{proof}

\begin{lemma}\label{poincare to norm in uniformly convex}
Fix $p\in [2,\infty)$ and let $X$ be a normed space with
$K_p(X)<\infty$. Then for every $n\times n$ symmetric stochastic
matrix $A=(a_{ij})$ we have
$$
\lambda_X^{(p)}(A)\le
\left(1-\frac{1}{(2^{p-1}-1)K_p(X)^p\bp\left(A,\|\cdot\|_X^p\right)}\right)^{1/p}.
$$
\end{lemma}

\begin{proof}
Fix $\bp>\bp\left(A,\|\cdot\|_X^p\right)$ and $f\in L_p^n(X)_0$. For
every $i\in \{1,\ldots,n\}$ consider the random vector $U_i\in X$
given by
$$
\Pr\left[U_i=f(j)\right]=a_{ij}.
$$
Lemma~\ref{lem:from Ball} implies that
\begin{equation}\label{eq:use ball}
\left\|\sum_{j=1}^n a_{ij}f(j)\right\|_X^p\le \sum_{j=1}^n
a_{ij}\|f(j)\|_X^p-\frac{1}{(2^{p-1}-1)K_p(X)^p}\sum_{j=1}^n
a_{ij}\left\|f(j)-\sum_{k=1}^n a_{ik}f(k)\right\|_X^p.
\end{equation}
Define for $i\in \{1,\ldots,n\}$,
$$
g(i)=\E [U_i]=\sum_{k=1}^n a_{ik}f(k).
$$
By averaging~\eqref{eq:use ball} over $i\in \{1,\ldots,n\}$ we see
that
\begin{eqnarray}\label{eq:sum i}
\left\|(A\otimes
I_X^n)f\right\|_{L_p^n(X)}^p&=&\frac{1}{n}\sum_{i=1}^n\left\|\sum_{j=1}^n
a_{ij}f(j)\right\|_X^p\nonumber\\&\le&\nonumber
\frac{1}{n}\sum_{i=1}^n\sum_{j=1}^n
a_{ij}\|f(j)\|_X^p-\frac{1}{n(2^{p-1}-1)K_p(X)^p}\sum_{i=1}^n\sum_{j=1}^n
a_{ij}\left\|f(j)-g(i)\right\|_X^p\\
&=&
\|f\|_{L_p^n(X)}^p-\frac{1}{n(2^{p-1}-1)K_p(X)^p}\sum_{i=1}^n\sum_{j=1}^n
a_{ij}\left\|f(j)-g(i)\right\|_X^p.
\end{eqnarray}
The definition of $\bp\left(A,\|\cdot\|_X^p\right)$ implies that
\begin{multline}\label{eq:use def bp}
\frac{1}{n}\sum_{i=1}^n\sum_{j=1}^n
a_{ij}\left\|f(j)-g(i)\right\|_X^p\ge \frac{1}{\bp
n^2}\sum_{i=1}^n\sum_{j=1}^n \left\|f(j)-g(i)\right\|_X^p\\\ge
\frac{1}{\bp n}\sum_{j=1}^n
\left\|f(j)-\frac{1}{n}\sum_{i=1}^ng(i)\right\|_X^p=\frac{1}{\bp
n}\sum_{j=1}^n
\left\|f(j)\right\|_X^p=\frac{1}{\bp}\|f\|_{L_p^n(X)}^p,
\end{multline}
where we used the fact that since $f\in L_p^n(X)_0$ we have
$$\sum_{i=1}^ng(i)=\sum_{k=1}^n
\left(\sum_{i=1}^na_{ik}\right)f(k)=\sum_{k=1}^n f(k)=0.$$
Substituting~\eqref{eq:use def bp} into~\eqref{eq:sum i} yields the
bound
\begin{equation}\label{eq:the norm bound}
\left\|(A\otimes I_X^n)f\right\|_{L_p^n(X)}^p\le
\left(1-\frac{1}{(2^{p-1}-1)K_p(X)^p\bp}\right)\|f\|_{L_p^n(X)}^p.
\end{equation}
Since~\eqref{eq:the norm bound} holds for every $f\in L_p^n(X)_0$
and $\bp>\bp\left(A,\|\cdot \|_X^p\right)$, inequality~\eqref{eq:the
norm bound} implies the required bound on
$\lambda_X^{(p)}(A)=\left\|A\otimes I^n_X\right\|_{L_p^n(X)_0\to
L_p^n(X)_0}$.
\end{proof}

\begin{theorem}\label{decay in uniformly convex}
Fix $p\in [2,\infty)$ and $t\in \N$. Let $X$ be a normed space with
$K_p(X)<\infty$. Then for every $n\times n$ symmetric stochastic
matrix $A=(a_{ij})$ we have
$$
\bp\left(A^t,\|\cdot\|_X^p\right)\le
\left[4K_p(X)\right]^{p^2}\cdot\max\left\{1,\left(\frac{\bp\left(A,\|\cdot\|_X^p\right)}{t}\right)^p\right\}.
$$
\end{theorem}

\begin{proof}
Note that since $A\otimes I_X^n$ preserves $L_p^n(X)_0$  we have
\begin{multline}\label{eq:power spectrum}
\lambda_X^{(p)}(A^t)=\left\|A^t\otimes I^n_X\right\|_{L_p^n(X)_0\to
L_p^n(X)_0}=\left\|(A\otimes I^n_X)^t\right\|_{L_p^n(X)_0\to
L_p^n(X)_0}\\\le \left\|A\otimes I^n_X\right\|_{L_p^n(X)_0\to
L_p^n(X)_0}^t= \lambda_X^{(p)}(A)^t.
\end{multline}
Lemma~\ref{lem:norm to poincare} applied to the matrix $A^t$, in
combination with~\eqref{eq:power spectrum}, yields the bound
\begin{equation}\label{eq:use for power}
\bp\left(A^t,\|\cdot\|_X^p\right)\le
\left(\frac{5-\lambda_X^{(p)}(A)^t}{1-\lambda_X^{(p)}(A)^t}\right)^p\le
\left(\frac{5}{1-\lambda_X^{(p)}(A)^t}\right)^p.
\end{equation}
On the other hand, using Lemma~\ref{poincare to norm in uniformly
convex} we have
\begin{multline}\label{eq:exp}
\lambda_X^{(p)}(A)\le
\left(1-\frac{1}{(2^{p-1}-1)K_p(X)^p\bp\left(A,\|\cdot\|_X^p\right)}\right)^{1/p}\\\le
\exp\left(-\frac{1}{p(2^{p-1}-1)K_p(X)^p\bp\left(A,\|\cdot\|_X^p\right)}\right).
\end{multline}
Thus
\begin{multline}\label{pass to min}
1-\lambda_X^{(p)}(A)^t\stackrel{\eqref{eq:exp}}{\ge}
1-\exp\left(-\frac{t}{p(2^{p-1}-1)K_p(X)^p\bp\left(A,\|\cdot\|_X^p\right)}\right)\\\ge
\frac12
\min\left\{1,\frac{t}{p(2^{p-1}-1)K_p(X)^p\bp\left(A,\|\cdot\|_X^p\right)}\right\}.
\end{multline}
The required result is now a combination of~\eqref{pass to min}
and~\eqref{eq:use for power}.
\end{proof}

\subsection{Martingale inequalities and metric Markov cotype}\label{sec:martingale}
Let $X$ be a Banach space with $K_p(X)<\infty$. Assume that
$\{M_k\}_{k=0}^n\subseteq X$ is a martingale with respect to the
filtration $\mathcal{F}_0\subseteq \mathcal{F}_1\subseteq
\cdots\subseteq \mathcal{F}_{n-1}$, i.e., $\E\left[M_{i+1}
|\mathcal{F}_{i}\right]=M_i$ for every $i\in\{0,1,\dots,n-1\}$.
Lemma~\ref{lem:from Ball} implies that
\begin{eqnarray}\label{eq:pisier iteration}
&&\!\!\!\!\!\!\!\!\!\!\!\!\!\!\!\!\!\!\!\!\!\!\E\left[\left\|M_n-M_0\right\|_X^p\Big|\mathcal{F}_{n-1}\right]
\nonumber\ge \left\|\E\left[
M_n-M_0\Big|\mathcal{F}_{n-1}\right]\right\|_X^p\\&\phantom{\le}&\quad+\frac{1}{(2^{p-1}-1)K_p(X)^p}
\E\left[\left\|M_n-M_0-\E\left[\left.M_n-M_0\Big|\mathcal{F}_{n-1}\right]\right\|_X^p\right|\mathcal{F}_{n-1}\right]\nonumber\\
&=&\left\|M_{n-1}-M_0\right\|_X^p+\frac{1}{(2^{p-1}-1)K_p(X)^p}\E\left[\left\|M_n-M_{n-1}\right\|_X^p\Big|\mathcal{F}_{n-1}\right].
\end{eqnarray}
Taking expectation in~\eqref{eq:pisier iteration} yields the
estimate
$$
\E\left[\left\|M_n-M_0\right\|_X^p\right]\ge
\E\left[\left\|M_{n-1}-M_{0}\right\|_X^p\right]+\frac{1}{(2^{p-1}-1)K_p(X)^p}\E\left[\left\|M_n-M_{n-1}\right\|_X^p\right].
$$
Iterating this argument we obtain the following famous inequality of
Pisier~\cite{Pisier-martingales}, which will be used crucially in
what follows.

\begin{theorem}[Pisier's martingale inequality]\label{thm:pisier ineq} Let $X$ be a
Banach space with $K_p(X)<\infty$. Suppose that
$\{M_k\}_{k=0}^n\subseteq X$ is a martingale (with respect some
filtration). Then
$$
\E \left[\left\|M_n-M_0\right\|_X^p\right]\ge
\frac{1}{(2^{p-1}-1)K_p(X)^p}\sum_{k=1}^n\E\left[\left\|M_k-M_{k-1}\right\|_X^p\right].
$$
\end{theorem}

We  also need the following variant of Pisier's inequality.

\begin{corollary}\label{coro L_2(X)}
Fix $p\in [2,\infty)$, $q\in (1,\infty)$ and let $X$ be a normed space
with $K_p(X)<\infty$. Then for every $q$-integrable martingale
$\{M_k\}_{k=0}^n\subseteq X$, if $q\in [p,\infty)$ then
\begin{equation}\label{eq:easy figiel}
\E \left[\left\|M_n-M_0\right\|_X^q\right]\ge
\frac{1}{(2^{q-1}-1)K_p(X)^q}\sum_{k=1}^n\E\left[\left\|M_k-M_{k-1}\right\|_X^q\right].
\end{equation}
and if  $q\in (1,p]$, then
\begin{equation}\label{eq:hard figiel}
\E \left[\left\|M_n-M_0\right\|_X^q\right]\ge
\frac{\left((1-1/p)(1-1/q)\right)^{q(1-1/p)}}{5^{q(1-1/p)}\left(2K_p(X)\right)^{q}n^{1-q/p}}\sum_{k=1}^n\E\left[\left\|M_k-M_{k-1}\right\|_X^q\right].
\end{equation}
\end{corollary}

\begin{proof} Denote the probability space on which the martingale $\{M_k\}_{k=0}^n$ is defined by $(\Omega,\mu)$.
 Suppose also that $\mathcal{F}_0\subseteq
\mathcal{F}_1\subseteq \cdots\subseteq \mathcal{F}_{n-1}$ is the
filtration with respect to which $\{M_k\}_{k=0}^n$ is a martingale.

If $p\le q$ then~\eqref{eq:easy figiel} is an immediate consequence
of Theorem~\ref{thm:pisier ineq} and~\eqref{eq:monotonicity K}. If
$q\in (1,p]$ then by Corollary~\ref{cor:L_p(X)} we have
$$
K\eqdef K_p(L_q(\mu,X))\le
\left(\frac{5pq}{(p-1)(q-1)}\right)^{1-1/p}K_p(X).
$$ We can therefore apply~\eqref{eq:two point convex} to the
following two vectors in $L_q(\mu,X)$.
$$x=M_{n-1}-M_0+\frac{M_n-M_{n-1}}{2}\quad\mathrm{and}\quad
y=\frac{M_n-M_{n-1}}{2},$$ yielding the following estimate.
\begin{multline}\label{eq:prove lin}
\left(\E\left[\left\|M_{n-1}-M_0+\frac{M_n-M_{n-1}}{2}\right\|_X^q\right]\right)^{p/q}
+\frac{1}{(2K)^p}\left(\E\left[\left\|M_n-M_{n-1}\right\|_X^q\right]\right)^{p/q}\\\le
\frac{\left(\E\left[\left\|M_n-M_{0}\right\|_X^q\right]\right)^{p/q}
+\left(\E\left[\left\|M_{n-1}-M_{0}\right\|_X^q\right]\right)^{p/q}}{2}.
\end{multline}
Now,
$$
\E\left[\left\|M_{n-1}-M_{0}\right\|_X^q\right]=
\E\left[\left\|M_{n-1}-M_{0}+\E\left[M_n-M_{n-1}\Big|\mathcal{F}_{n-1}\right]\right\|_X^q\right]\le
\E\left[\left\|M_{n}-M_{0}\right\|_X^q\right],
$$
and
\begin{multline*}
\E\left[\left\|M_{n-1}-M_{0}\right\|_X^q\right]=\E\left[\left\|M_{n-1}-M_{0}+\E\left[\frac{M_n-M_{n-1}}{2}\Big|
\mathcal{F}_{n-1}\right]\right\|_X^q\right]\\\le
\E\left[\left\|M_{n-1}-M_0+\frac{M_n-M_{n-1}}{2}\right\|_X^q\right].
\end{multline*}
Thus~\eqref{eq:prove lin} implies that
\begin{equation}\label{eq:before sum q}
\left(\E\left[\left\|M_{n-1}-M_0\right\|_X^q\right]\right)^{p/q}
+\frac{1}{(2K)^p}\left(\E\left[\left\|M_n-M_{n-1}\right\|_X^q\right]\right)^{p/q}\le
\left(\E\left[\left\|M_n-M_{0}\right\|_X^q\right]\right)^{p/q}.
\end{equation}
Applying~\eqref{eq:before sum q} inductively we get the lower bound
\begin{multline*}
(2K)^p\left(\E\left[\left\|M_n-M_{0}\right\|_X^q\right]\right)^{p/q}\ge
\sum_{k=1}^n\left(\E\left[\left\|M_k-M_{k-1}\right\|_X^q\right]\right)^{p/q}\\\ge
\frac{1}{n^{\frac{p}{q}-1}}\left(\sum_{k=1}^n\E\left[\left\|M_k-M_{k-1}\right\|_X^q\right]\right)^{p/q},
\end{multline*} which is precisely~\eqref{eq:hard figiel}.
\end{proof}

We are now in position to prove the main theorem of this section, which establishes metric Markov cotype $p$ inequalities (recall Definition~\ref{def:metric markov cotype}) for Banach space with modulus of convexity of power type $p$. An important theorem of Pisier~\cite{Pisier-martingales} asserts that if a normed space $(X,\|\cdot\|_X)$ is super-reflexive then there exists $p\in [2,\infty)$ such that $K_p(X)<\infty$. Thus the case $q=2$ of Theorem~\ref{thm:metric markov cotype general q} below corresponds to Theorem~\ref{thm:markov cotype thm in intro}.

\begin{theorem}\label{thm:metric markov cotype general q} Fix $p\in [2,\infty)$ and let $(X,\|\cdot\|_X)$ be a normed space with $K_p(X)<\infty$. Then for every $m,n\in \N$, every $n\times n$ symmetric stochastic matrix $A=(a_{ij})$ and every $x_1,\ldots,x_n\in X$ there exist $y_1,\ldots y_n\in X$ such that for all $q\in (1,\infty)$,
\begin{multline}\label{eq:max cotype}
\max\left\{\sum_{i=1}^n
\|x_i-y_i\|_X^q,\left(\frac{((1-1/p)(1-1/q))^{1-1/p}}{32\cdot 5^{1-1/p}K_p(X)}\right)^qm^{\min\{1,q/p\}}\sum_{i=1}^n\sum_{j=1}^n
a_{ij}\left\|y_i-y_j\right\|_X^q\right\}\\\le
\sum_{i=1}^n\sum_{j=1}^n\A_m(A)_{ij}\|x_i-x_j\|_X^q.
\end{multline}
In particular, for $q=2$ we have
\begin{equation*}
\sum_{i=1}^n
\|x_i-y_i\|_X^2+m^{2/p}\sum_{i=1}^n\sum_{j=1}^n
a_{ij}\left\|y_i-y_j\right\|_X^2\le (32K_p(X))^2\sum_{i=1}^n\sum_{j=1}^n\A_m(A)_{ij}\|x_i-x_j\|_X^2,
\end{equation*}
Thus $X$ has metric Markov cotype $p$ with exponent $2$ and with $C_p^{(2)}(X)\le 32K_p(X)$.
\end{theorem}

\begin{proof} Define $f\in L_p^n(X)$ by $f(i)=x_i$.
For every $\ell\in \{1,\ldots,n\}$ let
$$
Z_0^{(\ell)}, Z_1^{(\ell)}, Z_2^{(\ell)},\ldots
$$
%$\left\{Z_t^{(\ell)}\right\}_{t=0}^\infty$
be the Markov chain on
$\{1,\ldots,n\}$ which starts at $\ell$ and has transition matrix $A$. In other
words $Z_0^{(\ell)}=\ell$ with probability one and for all $t\in
\{1,\ldots,m\}$ and $i,j\in \{1,\ldots,n\}$ we have
$$\Pr\left[\left.Z_t^{(\ell)}=j\right|Z_{t-1}^{(\ell)}=i\right]=a_{ij}.$$
For $t\in \{0,\ldots,m\}$ define $f_t\in L_p^n(X)$ by
$$
f_t\eqdef (A^{m-t}\otimes I_X^n)f.
$$
 Observe that if we set
$$
M_t^{(\ell)}\eqdef f_t\left(Z_t^{(\ell)}\right)
$$
 then
$
M_0^{(\ell)},M_1^{(\ell)},\ldots,M_m^{(\ell)}$
is a martingale with respect
to the filtration induced by the random variables
$
Z_0^{(\ell)},Z_1^{(\ell)},\ldots,Z_m^{(\ell)}$.
Indeed, writing $L=A\otimes I_X^n$ we have for every $t\ge 1$,
\begin{multline*}
\E\left[\left.M_t^{(\ell)}\right|Z_0^{(\ell)},\ldots,Z_{t-1}^{(\ell)}\right]=\E\left[\left.\left(L^{m-t}f\right)\left(Z_t^{(\ell)}\right)\right|Z_{t-1}^{(\ell)}\right]=
L^{m-t}\E\left[\left.f\left(Z_t^{(\ell)}\right)\right|Z_{t-1}^{(\ell)}\right]\\=L^{m-t}(
Lf)\left(Z_{t-1}^{(\ell)}\right)=\left(
L^{m-(t-1)}f\right)\left(Z_{t-1}^{(\ell)}\right)=M_{t-1}^{(\ell)}.
\end{multline*}

Write
\begin{equation}\label{eq:defK}
K\eqdef\left\{\begin{array}{ll}
(2^{q-1}-1)K_p(X)^q & \mathrm{if\ } q\in [p,\infty),\\
\frac{5^{q(1-1/p)}(2K_p(X))^{q}m^{1-q/p}}{((1-1/p)(1-1/q))^{q(1-1/p)}} &\mathrm{if\ } q\in (1,p).
\end{array}\right.
\end{equation}
Then
Corollary~\ref{coro L_2(X)} applied to the martingale
$\left\{M_t^{(\ell)}\right\}_{t=0}^m$ implies that
\begin{equation}\label{eq:ell separately}
K\E
\left[\left\|f\left(Z_m^{(\ell)}\right)-(L^mf)(\ell)\right\|_X^q\right]\ge
\sum_{t=1}^m
\E\left[\left\|(L^{m-t}f)\left(Z_t^{(\ell)}\right)-(L^{m-t+1}f)\left(Z_{t-1}^{(\ell)}\right)\right\|_X^q\right].
\end{equation}
Let $\{Z_t\}_{t=0}^\infty$ be the Markov chain with transition
matrix $A$ such that $Z_0$ is uniformly distributed on
$\{1,\ldots,n\}$. Averaging~\eqref{eq:ell separately} over $\ell\in
\{1,\ldots,n\}$ yields the inequality
\begin{equation}\label{eq:uniform ell}
K\E \left[\left\|f\left(Z_m\right)-(L^mf)(Z_0)\right\|_X^q\right]\ge
\sum_{t=1}^m
\E\left[\left\|(L^{m-t}f)\left(Z_t\right)-(L^{m-t+1}f)\left(Z_{t-1}\right)\right\|_X^q\right],
\end{equation}
Which is the same as
\begin{multline}\label{eq:in coordinates}
K\sum_{i=1}^n\sum_{j=1}^n
(A^m)_{ij}\left\|f(i)-(L^mf)(j)\right\|_X^q\\\ge \sum_{t=1}^m
\sum_{i=1}^n\sum_{j=1}^n
a_{ij}\left\|(L^{m-t}f)(i)-(L^{m-t+1}f)(j)\right\|_X^q.
\end{multline}
In order to bound the right-hand side of~\eqref{eq:in coordinates}, for every $i\in \{1\ldots,n\}$ consider the vector
\begin{equation}\label{eq:def y_i}
y_i\eqdef \frac{1}{m}\sum_{j=1}^n\sum_{s=0}^{m-1}
(A^s)_{ij}x_j=\frac{1}{m}\sum_{s=0}^{m-1}L^sf(i),
\end{equation}
and observe that
\begin{equation}\label{eq:almost telescope}
\frac{1}{m}\sum_{s=1}^{m}L^sf(i)=y_i-\frac{1}{m}x_i+\frac{1}{m}L^mf(i)=y_i-\frac{1}{m}\sum_{r=1}^n(A^m)_{ir}(x_i-x_r).
\end{equation}
Therefore, using convexity we have:
\begin{eqnarray}\label{eq:get factor m}
&&\!\!\!\!\!\!\!\!\!\!\!\!\!\!\!\!\!\!\!\!\!\!\!\!\!\!\!\!\!\!\!\!\nonumber\sum_{t=1}^m
\sum_{i=1}^n\sum_{j=1}^n
a_{ij}\left\|(L^{m-t}f)(i)-(L^{m-t+1}f)(j)\right\|_X^q\\&\ge&\nonumber m
\sum_{i=1}^n\sum_{j=1}^n
a_{ij}\left\|\frac{1}{m}\sum_{t=1}^m\left((L^{m-t}f)(i)-(L^{m-t+1}f)(j)\right)\right\|_X^q\\
&\stackrel{\eqref{eq:def y_i}\wedge\eqref{eq:almost
telescope}}{=}&\nonumber m \sum_{i=1}^n\sum_{j=1}^n
a_{ij}\left\|y_i-y_j+\frac{1}{m}\sum_{r=1}^n(A^m)_{jr}(x_j-x_r)\right\|_X^q\\&\ge&\nonumber
\frac{m}{2^{q-1}}\sum_{i=1}^n\sum_{j=1}^n
a_{ij}\left\|y_i-y_j\right\|_X^q-\frac{1}{m^{q-1}}\sum_{i=1}^n\sum_{j=1}^n
a_{ij}\left\|\sum_{r=1}^n(A^m)_{jr}(x_j-x_r)\right\|_X^q\\
&=&\frac{m}{2^{q-1}}\sum_{i=1}^n\sum_{j=1}^n
a_{ij}\left\|y_i-y_j\right\|_X^q-\frac{1}{m^{q-1}}\sum_{j=1}^n
\left\|\sum_{r=1}^n(A^m)_{jr}(x_j-x_r)\right\|_X^q\nonumber\\& \ge&
\frac{m}{2^{q-1}}\sum_{i=1}^n\sum_{j=1}^n
a_{ij}\left\|y_i-y_j\right\|_X^q-\frac{1}{m^{q-1}}\sum_{j=1}^n\sum_{r=1}^n(A^m)_{jr}
\left\|x_j-x_r\right\|_X^q.
\end{eqnarray}
At the same time, we can bound the left-hand side of~\eqref{eq:in
coordinates} as follows:
\begin{eqnarray}\label{eq:LHS factor m}
&&\!\!\!\!\!\!\!\!\!\!\!\!\!\!\!\!\!\!\!\!\!\!\!\!\!\!\!\!\!\!\!\!\nonumber\sum_{i=1}^n\sum_{j=1}^n
(A^m)_{ij}\left\|f(i)-(L^mf)(j)\right\|_X^q=\sum_{i=1}^n\sum_{j=1}^n
(A^m)_{ij}\left\|x_i-\sum_{r=1}^n(A^m)_{jr}x_r\right\|_X^q\\&\le&\nonumber
\sum_{i=1}^n\sum_{j=1}^n\sum_{r=1}^n
(A^m)_{ij}(A^m)_{jr}\left\|x_i-x_r\right\|_X^q\\&\le&\nonumber
2^{q-1}\sum_{i=1}^n\sum_{j=1}^n\sum_{r=1}^n(A^m)_{ij}(A^m)_{jr}\left(\|x_i-x_j\|_X^q+\|x_j-x_r\|_X^q\right)\\
&=& 2^q\sum_{i=1}^n\sum_{j=1}^n(A^m)_{ij}\|x_i-x_j\|_X^q.
\end{eqnarray}
We note that,
\begin{eqnarray*}
&&\!\!\!\!\!\!\!\!\!\!\!\!\!\!\!\!\!\!\sum_{i=1}^n\sum_{j=1}^n(A^m)_{ij}\|x_i-x_j\|_X^q=
\sum_{i=1}^n\sum_{j=1}^n\left(\frac{1}{m}\sum_{t=0}^{m-1}A^tA^{m-t}\right)_{ij}\|x_i-x_j\|_X^q\\&\le&
\frac{2^{q-1}}{m}\sum_{i=1}^n\sum_{j=1}^n\sum_{r=1}^n\sum_{t=0}^{m-1}(A^t)_{ir}(A^{m-t})_{rj}
\left(\|x_i-x_r\|_X^q+\|x_r-x_j\|_X^q\right)\\
&=& 2^{q-1} \sum_{i=1}^n\sum_{r=1}^n
\left(\frac{1}{m}\sum_{t=0}^{m-1}A^t\right)_{ir}\|x_i-x_r\|_X^q+2^{q-1}
\sum_{j=1}^n\sum_{r=1}^n
\left(\frac{1}{m}\sum_{t=0}^{m-1}A^{m-t}\right)_{rj}\|x_r-x_j\|_X^q\\
&=& 2^q \sum_{i=1}^n\sum_{j=1}^n
\A_m(A)_{ij}\|x_i-x_j\|_X^q+
\frac{2^{q-1}}{m}\sum_{i=1}^n\sum_{j=1}^n(A^m)_{ij}\|x_i-x_j\|_X^q,
\end{eqnarray*}
which, assuming that $m\ge 2^{q}$ gives the following bound.
\begin{equation}\label{eq:case m big}
\sum_{i=1}^n\sum_{j=1}^n(A^m)_{ij}\|x_i-x_j\|_X^q\le
2^{q+1}\sum_{i=1}^n\sum_{j=1}^n
\A_m(A)_{ij}\|x_i-x_j\|_X^q.
\end{equation}
On the other hand, if $m\le 2^{q}$ then
\begin{eqnarray}\label{eq:case m small}
&&\!\!\!\!\!\!\!\!\!\!\!\!\!\!\!\!\!\!\!\!\!\!\!\!\!\nonumber\sum_{i=1}^n\sum_{j=1}^n(A^m)_{ij}\|x_i-x_j\|_X^q\le
\sum_{i=1}^n\sum_{j=1}^n\sum_{r=1}^na_{ir}(A^{m-1})_{rj}\|x_i-x_j\|_X^q\\\nonumber&\le&
2^{q-1}\sum_{i=1}^n\sum_{j=1}^n\sum_{r=1}^na_{ir}(A^{m-1})_{rj}\left(\|x_i-x_r\|_X^q+\|x_r-x_j\|_X^q\right)\\
\nonumber&=& 2^{q-1}\sum_{i=1}^n\sum_{j=1}^na_{ij}\|x_i-x_j\|_X^q+
2^{q-1}\sum_{i=1}^n\sum_{j=1}^n(A^{m-1})_{ij}\|x_i-x_j\|_X^q\\
&\le&
2^{q-1}m\sum_{i=1}^n\sum_{j=1}^n\A_m(A)_{ij}\|x_i-x_j\|_X^q
\le
2^{2q-1}\sum_{i=1}^n\sum_{j=1}^n\A_m(A)_{ij}\|x_i-x_j\|_X^q.
\end{eqnarray}
Thus, by combining~\eqref{eq:case m big} and~\eqref{eq:case m small} we get the estimate
\begin{equation}\label{eq:bound by average}
\sum_{i=1}^n\sum_{j=1}^n(A^m)_{ij}\|x_i-x_j\|_X^q\le
4^q\sum_{i=1}^n\sum_{j=1}^n\A_m(A)_{ij}\|x_i-x_j\|_X^q.
\end{equation}
Substituting~\eqref{eq:get factor m}
and~\eqref{eq:LHS factor m} into~\eqref{eq:in coordinates} yields
the bound
\begin{multline}\label{one of the terms in cotype}
m\sum_{i=1}^n\sum_{j=1}^n a_{ij}\left\|y_i-y_j\right\|_X^q\le
4^qK\sum_{i=1}^n\sum_{j=1}^n(A^m)_{ij}\|x_i-x_j\|_X^q\\\stackrel{\eqref{eq:bound
by average}}{\le}
2^{4q}K\sum_{i=1}^n\sum_{j=1}^n\A_m(A)_{ij}\|x_i-x_j\|_X^q.
\end{multline}
At the same time,
\begin{equation}\label{eq:displacement}
\sum_{i=1}^n
\|x_i-y_i\|_X^q=\sum_{i=1}^n\left\|\frac{1}{m}\sum_{j=1}^n\sum_{t=0}^{m-1}(A^t)_{ij}(x_i-x_j)\right\|_X^q \le
\sum_{i=1}^n\sum_{j=1}^n\A_m(A)_{ij}\|x_i-x_j\|_X^q.
\end{equation}
Recalling~\eqref{eq:defK}, the desired inequality~\eqref{eq:max cotype} is now a combination
of~\eqref{one of the terms in cotype} and~\eqref{eq:displacement}.
\end{proof}

\section{Construction of the base graph}\label{sec:base}

For $t\in (0,\infty)$ and $n\in \N$ write
\begin{equation}\label{eq:def tau sigma}
\tau_t\eqdef \frac{1-e^{-t}}{2}\quad\mathrm{and}\quad
\sigma_t^n\eqdef \tau_t^{4\tau_t n}(1-\tau_t)^{(1-4\tau_t)n}.
\end{equation}
We also define $e_t^n:\{0,\ldots,n\}\to \N\cup\{0\}$ by
\begin{equation}\label{eq:def e_t} e_t^n(k)\eqdef \left\lfloor
\frac{\tau_t^k\left(1-\tau_t\right)^{n-k}}{\sigma_t^n}
\right\rfloor.
\end{equation}

The following lemma records elementary estimates on binomial sums
that will be useful for us later.

\begin{lemma}\label{lem:tau estimates} Fix $t\in (0,1/4)$ and
$n\in \N\cap [8000,\infty)$ such that
\begin{equation}\label{eq:lower tau assumption}
\tau_t\ge \frac{1}{3\sqrt{n}}.
\end{equation}
Then
\begin{equation}\label{eq:useful1}
\frac{1}{3\sigma_t^n} \le \sum_{k\in \Z\cap [0,4\tau_tn]}\binom{n}{k} e_t^n(k)\le \frac{1}{\sigma_t^n}.
\end{equation}
Moreover,
for every $s\in \Z\cap (4\tau_t n,n]$ we have
\begin{equation}\label{eq:useful2}
\sum_{m\in \Z\cap[(s-4\tau_t n)/2,s/2]}\binom{n}{s-2m} e_t^n(s-2m)\ge \frac{1}{18\sigma_t^n}.
\end{equation}
\end{lemma}

\begin{proof}
For simplicity of notation write $\tau=\tau_t$ and
$\sigma=\sigma_t^n$. The rightmost inequality in~\eqref{eq:useful1}
is an immediate consequence of~\eqref{eq:def e_t}. To establish the
leftmost estimate in~\eqref{eq:useful1} note that by the Chernoff
inequality (e.g.~\cite[Thm. A.1.4]{AS}) we have
\begin{equation}\label{eq:chernoff d}
\sum_{k\in \Z\cap (4\tau n,n]} \binom{n}{k}
\tau^k(1-\tau)^{n-k}<
e^{-18\tau^2n}\stackrel{\eqref{eq:lower tau assumption}}{\le}
\frac13.
\end{equation}
For every $k\in \{1,\ldots,n\}$ satisfying $k\le 4\tau n$ we have
$\tau^k(1-\tau)^{n-k}\ge \sigma$, and therefore $e_t^n(k)\ge
\frac{1}{2\sigma}\tau^k(1-\tau)^{n-k}$. Hence
\begin{equation}\label{eq:lower d}
\sum_{k\in \Z\cap [0,4\tau n]}\binom{n}{k} e_t^n(k) \ge \frac{1}{2\sigma}\sum_{k\in \Z\cap [0,4\tau n]}\binom{n}{k}
\tau^k(1-\tau)^{n-k}\stackrel{\eqref{eq:chernoff
d}}{>}\frac{1}{2\sigma}\left(1-\frac13\right)=\frac{1}{3\sigma}.
\end{equation}
This completes the proof of~\eqref{eq:useful1}.

To prove~\eqref{eq:useful2}, we apply a standard binomial
concentration bound (e.g.~\cite[Cor.~A.1.14]{AS}) to get the
 estimate
\begin{equation}\label{eq:8/9}
\sum_{k\in \Z\cap [\tau n/2,3\tau n/2]}\binom{n}{k}\tau^k(1-\tau)^{n-k}\ge 1-2 e^{-\tau n/10}\ge \frac89,
\end{equation}
where in the rightmost inequality in~\eqref{eq:8/9} we used the assumptions~\eqref{eq:lower tau assumption} and $n\ge 8000$. Observe that for every $k\in \Z\cap [\tau n/2,3\tau n/2]$, since
by the assumption $t\in (0,1/4)$ we have $\tau\in (0,1/8)$,
\begin{equation}\label{eq:1/4,4}
\frac{\binom{n}{k+1}\tau^{k+1}(1-\tau)^{n-k-1}}{\binom{n}{k}\tau^k(1-\tau)^{n-k}}
=\frac{\tau}{1-\tau}\cdot\frac{n-k}{k+1}\in \left[\frac{1-3\tau/2}{3(1-\tau)},\frac{2-\tau}{1-\tau}\right]\subseteq
\left[\frac{1}{4},4\right].
\end{equation}
It follows that
$$
\sum_{k\in (2\Z)\cap [\tau n/2,3\tau n/2]}\binom{n}{k}\tau^k(1-\tau)^{n-k}
\stackrel{\eqref{eq:1/4,4}}{\ge} \frac{1}{8}\sum_{k\in \Z\cap [\tau n/2,3\tau n/2]}\binom{n}{k}\tau^k(1-\tau)^{n-k}
\stackrel{\eqref{eq:8/9}}{\ge} \frac19,
$$
and, for the same reason,
$$
\sum_{k\in (2\Z+1)\cap [\tau n/2,3\tau n/2]}\binom{n}{k}\tau^k(1-\tau)^{n-k}\ge \frac19.
$$
Thus,
\begin{equation}\label{eq:1/9}
\sum_{m\in \Z\cap [(s-3\tau n/2)/2, (s-\tau n/2)/2]} \binom{n}{s-2m}\tau^{s-2m}(1-\tau)^{n-(s-2m)}
 \ge \frac19.
\end{equation}
Finally,
\begin{multline*}
\sum_{m\in \Z\cap[(s-4\tau n)/2,s/2]}\binom{n}{s-2m}
e_t^n(s-2m)\\\stackrel{\eqref{eq:def e_t}}{\ge}
 \frac{1}{2\sigma}\sum_{m\in \Z\cap [(s-3\tau n/2)/2, (s-\tau n/2)/2]}
 \binom{n}{s-2m}\tau^{s-2m}(1-\tau)^{n-(s-2m)}\stackrel{\eqref{eq:1/9}}{\ge}
 \frac{1}{18\sigma}.\tag*{\qedhere}
\end{multline*}
\end{proof}

\begin{lemma}[Discretization of $e^{-t\Delta}$ w.r.t. Poincar\'e
inequalities]\label{lem:heat discretization} Fix $t\in (0,1/4)$, $p\in [1,\infty)$ and $n\in
\N\cap [2^{13},\infty)$ such that
\begin{equation}\label{eq:discretization assumption}
\tau_t\ge \sqrt{\frac{p\log(18n)}{18n}}.
\end{equation}
Let $G_t^n=(\F_2^n,E_t^n)$ be the graph whose vertex set is $\F_2^n$
and every $x,y\in \F_2^n$ is joined by $e_t^n(\|x-y\|_1)$ edges. Then
the graph $G_t^n$ is $d_t^n\in \N$ regular, where
\begin{equation}\label{eq:degree bounds}
\frac{1}{3\sigma_t^n}\le d_t^n\le \frac{1}{\sigma_t^n}.
\end{equation}
Moreover, for every metric space $(X,d_X)$ and every $f,g:\F_2^n\to
X$ we have
\begin{multline}\label{eq:discrete heat}
\frac{1}{3|E_t^n|}\sum_{(x,y) \in E_t^n}
d_X(f(x),g(y))^p\le \frac{1}{2^n}\sum_{(x,y)\in
\F_2^n\times \F_2^n}\left(e^{-t\Delta}\delta_x\right)(y)d_X(f(x),g(y))^p\\ \le
\frac{3}{|E_t^n|}\sum_{(x,y)\in E_t^n}
d_X(f(x),g(y))^p.
\end{multline}
\end{lemma}

\begin{proof} Observe that  the assumptions of Lemma~\ref{lem:heat discretization} imply the assumptions of Lemma~\ref{lem:tau estimates}. We may therefore use the conclusions of
Lemma~\ref{lem:tau estimates} in the ensuing proof. For simplicity
of notation write $\tau=\tau_t$ and $\sigma=\sigma_t^n$. By
definition $G_t^n$ is a regular graph. Denote its degree by
$d=d_t^n$. Then,
\begin{equation}\label{eq:d formula}
d=\sum_{k=0}^n \binom{n}{k}e_t(k)\stackrel{\eqref{eq:def e_t}}{\in}\left[\frac{1}{3\sigma},\frac{1}{\sigma}\right].
\end{equation}
This proves~\eqref{eq:degree bounds}.  We also immediately deduce
the leftmost inequality in~\eqref{eq:discrete heat} as follows.
\begin{eqnarray*}
&&\!\!\!\!\!\!\!\!\!\!\!\!\!\!\!\!\!\!\!\!\!\!\!\!\!\!\!\!\!\!\frac{1}{2^n}\sum_{(x,y)\in
\F_2^n\times \F_2^n}\left(e^{-t\Delta}\delta_x\right)(y)d_X(f(x),g(y))^p\\&\stackrel{\eqref{eq:matrix
form heat}}{=}&\frac{1}{2^n}\sum_{(x,y)\in \F_2^n\times \F_2^n}
\tau^{\|x-y\|_1}(1-\tau)^{n-\|x-y\|_1}d_X(f(x),g(y))^p\\&\stackrel{\eqref{eq:def
e_t}}{\ge}& \frac{\sigma}{2^n}\sum_{(x,y)\in \F_2^n\times \F_2^n}e_t^n(\|x-y\|_1)d_X(f(x),g(y))^p\\ &\stackrel{\eqref{eq:degree bounds}}{\ge}& \frac{1}{3|E_t^n|} \sum_{(x,y)\in E_t^n} d_X(f(x),g(y))^p,
\end{eqnarray*}
where we used the fact that $|E_n^t|=2^nd$.

It remains to prove the rightmost inequality in~\eqref{eq:discrete
heat}. To this end fix $k\in \Z$ satisfying $0\le k\le 4\tau n$ and
$m\in \N\cup\{0\}$ satisfying $k+2m\le n$. For every permutation $\pi\in S_n$ define $z_0^\pi,\ldots,z_{2m+1}^\pi,y_0^\pi,\ldots,y^\pi_{2m+1}\in \F_2^n$ by setting $z^\pi_0=y_0^\pi=0$ and for $i\in \{1,\ldots,2m+1\}$,
\begin{equation*}\label{eq:def zipi}
z_i^\pi\eqdef \sum_{j=1}^{k-1} e_{\pi(j)}+e_{\pi(k+i-1)},
\end{equation*}
and
\begin{equation}\label{eq:def yipi}
y_i^\pi\eqdef \sum_{j=1}^{i} z_j^\pi,
\end{equation}
where the sum in~\eqref{eq:def yipi} is performed in $\F_2^n$ (i.e., modulo $2$), and we recall that $e_1,\ldots,e_n$ is the standard basis of $\F_2^n$. For every $x\in \F_2^n$ we have
\begin{multline*}
d_X\left(f(x),g\left(x+y_{2m+1}^\pi\right)\right)\\\le \sum_{i=0}^md_X\left(f\left(x+y_{2i}^\pi\right),g\left(x+y_{2i+1}^\pi\right)\right)+
\sum_{i=0}^{m-1}d_X\left(g\left(x+y_{2i+1}^\pi\right),f\left(x+y_{2i+2}^\pi\right)\right).
\end{multline*}
Hence, H\"older's inequality yields the following estimate.
\begin{multline}\label{eq:triangle holder}
\frac{d_X\left(f(x),g\left(x+y_{2m+1}^\pi\right)\right)^p}{(2m+1)^{p-1}}\\\le \sum_{i=0}^m
d_X\left(f\left(x+y_{2i}^\pi\right),g\left(x+y_{2i+1}^\pi\right)\right)^p+
\sum_{i=0}^{m-1}d_X\left(g\left(x+y_{2i+1}^\pi\right),f\left(x+y_{2i+2}^\pi\right)\right)^p.
\end{multline}

Note that
$$
y_{2m+1}^\pi=\sum_{j=1}^{k+2m}e_{\pi(j)}.
$$
Therefore, if $\pi\in S_n$ is chosen uniformly at random then
$y_{2m+1}^\pi$ is distributed uniformly over the $\binom{n}{k+2m}$
elements $w\in \F_2$ with $\|w\|_1=k+2m$. This observation implies
that
\begin{equation}\label{eq:left average}
\frac{1}{2^nn!}\sum_{x\in \F_2^n} \sum_{\pi\in S_n}
d_X\left(f(x),g\left(x+y_{2m+1}^\pi\right)\right)^p=\frac{1}{2^n\binom{n}{k+2m}}\sum_{\substack{(x,y)\in
\F_2^n\times \F_2^n\\ \|x-y\|_1=k+2m}} d_X\left(f(x),g\left(y\right)\right)^p.
\end{equation}
Similarly, for every $j\in \{0,\ldots,2m\}$ we have
\begin{multline}\label{eq:invariance}
\sum_{x\in \F_2^n} \sum_{\pi\in
S_n}d_X\left(f\left(x+y_{j}^\pi\right),g\left(x+y_{j+1}^\pi\right)\right)^p=
\sum_{\pi\in S_n}\sum_{x\in \F_2^n}
d_X\left(f\left(x+y_{j}^\pi\right),g\left(x+y_{j}^\pi+z_j^\pi\right)\right)^p\\=\sum_{\pi\in
S_n}\sum_{u\in \F_2^n}
d_X\left(f\left(u\right),g\left(u+z_j^\pi\right)\right)^p=
\frac{n!}{\binom{n}{k}}\sum_{\substack{(u,v)\in \F_2^n\times \F_2^n\\ \|u-v\|_1=k}}
d_X\left(f(u),g\left(v\right)\right)^p,
\end{multline}
where in the penultimate equality of~\eqref{eq:invariance} we used
the fact that for each $\pi\in S_n$, if $x$ is chosen uniformly at
random from $\F_2^n$ then $x+y_j^\pi$ is distributed uniformly over
$\F_2^n$, and in the last equality of~\eqref{eq:invariance} we used
the fact that, because $\|z_j^\pi\|_1=k$, if $\pi\in S_n$ is chosen
uniformly at random then $z_{j}^\pi$ is distributed uniformly over
the $\binom{n}{k}$ elements $w\in \F_2$ with $\|w\|_1=k$.

A combination of~\eqref{eq:triangle holder}, \eqref{eq:left average}
and~\eqref{eq:invariance} yields the following (crude) estimate.
\begin{equation}\label{eq:average at fixed distance}
\frac{1}{2^n\binom{n}{k+2m}}\sum_{\substack{(x,y)\in \F_2^n\times \F_2^n\\
\|x-y\|_1=k+2m}} d_X\left(f(x),g\left(y\right)\right)^p\le
\frac{n^p}{2^n\binom{n}{k}}\sum_{\substack{(x,y)\in \F_2^n\times \F_2^n\\
\|x-y\|_1=k}} d_X\left(f(x),g\left(y\right)\right)^p.
\end{equation}
If we fix $s\in \N\cap(4\tau n,n]$ then~\eqref{eq:average at fixed distance} implies that for every $m\in \N\cap [(s-4\tau n)/2,s/2]$,
\begin{equation}\label{eq:pass to s notation}
\frac{\binom{n}{s-2m}}{n^p\binom{n}{s}}\sum_{\substack{(x,y)\in \F_2^n\times \F_2^n\\
\|x-y\|_1=s}} d_X\left(f(x),g\left(y\right)\right)^p\le \sum_{\substack{(x,y)\in \F_2^n\times \F_2^n\\
\|x-y\|_1=s-2m}}  d_X\left(f(x),g\left(y\right)\right)^p.
\end{equation}
Multiplying both sides of~\eqref{eq:pass to s notation} by
$e_t^n(s-2m)$ and summing over $m\in \N\cap [(s-4\tau n)/2,s/2]$ yields the
following estimate.
\begin{multline*}
\frac{\sum_{m\in \Z\cap [(s-4\tau
n)/2,s/2]}\binom{n}{s-2m}e_t^n(s-2m)}
{n^p\binom{n}{s}}\sum_{\substack{x,y\in \F_2^n\\
\|x-y\|_1=s}} d_X\left(f(x),g\left(y\right)\right)^p\\\le
\sum_{m\in \Z\cap [(s-4\tau n)/2,s/2]}e_t^n(s-2m)\sum_{\substack{(x,y)\in \F_2^n\times \F_2^n\\
\|x-y\|_1=s-2m}}  d_X\left(f(x),g\left(y\right)\right)^p\le
\sum_{(x,y)\in  E_t^n} d_X(f(x),g(y))^p.
\end{multline*}
Due to~\eqref{eq:useful2} it follows that for every $s\in
\N\cap(4\tau n,n]$ we have
\begin{equation}\label{eq:average dist s}
 \frac{1}{\binom{n}{s}}\sum_{\substack{(x,y)\in \F_2^n\times \F_2^n\\
\|x-y\|_1=s}} d_X\left(f(x),g\left(y\right)\right)^p\le 18\sigma n^p
\sum_{(x,y)\in  E_t^n} d_X(f(x),g(y))^p.
\end{equation}
Now,
\begin{eqnarray*}
&&\!\!\!\!\!\!\!\!\!\!\!\!\!\!\!\!\!\!\!\!\!\!\!\!\!\!\!\!\!\!\!\!\!\!\!\!\frac{1}{2^n}\sum_{(x,y)\in
\F_2^n\times \F_2^n}\left(e^{-t\Delta}\delta_x\right)(y)d_X(f(x),g(y))^p\\&\stackrel{\eqref{eq:matrix
form heat}}{=}&
\frac{1}{2^n}\sum_{s=0}^n\tau^s(1-\tau)^{n-s}\sum_{\substack{(x,y)\in \F_2^n\times \F_2^n\\ \|x-y\|_1=s}}d(f(x),g(y))^p\\
&\stackrel{\eqref{eq:def e_t}\wedge\eqref{eq:average dist s}}{\le}&
\frac{\sigma}{2^n}\left(2+18 n^p\sum_{s\in \Z\cap (4\tau n,n]}\binom{n}{s}\tau^s(1-\tau)^{n-s}\right)
\sum_{(x,y)\in E_t^n} d_X(f(x),g(y))^p\\
&\stackrel{\eqref{eq:chernoff d}\wedge \eqref{eq:d formula}}{\le}& \left(2+18n^p e^{-18\tau^2 n}\right)\frac{1}{d2^n}\sum_{(x,y)\in
 E_t^n} d_X(f(x),g(y))^p\\
&\stackrel{\eqref{eq:discretization assumption}}{\le}& \frac{3}{|E_t^n|}\sum_{(x,y)\in
 E_t^n} d_X(f(x),g(y))^p.
\end{eqnarray*}
This concludes the proof of~\eqref{eq:discrete heat}.
\end{proof}

In what follows for every $n\in \N$ we fix $V_n\subseteq \F_2^n$ which is a ``good linear code", i.e., a linear subspace over $\F_2$ with
\begin{equation}\label{eq:C assumptions}
D_n\eqdef \dim(V_n)\ge \frac{n}{10}\quad \mathrm{and}\quad k_n\eqdef \min_{x\in V_n\setminus \{0\}} \|x\|_1\ge \frac{n}{10}.
\end{equation}
Also, we assume that the sequences $\{D_n\}_{n=1}^\infty$ and $\{k_n\}_{n=1}^\infty$ are increasing. The essentially arbitrary choice of the constant $10$ in~\eqref{eq:C assumptions} does not play an important role in what follows. The fact that $\{V_n\}_{n=1}^\infty$ exists is simple; see~\cite{MS77}. We shall use the standard notation
$$
V_n^\perp\eqdef \left\{x\in \F_2^n:\ \forall\, y\in V_n,\quad \sum_{j=1}^n x_jy_j\equiv 0\mod 2\right\}.
$$

\begin{lemma}\label{lem:base code} For every $K,p\in (1,\infty)$ there exists $n(K,p)\in \N$ and $\delta(K,p)\in (0,1)$ with the following properties. Setting
\begin{equation}\label{eq:def m_n}
m_n\eqdef |\F_2^n/V_n^\perp|\stackrel{\eqref{eq:C assumptions}}{=}2^{D_n},
\end{equation}
 there exists a sequence of connected regular graphs $$\{H_n(K,p)\}_{n=n(K,p)}^\infty$$  such that for every integer $n\ge n(K,p)$ the graph $H_n(K,p)$ has $m_n$ vertices and degree
 \begin{equation}\label{e:def d_n}
 d_n(K,p)\le e^{(\log m_n)^{1-\delta(K,p)}},
 \end{equation}
and for every $K$-convex Banach space $X=(X,\|\cdot\|_X)$ with $K(X)\le K$, %if $n\in \N$ satisfies $n\ge n(K,p)$ then
\begin{equation}\label{eq:gamma+ base grapg bound for large index}
\forall\, n\in [n(K,p),\infty)\cap\N,\quad \gamma_+\left((H_n(K,p),\|\cdot\|_X^p\right)\le 9^{p+1}.
\end{equation}
\end{lemma}

\begin{proof} Fix $K,p\in (1,\infty)$. Let $A=A(K,p), B=B(K,p), C=C(K,p)$ be the constants of Theorem~\ref{thm:AB}. Recall that $B>2$. Set 
\begin{equation}\label{eq:our choice of t}
t=t(n,K,p)\eqdef \left(\frac{\log(2C)}{k_nA}\right)^{1/B},
\end{equation}
where $k_n$ is given in~\eqref{eq:C assumptions}.
Then there exists $n(K,p)\in \N$ such that every integer $n\ge n(K,p)$ satisfies the assumptions of Lemma~\ref{lem:heat discretization}, and moreover there exists $\delta(K,p)\in (0,1)$ such that for every integer $n\ge n(K,p)$ we have
\begin{equation}\label{eq:def delta(K,p)}
\frac{1}{\tau_t^{8n\tau_t}}\le e^{(\log m_n)^{1-\delta(K,p)}}.
\end{equation}
(To verify~\eqref{eq:def delta(K,p)} recall that $\log m_n=D_n\log 2\ge n/20$.)

Assume from now on that $n\in \N$ satisfies $n\ge n(K,p)$. Let $G_t^n=(\F_2^n,E_t^n)$ be the graph constructed in Lemma~\ref{lem:heat discretization}. The degree of $G_t^n$ is
\begin{equation*}\label{eq:degree bound for K-convex}
d_t^n\stackrel{\eqref{eq:degree bounds}}{\le}\frac{1}{\sigma_t^n}\stackrel{\eqref{eq:def tau sigma}}{\le}\frac{1}{\tau_t^{8n\tau_t}}\stackrel{\eqref{eq:def delta(K,p)}}{\le} e^{n^{1-\delta(K,p)}}.
\end{equation*}

The desired graph $H_n=H_n(K,p)$ is defined to be the following quotient of $G_t^n$. The vertex set of $H_n$  is $\F_2^n/V_n^\perp$. Given two cosets $x+V_n^\perp,y+V_n^\perp\in \F_2^n/V_n^\perp$, the number of edges  joining $x+V_n^\perp$ and $y+V_n^\perp$ in $H_n$ is defined to be the number of edges of $G_t^n$ with one endpoint in $x+V_n^\perp$ and the other endpoint in $y+V_n^\perp$, divided by the cardinality of $V_n^\perp$. Thus, the number of edges joining $x+V_n^\perp$ and $y+V_n^\perp$ in the graph $H_n$ equals
$$
\frac{1}{|V_n^{\perp}|}\sum_{(u^\perp,v^\perp)\in V_n^\perp\times V_n^\perp}e_t^n\left(\left\|x-y+(u^\perp-v^\perp)\right\|_1\right)=\sum_{u^\perp\in V_n^\perp} e_t^n\left(\left\|x-y+u^\perp\right\|_1\right).
$$
Hence $H_n$ is a regular graph of the same degree as $G_t^n$ (i.e., the degree of $H_n$ equals $d_t^n$). In what follows we let $\pi:\F_2^n\to \F_2^n/V_n^\perp$ denote the quotient map.

Fix a $K$-convex Banach space $(X,\|\cdot\|_X)$ with $K(X)\le K$. For every $f\in L_p(\F_2^n/V_n^\perp,X)$ define $\pi f:\F_2^n\to X$ by $\pi f(x)=f(\pi(x))$. Thus $\pi f$ is constant on the cosets of $V_n^\perp$. It follows from~\cite[Lem.~3.3]{KN06} that if $\sum_{x\in \F_2^n/V_n^\perp}f(x)=0$ then $\pi f\in L_p^{\ge k_n}(\F_2^n,X)$, where $k_n$ is defined in~\eqref{eq:C assumptions}. By Theorem~\ref{thm:AB} we therefore have
\begin{equation}\label{eq:1/2 bound}
\frac{\left\|\left(e^{-t\Delta}\pi\right)f\right\|_{L_p(\F_2^n/V_n^\perp,X)}}{\|f\|_{L_p(\F_2^n/V_n^\perp,X)}}\le Ce^{-Ak_n\min\left\{t,t^{B}\right\}}\stackrel{\eqref{eq:our choice of t}}{=}\frac12.
\end{equation}
Let $Q$ be the $(\F_2^n/V_n^\perp)\times (\F_2^n/V_n^\perp)$ symmetric stochastic matrix corresponding to the averaging operator $e^{-t\Delta}\pi$, i.e., the entry of $Q$ at $(x+V_n^\perp,y+V_n^\perp)\in (\F_2^n/V_n^\perp)\times (\F_2^n/V_n^\perp)$ is
\begin{equation}\label{eq:def P}
q_{x+V_n^\perp,y+V_n^\perp}\eqdef \left(\left(e^{-t\Delta}\pi\right)\delta_{x+V_n^\perp}\right)\left(y+V_n^\perp\right)=\sum_{\substack{u\in a+V_n^\perp\\ v\in b+V_n^\perp}}\tau_t^{\|a-b\|_1}(1-\tau_t)^{n-\|a-b\|_1}.
\end{equation}
Since~\eqref{eq:1/2 bound} holds for all $f\in L_p(\F_2^n/V_n^\perp,X)$ with $\sum_{x\in \F_2^n/V_n^\perp}f(x)=0$, we have $\lambda_X^{(p)}(Q)\le \frac12$ (recall here the notation introduced in~\eqref{eq:def lambdap}). Consequently, Lemma~\ref{lem:norm to poincare} implies that
$$
\gamma_+\left(Q,\|\cdot\|_X^p\right)\le 9^p.
$$
Thus every $f,g:\F_2^n/V_n^\perp \to X$ satisfy
\begin{multline}\label{eq:gap on quotient}
\frac{1}{|\F_2^n/V_n^\perp|^2}\sum_{(S,T)\in (\F_2^n/V_n^\perp)\times (\F_2^n/V_n^\perp)}\|f(S)-g(T)\|_X^p\\\le \frac{9^p}{|\F_2^n/V_n^\perp|}\sum_{(S,T)\in (\F_2^n/V_n^\perp)\times (\F_2^n/V_n^\perp)}q_{S,T}\|f(S)-g(T)\|_X^p.
\end{multline}
Observe that
\begin{eqnarray}\label{rewrite quotient}
&&\nonumber\!\!\!\!\!\!\!\!\!\!\!\!\!\!\!\!\!\!\!\!\!\!\!\!\!\!\!\!\!\!\sum_{(S,T)\in (\F_2^n/V_n^\perp)\times (\F_2^n/V_n^\perp)}q_{S,T}\|f(S)-g(T)\|_X^p\stackrel{\eqref{eq:def P}}{=}\sum_{(a,b)\in \F_2^n\times \F_2^n} \left(e^{-t\Delta}\delta_a\right)(b)\|\pi f(a)-\pi g(b)\|_X^p\\\nonumber
&\stackrel{\eqref{eq:discrete heat}}{\le}& \frac{3}{|E_t^n|}\sum_{(a,b)\in  E_t^n}\|\pi f(a)-\pi g(b)\|_X^p\\\nonumber
&=& \frac{3}{2^nd_t^n}\sum_{(S,T)\in (\F_2^n/V_n^\perp)\times (\F_2^n/V_n^\perp)}\left(\sum_{(a,b)\in S\times T}E_t^n(a,b)\right) \|f(S)-g(T)\|_X^p\\
&=&\frac{3}{|E(H_n)|}
\sum_{(S,T)\in E(H_n)}\|f(S)-g(T)\|_X^p.
\end{eqnarray}
In~\eqref{rewrite quotient} we used the fact that for every $S,T\in \F_2^n/V_n^\perp$, by the definition of the graph $H_n$, the quantity $$\frac{1}{|V_n^\perp|}\sum_{(a,b)\in S\times T}E_t^n(a,b)$$ equals the number of edges joining $S$ and $T$ in $H_n$, and that since $H_n$ is a $d_t^n$-regular graph we have $|V_n^\perp|/(2^nd_t^n)=1/|E(H_n)|$.

The desired estimate~\eqref{eq:gamma+ base grapg bound for large index} now follows from~\eqref{eq:gap on quotient} and~\eqref{rewrite quotient}.
\end{proof}

The case $p=2$ of Corollary~\ref{cor:base graph with p} below (which is nothing more than a convenient way to restate Lemma~\ref{lem:base code}) corresponds to Lemma~\ref{lem:base in section}.

\begin{corollary}\label{cor:base graph with p}
For every $\delta\in (0,1)$ and $p\in (1,\infty)$ there exists $n_0^p(\d)\in \N$ and a sequence of regular graphs $\{H_n^p(\delta)\}_{n=n_0^p(\d)}^\infty$ such that for every every $n\ge n_0^p(\d)$ the graph $H_n^p(\delta)$ is regular and has $m_n$ vertices, with $m_n$ given in~\eqref{eq:def m_n}. The degree of $H_n^p(\d)$, denoted $d_n^p(\d)$, satisfies
\begin{equation}\label{eq:dnp bound}
d_n^p(\delta)\le e^{(\log m_n)^{1-\delta}}.
\end{equation}
Moreover, for every $K$-convex Banach space $(X,\|\cdot\|_X)$ we have $\gamma_+\left(H_n^p(\d),\|\cdot\|_X^p\right)<\infty$ for all integers $n\ge n_0^p(\d)$, and  there exists $\delta_0^p(X)\in (0,1)$ such that for every $0<\delta\le \delta_0^p(X)$ and every integer $n\ge n_0^p(\d)$ we have
\begin{equation}\label{eq:invert}
\gamma_+\left(H_n^p(\delta),\|\cdot\|_X^p\right)\le 9^{p+1}.
\end{equation}
\end{corollary}

\begin{proof} We shall use here the notation of Lemma~\ref{lem:base code}. We may assume without loss of generality that $\d(K,p)$ decreases continuously with $K$ and that $\lim_{K\to \infty} \d(K,p)=0$. If $\d\in (\d(2,p),1)$ then let $n_0^p(\d)$ be the smallest integer such that $(\log m_n)^{1-\d}\ge \log 3$ and set $H_n^p(\d)=C_{m_n}^\circ$ be the $m_n$-cycle with self loops. Since in this case $d_n^p(\d)=3$, the desired degree bound~\eqref{eq:dnp bound} holds true by design. Moreover, in this case the finiteness of $\gamma_+(H_n^p(\d),\|\cdot\|_X^p)$  is a consequence of Lemma~\ref{lem:general graph}. For $\d\in (0,\d(2,p)]$ we can define $
K_\delta^p= \sup\left\{K\in [2,\infty):\ \delta(K,p)\ge \delta\right\}$.
Set $n_0^p(\d)=n(K_\d^p,p)$ and for every integer $n\ge n_0^p(\d)$ define $
H_n^p(\delta)= H_n(K_\delta^p,p)$. Thus $d_n^p(\d)=d_n(K_\d^p,p)$ and~\eqref{eq:dnp bound} follows from~\eqref{e:def d_n}. Finally, setting $\delta_0^p(X)=\inf\{\d\in (0,\delta(2,p)]: K_\d^p\le 2K(X)\}$, it follows that for every $\d\in (0,\d_0^p(X)]$ we have $K_\d^p\ge 2K(X)$, so that~\eqref{eq:invert} follows from~\eqref{eq:gamma+ base grapg bound for large index}.
\end{proof}

\begin{remark}\label{rem:vanilla KN}
In Remark~\ref{rem:better laplace bound?} we asked whether Theorem~\ref{thm:lower laplace} can be improved so as to yield the estimate
\begin{equation}\label{eq:repeat conjecture}
\|\Delta f\|_{L_p(\F_2^n,X)}\gtrsim_{X,p} k\|f\|_{L_p(\F_2^n,X)}
\end{equation}
 for every $f\in L_p^{\ge k}(\F_2^n,X)$. Here $(X,\|\cdot\|_X)$ is a $K$-convex Banach space and the implied constant is allowed to depend only on $p\in (1,\infty)$ and the $K$-convexity constant $K(X)$. If true, this would yield the following simpler proof of Lemma~\ref{lem:base code}, with better degree bounds. Continuing to use the notation of Lemma~\ref{lem:base code}, we would consider instead the ``vanilla" quotient graph $G$ on $\F_2^n/V_n^\perp$, i.e., the graph in which the number of edges joining two cosets $x+V_n^\perp, y+V_n^\perp$ equals the number of standard hypercube edges joining these two sets divided by $|V_n^\perp|$. The degree of this graph is $n\asymp \log m_n$.  Given a mean-zero $f:\F_2^n/V_n^\perp\to X$ we think of $f$ as being a $V_n^\perp$-invariant function defined on $\F_2^n$, in which case by~\cite[Lem.~3.3]{KN06} we have $f\in L_p^{\ge k_n}(\F_2^n,X)$, where $k_n\asymp n$ is given in~\eqref{eq:C assumptions}. Assuming the validity of~\eqref{eq:repeat conjecture},
\begin{multline}\label{eq:triangle laplace}
 n\|f\|_{L_p(\F_2^n,X)}\lesssim k_n\|f\|_{L_p(\F_2^n,X)}\lesssim_{X,p} \|\Delta f\|_{L_p(\F_2^n,X)}\\=\left\|\sum_{i=1}^n \partial_i f\right\|_{L_p(\F_2^n,X)}\le \sum_{i=1}^n \|\partial_i f\|_{L_p(\F_2^n,X)}\le n^{1-1/p}\left(\sum_{i=1}^n \|\partial_i f\|_{L_p(\F_2^n,X)}^p\right)^{1/p}.
\end{multline}
It follows that
\begin{equation}\label{eq:raise to p in remark}
\frac{1}{2^{2n}}\sum_{(x,y)\in \F_2^n\times \F_2^n}\|f(x)-f(y)\|_X^p\le 2^p\|f\|_{L_p(\F_2^n,X)}^p\stackrel{\eqref{eq:triangle laplace}}{\lesssim_{X,p}} \frac{1}{n}\sum_{i=1}^n \|\partial_i f\|_{L_p(\F_2^n,X)}^p.
\end{equation}
By the definition of the quotient graph $G$, it follows from~\eqref{eq:raise to p in remark} that $\gamma(G,X)\lesssim_{p,X} 1$.  Using Lemma~\ref{lem:half size} we conclude that there exists a regular graph $G'$ with $m_n/2=2^{D_n-1}$ vertices and degree at most a constant multiple of $\log m_n$ such that $\gamma_+(G',X)\lesssim_{p,X} 1$.
\end{remark}

\section{Graph products}\label{sec:products}

The purpose of this section is to recall the definitions of the various graph products that were mentioned in the introduction, and to prove Theorem~\ref{thm:products intro}.

\subsection{Sub-multiplicativity for tensor products}\label{sec:tensor}

The case of tensor products, i.e., part~\eqref{item:product} of Theorem~\ref{thm:products intro}, is very simple, and should mainly serve as warmup for the other parts of Theorem~\ref{thm:products intro}.

\begin{proposition}[Sub-multiplicativity for tensor products]\label{prop:tensor}
Fix $m,n\in \N$. Let $A=(a_{ij})$ be an $m\times m$ symmetric stochastic matrix and let $B=(b_{ij})$ be an $n\times n$ symmetric stochastic matrix. Then every kernel $K:X\times X\to [0,\infty)$ satisfies
\begin{equation}\label{eq:tensor sub equation}
\gamma_+(A\otimes B,K)\le \gamma_+(A,K)\gamma_+(B,K).
\end{equation}
\end{proposition}

\begin{proof}
Fix $f,g:\m\times \n\to X$. Then for every fixed $s,t\in \n$,
\begin{equation}\label{eq:fixed st}
\frac{1}{m^2}\sum_{i=1}^m\sum_{j=1}^mK\left(f(i,s),g(j,t)\right)\le \frac{\gamma_+(A,K)}{m}\sum_{i=1}^m\sum_{j=1}^m a_{ij}K\left(f(i,s),g(j,t)\right).
\end{equation}
Also, for every fixed $i,j\in \m$ we have
\begin{equation}\label{eq:fixed ij}
\frac{1}{n^2}\sum_{s=1}^m\sum_{t=1}^mK\left(f(i,s),g(j,t)\right)\le \frac{\gamma_+(B,K)}{n}\sum_{s=1}^n\sum_{t=1}^n b_{st}K\left(f(i,s),g(j,t)\right).
\end{equation}
Consequently,
\begin{eqnarray}\label{eq:do tensor sub}
&&\nonumber\!\!\!\!\!\!\!\!\!\!\!\!\!\!\!\!\!\!\!\!\!\!\!\!\!\!\!\!\!\!\frac{1}{m^2n^2}\sum_{i=1}^m\sum_{j=1}^m\sum_{s=1}^n\sum_{t=1}^n K\left(f(i,s),g(j,t)\right)=\frac{1}{n^2}\sum_{s=1}^n\sum_{t=1}^n\frac{1}{m^2}\sum_{i=1}^m\sum_{j=1}^mK\left(f(i,s),g(j,t)\right)\\
&\stackrel{\eqref{eq:fixed st}}{\le}& \nonumber\frac{1}{n^2}\sum_{s=1}^n\sum_{t=1}^n\frac{\gamma_+(A,K)}{m}\sum_{i=1}^m\sum_{j=1}^m a_{ij}K\left(f(i,s),g(j,t)\right)\\
&=& \nonumber\frac{\gamma_+(A,K)}{m}\sum_{i=1}^m\sum_{j=1}^m a_{ij}\frac{1}{n^2}\sum_{s=1}^m\sum_{t=1}^mK\left(f(i,s),g(j,t)\right)\\
&\stackrel{\eqref{eq:fixed ij}}{\le}& \nonumber\frac{\gamma_+(A,K)}{m}\sum_{i=1}^m\sum_{j=1}^m a_{ij}\frac{\gamma_+(B,K)}{n}\sum_{s=1}^n\sum_{t=1}^n b_{st}K\left(f(i,s),g(j,t)\right)\\
&=& \frac{\gamma_+(A,K)\gamma_+(B,K)}{mn} \sum_{i=1}^m\sum_{j=1}^m\sum_{s=1}^n\sum_{t=1}^n(A\otimes B)_{ijst}K\left(f(i,s),g(j,t)\right).
\end{eqnarray}
Since~\eqref{eq:do tensor sub} holds for every $f,g:\n\times \m\to X$, \eqref{eq:tensor sub equation} follows.
\end{proof}

This concludes the proof of part~\eqref{item:product} of Theorem~\ref{thm:products intro}. Nevertheless, when the kernel in question is the $p$th power of a norm whose modulus of convexity has power type $p$ it is possible improve Proposition~\ref{prop:tensor} as follows.

\begin{lemma}\label{lem:tesor sup uniformly convex}
Fix $m,n\in \N$ and $p\in [2,\infty)$. Let $A=(a_{ij})$ be an $m\times m$ symmetric stochastic matrix and let $B=(b_{ij})$ be an $n\times n$ symmetric stochastic matrix. Suppose that $(X,\|\cdot\|_X)$ is a Banach space that satisfies the $p$-uniform convexity inequality~\eqref{eq:two point convex}. Then
\begin{equation}\label{eq:improved tensor}
\gamma_+\left(A\otimes B,\|\cdot\|_X^p\right)\le 2^{p-1}\max\left\{\gamma_+\left(A,\|\cdot\|_X^p\right),\left(2^{p-1}-1\right)K_p(X)^p\gamma_+\left(B,\|\cdot\|_X^p\right)\right\}.
\end{equation}
\end{lemma}

\begin{proof} For simplicity of notation write
\begin{equation*}\label{eq:tensor def c}
c\eqdef \frac{1}{\left(2^{p-1}-1\right)K_p(X)^p},
\end{equation*}
and
\begin{equation}\label{eq:def Gamma tensor}
\Gamma\eqdef 2^{p-1}\max\left\{\gamma_+\left(A,\|\cdot\|_X^p\right),\frac{1}{c}\gamma_+\left(B,\|\cdot\|_X^p\right)\right\}.
\end{equation}
Fix $f,g:\m\times \n\to X$. For every $i,j\in \m$ and $s\in \n$ consider the $X$-valued random variable $U_{ij}^s$ which, for every $t\in \m$, takes the value $f(i,s)-g(j,t)$ with probability $b_{st}$. An application of Lemma~\ref{lem:from Ball} with $U=U_{ij}^s$ shows that if for every $j\in \m$ and $s\in \n$ we define
$$
h(j,s)\eqdef \sum_{t=1}^n b_{st}g(j,t),
$$
then for every $i,j\in \m$ and $s\in \n$ we have
\begin{equation}\label{eq:tensor use RV ineq}
\left\|f(i,s)-h(j,s)\right\|_X^p+c\sum_{t=1}^nb_{st}\left\|h(j,s)-g(j,t)\right\|_X^p\le \sum_{t=1}^nb_{st}\left\|f(i,s)-g(j,t)\right\|_X^p.
\end{equation}

By the definition of $\gamma_+\left(A,\|\cdot\|_X^p\right)$,  for every fixed $s\in \n$ we have
\begin{equation}\label{eq:tesor fh}
\frac{1}{m^2}\sum_{i=1}^m\sum_{j=1}^m \left\|f(i,s)-h(j,s)\right\|_X^p\le \frac{\gamma_+\left(A,\|\cdot\|_X^p\right)}{m}\sum_{i=1}^m\sum_{j=1}^m a_{ij}\left\|f(i,s)-h(j,s)\right\|_X^p,
\end{equation}
Similarly, for every fixed $j\in \m$ we have
\begin{equation}\label{eq:tesor hg}
\frac{1}{n^2}\sum_{s=1}^n\sum_{t=1}^n \left\|h(j,s)-g(j,t)\right\|_X^p\le \frac{\gamma_+\left(B,\|\cdot\|_X^p\right)}{n}\sum_{s=1}^n\sum_{t=1}^n b_{st}\left\|h(j,s)-g(j,t)\right\|_X^p.
\end{equation}
By the triangle inequality, for every fixed $i,j\in \m$ and $s\in \n$ we have
\begin{equation}\label{eq:tensor triangle ineq}
\frac{1}{n}\sum_{t=1}^n \left\|f(i,s)-g(j,t)\right\|_X^p\le 2^{p-1}\left\|f(i,s)-h(j,s)\right\|_X^p+\frac{2^{p-1}}{n}\sum_{t=1}^n \left\|h(j,s)-g(j,t)\right\|_X^p.
\end{equation}
By averaging~\eqref{eq:tensor triangle ineq} over $i,j\in \m$ and $s\in \n$ we deduce that
\begin{multline}\label{eq:tesor sum triangle}
\frac{1}{m^2n^2}\sum_{i=1}^m\sum_{j=1}^m\sum_{s=1}^n\sum_{t=1}^n\left\|f(i,s)-g(j,t)\right\|_X^p\\
\le \frac{2^{p-1}}{n}\sum_{s=1}^n\frac{1}{m^2}\sum_{i=1}^m\sum_{j=1}^m\left\|f(i,s)-h(j,s)\right\|_X^p+\frac{2^{p-1}}{m}\sum_{j=1}^m\frac{1}{n^2}\sum_{s=1}^n\sum_{t=1}^n \left\|h(j,s)-g(j,t)\right\|_X^p.
\end{multline}
By substituting~\eqref{eq:tesor fh} and~\eqref{eq:tesor hg} into~\eqref{eq:tesor sum triangle} we obtain the estimate
\begin{eqnarray}\label{eq:all together tensor}
&&\nonumber\!\!\!\!\!\!\!\!\!\!\!\!\!\!\!\!\!\!\!\!\!\!\!\!\frac{1}{m^2n^2}\sum_{i=1}^m\sum_{j=1}^m\sum_{s=1}^n\sum_{t=1}^n\left\|f(i,s)-g(j,t)\right\|_X^p\\ \nonumber&\le& \frac{2^{p-1}\gamma_+\left(A,\|\cdot\|_X^p\right)}{mn}\sum_{i=1}^m\sum_{j=1}^m  \sum_{s=1}^n a_{ij}\left\|f(i,s)-h(j,s)\right\|_X^p\\&&+\nonumber\frac{2^{p-1}\gamma_+\left(B,\|\cdot\|_X^p\right)}{mn}\sum_{s=1}^n\sum_{t=1}^n \sum_{j=1}^m b_{st}\left\|h(j,s)-g(j,t)\right\|_X^p\\\nonumber
&\stackrel{\eqref{eq:def Gamma tensor}}{\le}&\frac{\Gamma}{mn}\sum_{i=1}^m\sum_{j=1}^n a_{ij}\sum_{s=1}^n\left(\left\|f(i,s)-h(j,s)\right\|_X^p+c\sum_{t=1}^n b_{st}\left\|h(j,s)-g(j,t)\right\|_X^p\right)\\
&\stackrel{\eqref{eq:tensor use RV ineq}}{\le}& \frac{\Gamma}{mn}\sum_{i=1}^m\sum_{j=1}^n \sum_{s=1}^n\sum_{t=1}^n a_{ij}b_{st} \left\|f(i,s)-g(j,t)\right\|_X^p.
\end{eqnarray}
Since~\eqref{eq:all together tensor} holds for every $f,g:\m\times \n\to X$, \eqref{eq:improved tensor} follows.
\end{proof}

\subsection{Sub-multiplicativity for the zigzag product}\label{sec:zigzag proof} Here we prove Theorem~\ref{thm:sub}. Before doing so, we need to recall the definition of the zigzag product of Reingold, Vadhan and Wigderson~\cite{RVW}. The notation used below, which lends itself well to the ensuing proof of Theorem~\ref{thm:sub}, was suggested to us by K. Ball.

Fix $n_1,d_1,d_2\in \N$. Suppose that $G_1=(V_1,E_1)$ is an $n_1$-vertex graph which is $d_1$-regular and that $G_2=(V_2,E_2)$ is a $d_1$-vertex graph which is $d_2$-regular. Since the number of vertices in $G_2$ is the same as the degree of $G_1$, we can identify $V_2$ with the edges emanating from a given vertex $u\in V_1$. Formally, we fix for every $u\in V_1$ a bijection
\begin{equation}\label{eq:pi label}
\pi_u:\{(u,v)\in \{u\}\times V_1: \; (u,v)\in E_1\} \to V_2.
\end{equation}
Moreover, we fix for every $a\in V_2$ a bijection between $\{1,\ldots,d_2\}$ and the multiset of the vertices adjacent to $a$ in $G_2$, i.e.,
\begin{equation}\label{eq:kappa label}
\kappa_a:\{1,\ldots,d_2\} \to \{b\in V_2:\; (a,b)\in E_2\}.
\end{equation}

The zigzag product $G_1\oz G_2$ is the graph whose vertices are $V_1\times V_2$ and the ordered pair $((u,a),(v,b))\in V_1\times V_2$ is added to $E(G_1\oz G_2)$  whenever there exist $i,j\in \{1,\ldots,d_2\}$ satisfying
\begin{equation}\label{eq:def zigzag edges} (u,v)\in E_1\quad
\text{and}\quad a=\kappa_{\pi_u(u,v)}(i)
\quad\text{and} \quad b=\kappa_{\pi_v(v,u)}(j) .\end{equation}
Thus,
$$
E(G_1\oz G_2)((u,a),(v,b))\eqdef \sum_{i=1}^{d_2}\sum_{j=1}^{d_2} E_1(u,v)\cdot \1_{\{a=\kappa_{\pi_u(u,v)}(i)\}}\cdot \1_{\{b=\kappa_{\pi_v(u,v)}(j)\}}.
$$

The schematic description of this construction is as follows. Think of the vertex set of $G_1\oz G_2$ as a disjoint union of ``clouds" which are copies of $V_2=\{1,\ldots,d_1\}$ indexed by $V_1$. Thus $(u,a)$ is the point indexed by $a$ in the cloud labeled by $u$. Every edge $((u,a),(v,b))$ of $G_1\oz G_2$ is the result of a three step walk: a ``zig" step in $G_2$ from $a$ to $\pi_u(u,v)$  in $u$'s cloud, a ``zag" step in $G_1$ from $u$'s cloud to $v$'s cloud along the edge $(u,v)$ and a final ``zig" step in $G_2$ from $\pi_v(u,v)$ to $b$ in $v$'s cloud.  The  zigzag product is illustrated in Figure~\ref{fig:zigzag-op}. The number of vertices of $G_1\oz G_2$ is $n_1d_1$ and its degree is $d_2^2$.
\begin{figure}[ht]
\begin{center}
\includegraphics{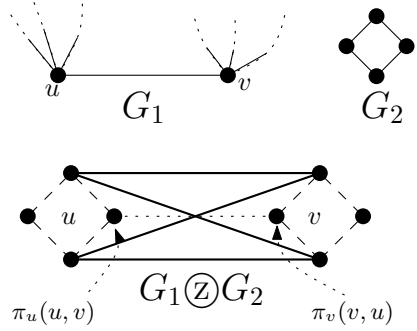}%
\caption{
A schematic illustration of the zigzag product. The upper part of the figure depicts \emph{part} of a 4-regular graph  $G_1$,
and a 4-vertex cycle $G_2$. The bottom part of the figure depicts the edges of  the zigzag product %$G_1\zigzag G_2$
between $u$'s cloud and $v$'s cloud.
The original edges of $G_1$ and $G_2$ are drawn as dotted and dashed lines, respectively.}
\label{fig:zigzag-op}
\end{center}
\end{figure}
The zigzag product depends on the choice of labels $\{\pi_u\}_{u\in V_1}$, and in fact different labels of the same graphs can produce non-isomorphic products\footnote{The labels $\{\kappa_a\}_{a\in V_2}$ do not affect the structure of the zigzag product but they are useful in the subsequent analysis.}. However,
 the estimates below will be independent of the actual choice of the labeling, so while our notation should formally depend on the labeling, we will drop its explicit mention for the sake of simplicity.

 \begin{proof}[Proof of Theorem~\ref{thm:sub}]
Fix $f,g:V_1\times V_2\to X$. The definition of $\gamma_+(G_1,K)$ implies that for all $a,b\in V_2$ we have
\begin{equation}\label{zigzag1-b}
\frac{1}{n_1^2}\sum_{(u,v)\in V_1\times V_1}K\left(f(u,a),g(v,b)\right)
 \le \frac{\gamma_+(G_1,K)}{n_1d_1}\sum_{(u,v)\in E_1} K\left(f(u,a),g\left(v,b\right)\right).
\end{equation}
Hence,
\begin{eqnarray}\label{zigzag2-b}
&&\!\!\!\!\!\!\!\!\!\!\!\!\!\!\!\!\!\!\!\!\!\!\!\!\!\!\!\!\!\!\!\!\!\!\!\!\!\!\!\!\!\!\!\!\!\!\!\!\!\!\!\!\!\!\!\nonumber\frac{1}{|V_1\times V_2|^2}\sum_{((u,a),(v,b))\in (V_1\times V_2)\times (V_1\times V_2)}K(f(u,a),g(v,b))
\\&=&\nonumber
\frac{1}{d_1^2}\sum_{(a,b)\in V_2\times V_2}\frac{1}{n_1^2}\sum_{(u,v)\in V_1\times V_1}K(f(u,a),g(v,b))\\
&\stackrel{\eqref{zigzag1-b}}{\le}&
\frac{\gamma_+(G_1,K)}{n_1d_1^3}\sum_{(a,b)\in V_2\times V_2}\sum_{(u,v)\in  E_1}K\left(f(u,a),g\left(v,b\right)\right).
\end{eqnarray}
Next,  fix $u\in V_1$ and $b\in V_2$, and define $\phi_{b}^u:V_2\to X$ as follows. Recalling~\eqref{eq:pi label}, for $c\in V_2$ write $\pi_u^{-1}(c)=(u,v)\in E_1$ for some $v\in V_1$, and define $\phi_b^u(c)=g(v,b)$. The definition of $\gamma_+(G_2,K)$ implies that
\begin{multline}\label{zigzag3-b}
\frac{1}{d_1^2}\sum_{a\in V_2} \sum_{\substack{v\in V_1\\ (u,v)\in E_1}} K\left(f(u,a),g\left(v,b\right)\right)=\frac{1}{d_1^2}\sum_{a\in V_2}\sum_{c\in V_2} K\left(f(u,a),\phi_b^u(c)\right)
\\ \le
\frac{\gamma_+(G_2,K)}{d_1d_2}
\sum_{\substack{v\in V_1\\ (u,v)\in E_1}}\sum_{i=1 }^{d_2} K\left(f\left(u,\kappa_{\pi_u(u,v)}(i)\right),g\left(v,b\right)\right),
\end{multline}

Summing~\eqref{zigzag3-b} over $u\in V_1$ and $b\in V_2$ and substituting the  resulting expression into~\eqref{zigzag2-b} yields the bound
\begin{multline}\label{zigzag4-b}
\frac{1}{|V_1\times V_2|^2}\sum_{((u,a),(v,b))\in (V_1\times V_2)\times (V_1\times V_2)}K(f(u,a),g(v,b))\\\le \frac{\gamma_+(G_1,K)\gamma_+(G_2,K)}{n_1d_1^2d_2}\sum_{v\in V_1}\sum_{i=1}^{d_2}\sum_{\substack{u\in V_1\\ (u,v)\in E_1}} \sum_{b\in V_2}
K\left(f\left(u,\kappa_{\pi_u(u,v)}(i)\right),g\left(v,b\right)\right).
\end{multline}
Fix $i\in \{1,\ldots,d_2\} $  and $v\in V_1$, and define $\psi_i^v:V_2\to X$ as follows. For $c\in V_2$ write $\pi_v^{-1}(c)=(v,u)$ for some $u\in V_1$ such that $(v,u)\in E_1$ (equivalently, $(u,v)\in E_1$), and set $\psi_i^v(c)=f\left(u,\kappa_{\pi_u(u,v)}(i)\right)$. Another application of the definition of $\gamma_+(G_2,K)$ implies that
\begin{multline}\label{zigzag5-b}
\frac{1}{d_1^2}\sum_{\substack{u\in V_1\\ (u,v)\in E_1}} \sum_{b\in V_2}
K\left(f\left(u,\kappa_{\pi_u(u,v)}(i)\right),g\left(v,b\right)\right)=\frac{1}{d_1^2} \sum_{c\in V_2}\sum_{b\in V_2} K\left(\psi_i^v(c),g(v,b)\right)
\\
\le
 \frac{\gamma_+(G_2,K)}{d_1d_2}
\sum_{\substack{u\in V_1\\ (u,v)\in E_1}} \sum_{j=1}^{d_2}
K\left(f\left(u,\kappa_{\pi_u(u,v)}(i)\right),g\left(v,\kappa_{\pi_v(v,u)}(j)\right)\right).
\end{multline}
Summing~\eqref{zigzag5-b} over $v\in V_1$ and $i\in \{1,\ldots,d_2\}$,  and combining the resulting inequality with~\eqref{zigzag4-b}, yields the bound
\begin{eqnarray} \label{zigzag6-b}
&&\nonumber\!\!\!\!\!\!\!\!\!\!\!\!\!\!\!\!\!\frac{1}{|V_1\times V_2|^2}\sum_{((u,a),(v,b))\in (V_1\times V_2)\times (V_1\times V_2)}K(f(u,a),g(v,b))\\
&{\le}&
\frac{\gamma_+(G_1,K)\gamma_+(G_2,K)^2}{n_1d_1d_2^2}
\sum_{(u,v)\in  E_1} \sum_{i=1}^{d_2}\sum_{j=1}^{d_2}
              K\left(f\left(u,\kappa_{\pi_u(u,v)}(i)\right),g
              \left(v,\kappa_{\pi_v(v,u)}(j)\right)\right)
\nonumber\\ &\stackrel{\eqref{eq:def zigzag edges}}{=}&
\frac{\gamma_+(G_1,K)\gamma_+(G_2,K)^2}{n_1d_1d_2^2}
{\!\!\!\!\!\!\!\!\!}
\sum_{((u,a),(v,b)) E\left(G_1\scriptsize{\textcircled z} G_2\right)}
{\!\!\!\!\!\!\!}
K\left(f\left(u,a\right),g\left(v,b\right)\right).
\end{eqnarray}
  Since~\eqref{zigzag6-b} holds for every $f,g:V_1\times V_2\to X$, the proof of Theorem~\ref{thm:sub} is complete.
 \end{proof}

\subsection{Sub-multiplicativity for replacement products}\label{sec:replacement} Here we continue to use the notation of Section~\ref{sec:zigzag proof}. Specifically, we fix   $n_1,d_1,d_2\in \N$ and suppose that $G_1=(V_1,E_1)$ is an $n_1$-vertex graph which is $d_1$-regular and that $G_2=(V_2,E_2)$ is a $d_1$-vertex graph which is $d_2$-regular. We also identify $V_1=\{1,\ldots,n_1\}$ and $V_2=\{1,\ldots,d_1\}$, and for every $u\in V_1$ and $a\in V_2$ we fix a bijections $\pi_u$ and $\kappa_a$ as in~\eqref{eq:pi label} and~\eqref{eq:kappa label}, respectively. The {\em replacement product}~\cite{Gromov-filling,RVW} of $G_1$ and $G_2$, denoted $G_1\circr G_2$, is the graph with vertex set $\{1,\ldots,n_1\}\times \{1,\ldots,d_1\}$ in which the ordered pair $((u,i), (v,j))\in \{1,\ldots,n_1\}\times \{1,\ldots,d_1\}$ is added to $E(G_1\circr G_2)$ if and only if either $u=v$ and $(i,j)\in E_2$ or $(u,v)\in E_1$ and $i=\pi_u(u,v)$ and $j=\pi_v(v,u)$. Thus,
$$
E(G_1\circr G_2)((u,i), (v,j))\eqdef  E_2(i,j)\cdot \1_{\{u=v\}}+E_1(u,v)\cdot \1_{\{i=\pi_u(u,v)\}}\cdot \1_{\{j=\pi_v(v,u)\}}.
$$
This definition makes $G_1\circr G_2$ be a $(d_2+1)$-regular graph.

The following lemma shows that the ``discrete gradient" associated to $G_1\oz G_2$ is dominated by $3^{p-1}(d_2+1)$ times the ``discrete gradient" associated to $G_1\circr G_2$.

\begin{lemma}\label{lem:zigzag gradient bigger than replacement}
Fix $p\in [1,\infty)$, a metric space $(X,d_X)$ and $n_1,d_1,d_2\in \N$. Suppose that $G_1=(V_1,E_1)$ is an $n_1$-vertex graph which is $d_1$-regular and that $G_2=(V_2,E_2)$ is a $d_1$-vertex graph which is $d_2$-regular. Then every $f,g:V_1\times V_2\to X$ satisfy
\begin{multline}\label{eq:gradient domination replacement}
\frac{1}{\left|E\left(G_1\scriptsize{\textcircled z} G_2\right)\right|}
\sum_{((u,a),(v,b))\in  E\left(G_1\scriptsize{\textcircled z} G_2\right)}
d_X\left(f\left(u,a\right),g\left(v,b\right)\right)^p\\\le
\frac{3^{p-1}(d_2+1)}{\left|E\left(G_1\scriptsize{\textcircled r} G_2\right)\right|}
\sum_{((u,a),(v,b))\in E\left(G_1\scriptsize{\textcircled r} G_2\right)}
d_X\left(f\left(u,a\right),g\left(v,b\right)\right)^p.
\end{multline}
\end{lemma}

Before proving Lemma~\ref{lem:zigzag gradient bigger than replacement} we record two of its immediate (yet useful) consequences.
\begin{corollary}\label{cor:poincare domination replacement}
Under the assumptions of Lemma~\ref{lem:zigzag gradient bigger than replacement} we have
$$
\gamma_+\left(G_1\circr G_2,d_X^p\right)\le 3^{p-1}(d_2+1)\cdot \gamma_+\left(G_1\oz G_2,d_X^p\right).
$$
\end{corollary}

Now, part~\eqref{item:replacement product} of Theorem~\ref{thm:products intro} corresponds to the case $p=2$ of the following combination of Theorem~\ref{thm:sub} and Corollary~\ref{cor:poincare domination replacement}.

\begin{corollary}\label{cor:replacement}
Under the assumptions of Lemma~\ref{lem:zigzag gradient bigger than replacement} we have
$$
\gamma_+\left(G_1\circr G_2,d_X^p\right)\le 3^{p-1}(d_2+1)\cdot\gamma_+\left(G_1,d_X^p\right)\cdot \gamma_+\left(G_2,d_X^p\right)^2.
$$
\end{corollary}

\begin{proof}[Proof of Lemma~\ref{lem:zigzag gradient bigger than replacement}] Fix $((u,a),(v,b))\in E(G_1\oz G_2)$. Thus by the definition of the zigzag product we have $(u,v)\in E_1$ and $(a,\pi_u(u,v)), (b,\pi_v(v,u))\in E_2$. Observe that the following three pairs are edges of $G_1\circr G_2$.
$$
\left((u,a),(u,\pi_u(u,v)\right),\ \left((u,\pi_u(u,v),(v,\pi_v(v,u)\right),\ \left((v,\pi_v(v,u),(v,b)\right).
$$

By the triangle inequality,
\begin{multline}\label{eq:triangle 3 term replacement}
d_X\left(f(u,a),g(v,b)\right)^p\le 3^{p-1}\Big(d_X\left(f(u,a),g(u,\pi_u(u,v))\right)^p\\+d_X\left(g(u,\pi_u(u,v)),f(v,\pi_v(v,u))\right)^p
+d_X\left(f(v,\pi_v(v,u)),g(v,b))\right)^p\Big).
\end{multline}
Therefore,
\begin{eqnarray}\label{eq:zigzag decomposed into 3}
&&\nonumber\!\!\!\!\!\!\!\!\!\!\!\!\!\!\!\!\!\!\!\!\!\!\!\!\!\!\!\!\!\!\!\!\!\!\frac{1}{\left|E\left(G_1\scriptsize{\textcircled z} G_2\right)\right|}
\sum_{((u,a),(v,b))\in  E\left(G_1\scriptsize{\textcircled z} G_2\right)}
d_X\left(f\left(u,a\right),g\left(v,b\right)\right)^p\\\nonumber
&=&\frac{1}{n_1d_1d_2^2}\sum_{(u,v)\in  E_1}\sum_{\substack{a\in V_2\\ (a,\pi_u(u,v))\in E_2}}\sum_{\substack{b\in V_2\\ (b,\pi_v(v,u))\in E_2}}d_X\left(f\left(u,a\right),g\left(v,b\right)\right)^p\\&\stackrel{\eqref{eq:triangle 3 term replacement}}{\le}& \frac{3^{p-1}}{n_1d_1d_2^2}\left(S_1+S_2+S_3\right),
\end{eqnarray}
where the quantities $S_1,S_2,S_3$ are defined as follows.
\begin{multline*}
S_1\eqdef \sum_{(u,v)\in  E_1}\sum_{\substack{a\in V_2\\ (a,\pi_u(u,v))\in E_2}}\sum_{\substack{b\in V_2\\ (b,\pi_v(v,u))\in E_2}}d_X\left(f(u,a),g(u,\pi_u(u,v))\right)^p\\
=d_2\sum_{(u,v)\in  E_1}\sum_{\substack{a\in V_2\\ (a,\pi_u(u,v))\in E_2}}d_X\left(f(u,a),g(u,\pi_u(u,v))\right)^p,
\end{multline*}
\begin{multline*}
S_2\eqdef \sum_{(u,v)\in  E_1}\sum_{\substack{a\in V_2\\ (a,\pi_u(u,v))\in E_2}}\sum_{\substack{b\in V_2\\ (b,\pi_v(v,u))\in E_2}}d_X\left(g(u,\pi_u(u,v)),f(v,\pi_v(v,u))\right)^p\\
=d_2^2\sum_{(u,v)\in  E_1}d_X\left(g(u,\pi_u(u,v)),f(v,\pi_v(v,u))\right)^p,
\end{multline*}
\begin{multline*}
S_3\eqdef \sum_{(u,v)\in E_1}\sum_{\substack{a\in V_2\\ (a,\pi_u(u,v))\in E_2}}\sum_{\substack{b\in V_2\\ (b,\pi_v(v,u))\in E_2}}d_X\left(f(v,\pi_v(v,u)),g(v,b))\right)^p\\
=d_2\sum_{(u,v)\in E_1}\sum_{\substack{b\in V_2\\ (b,\pi_v(v,u))\}\in E_2}}d_X\left(f(v,\pi_v(v,u)),g(v,b))\right)^p.
\end{multline*}
By the definition of the replacement product we have
\begin{eqnarray}\label{eq:sum of three S}
&&\!\!\!\!\!\!\!\!\!\!\!\!\!\!\!\!\!\!\!\!\!\!\nonumber S_1+S_2+S_3\\\nonumber&=&d_2\sum_{u\in V_1}\sum_{(i,j)\in E_2}d_X(f(u,i),g(u,j))^p+d_2^2\sum_{(u,v)\in E_1}d_X\left(g(u,\pi_u(u,v)),f(v,\pi_v(v,u))\right)^p
\\&\le& d_2^2\sum_{((u,i),(v,j))\in E\left(G_1\scriptsize{\textcircled r} G_2\right)} d_X\left(f(u,i),g(v,j)\right)^p.
\end{eqnarray}
Recalling that $\left|E(G_1\circr G_2)\right|=n_1d_1(d_2+1)$, the desired estimate~\eqref{eq:gradient domination replacement} is now a consequence of~\eqref{eq:zigzag decomposed into 3} and~\eqref{eq:sum of three S}.
\end{proof}

The {\em balanced replacement product} of $G_1$ and $G_2$, denoted $G_1\ob G_2$, is a useful variant of $G_1\circr G_2$ that was introduced in~\cite{RVW}. The vertex set of $G_1\ob G_2$ is still $\{1,\ldots,n_1\}\times \{1,\ldots,d_1\}$, but the edges of $G_1\ob G_2$ are now given by
\begin{multline*}
\forall ((u,i),(v,j))\in \{1,\ldots,n_1\}\times \{1,\ldots,d_1\},\\  E(G_1\ob G_2)((u,i), (v,j))\eqdef  E_2(i,j)\cdot \1_{\{u=v\}}+d_2E_1(u,v)\cdot \1_{\{i=\pi_u(u,v)\}}\cdot \1_{\{j=\pi_v(v,u)\}}.
\end{multline*}
This definition makes $G_1\ob G_2$ be a $2d_2$-regular graph.

Arguing analogously to the proof of Lemma~\ref{lem:zigzag gradient bigger than replacement}, we have the following statements.
\begin{lemma}\label{lem:zigzag gradient bigger than balanced replacement}
Fix $p\in [1,\infty)$, a metric space $(X,d_X)$ and $n_1,d_1,d_2\in \N$. Suppose that $G_1=(V_1,E_1)$ is an $n_1$-vertex graph which is $d_1$-regular and that $G_2=(V_2,E_2)$ is a $d_1$-vertex graph which is $d_2$-regular. Then every $f,g:V_1\times V_2\to X$ satisfy
\begin{multline}\label{eq:gradient domination balanced replacement}
\frac{1}{\left|E\left(G_1\scriptsize{\textcircled z} G_2\right)\right|}
\sum_{((u,a),(v,b))\in  E\left(G_1\scriptsize{\textcircled z} G_2\right)}
d_X\left(f\left(u,a\right),g\left(v,b\right)\right)^p\\\le
\frac{2\cdot 3^{p-1}}{\left|E\left(G_1\scriptsize{\textcircled b} G_2\right)\right|}
\sum_{((u,a),(v,b))\in E\left(G_1\scriptsize{\textcircled b} G_2\right)}
d_X\left(f\left(u,a\right),g\left(v,b\right)\right)^p.
\end{multline}
\end{lemma}
\begin{corollary}\label{cor:poincare domination balanced replacement}
Under the assumptions of Lemma~\ref{lem:zigzag gradient bigger than balanced replacement} we have
$$
\gamma_+\left(G_1\ob  G_2,d_X^p\right)\le 2\cdot 3^{p-1}\cdot \gamma_+\left(G_1\oz G_2,d_X^p\right).
$$
\end{corollary}
Part~\eqref{item:balanced replacement product} of Theorem~\ref{thm:products intro} corresponds to the case $p=2$ of the following combination of Theorem~\ref{thm:sub} and Corollary~\ref{cor:poincare domination balanced replacement}.
\begin{corollary}\label{cor:balanced replacement}
Under the assumptions of Lemma~\ref{lem:zigzag gradient bigger than balanced replacement} we have
$$
\gamma_+\left(G_1\ob  G_2,d_X^p\right)\le 2\cdot 3^{p-1}\cdot\gamma_+\left(G_1,d_X^p\right)\cdot \gamma_+\left(G_2,d_X^p\right)^2.
$$
\end{corollary}

\begin{remark}\label{rem:alon balanced}
An analysis of the behavior of spectral gaps under the balanced replacement product was previously performed in a non-Euclidean setting by Alon, Schwartz and Shapira~\cite{ASS}. Specifically, \cite[Thm.~1.3]{ASS} estimates the edge expansion of $G_1\ob G_2$ in terms of the edge expansion of $G_1$ and $G_2$ via a direct combinatorial argument. The edge expansion of a graph $G$ is equivalent up to universal constant factors to $\gamma(G,|\cdot|)$, where $|\cdot|$ is the standard absolute value on $\R$. The corresponding bound arising from Corollary~\ref{cor:balanced replacement} is better than the bound of~\cite[Thm.~1.3]{ASS} in terms of constant factors.
\end{remark}

\subsection{Sub-multiplicativity for derandomized squaring}\label{sec:squaring} Here we continue to use the notation of Section~\ref{sec:zigzag proof} and Section~\ref{sec:replacement}. The derandomized squaring of $G_1$ and $G_2$, as introduced by Rozenman and Vadhan in~\cite{RV05} and denoted $G_1\os G_2$, is defined as follows. The vertex set of $G_1\os G_2$ is $V_1=\{1,\ldots,n_1\}$, and the edges $E(G_1\os G_2)$ are given by
$$
\forall (u,v)\in V_1\times V_1,\quad E(G_1\os G_2)(u,v)\eqdef \sum_{w\in V_1} E_1(w,u)E_1(w,v)E_2\left(\pi_w(w,u),\pi_w(w,v)\right).
$$
Thus, given $(u,v)\in V_1\times V_1$, we add a copy of $(u,v)$ to $E(G_1\os G_2)$ for every $(i,j)\in E_2$ such that there exists $w\in V_1$ with $(w,u),(w,v)\in E_1$ and $\pi_w(w,u)=i, \pi_w(w,v)=j$. With this definition one checks that $G_1\os G_2$ is $d_1d_2$-regular.

The following proposition corresponds to part~\eqref{item:deradomized squaring} of Theorem~\ref{thm:products intro}.

\begin{proposition}\label{prop:derandomized squaring}
Fix  $n_1,d_1,d_2\in \N$ and suppose that $G_1=(V_1,E_1)$ is an $n_1$-vertex graph which is $d_1$-regular and that $G_2=(V_2,E_2)$ is a $d_1$-vertex graph which is $d_2$-regular. Then for every kernel $K:X\times X\to [0,\infty)$ we have
\begin{equation}\label{eq:gamma+ squaring}
\gamma_+\left(G_1\os G_2,K\right)\le \gamma_+\left(G_1^2,K\right)\gamma_+\left(G_2,K\right).
\end{equation}
\end{proposition}
In~\cite{RV05} Rozenman and Vadhan used a spectral argument to prove the Euclidean case of~\eqref{eq:gamma+ squaring}, i.e., the special case of~\eqref{eq:gamma+ squaring} when $K:\R\times \R\to [0,\infty)$ is given by $K(x,y)=(x-y)^2$.
\begin{proof}[Proof of Proposition~\ref{prop:derandomized squaring}]
Fix $f,g:V_1\to X$. The definition of $\gamma_+\left(G_1^2,K\right)$ implies that
\begin{multline}\label{squared graph use}
\frac{1}{n_1^2}\sum_{(u,v)\in V_1\times V_1}K(f(u),f(v))\le \frac{\gamma_+\left(G_1^2,K\right)}{n_1d_1^2}\sum_{(u,v)\in E(G_1^2)}K(f(u),f(v))\\=\frac{\gamma_+\left(G_1^2,K\right)}{n_1d_1^2}\sum_{w\in V_1}\sum_{(u,w)\in E_1}\sum_{(w,v)\in E_1}K\left(f(u),g(v)\right).
\end{multline}
For every fixed $w\in V_1$ define $\phi^w,\psi^w:V_2\to X$ as follows. For $i,j\in V_2$ consider the unique vertices $u,v\in V_1$ such that $\pi_w(w,u)=i$ and $\pi_w(w,v)=j$, and define $\phi^w(i)=f(u)$ and $\psi^w(j)=g(v)$. The definition of $\gamma_+(G_2,K)$ implies that
\begin{eqnarray}\label{eq:square sub E2}
&&\nonumber\!\!\!\!\!\!\!\!\!\!\!\!\!\!\!\!\!\!\!\!\!\frac{1}{d_1^2}\sum_{(u,w)\in E_1}\sum_{(w,v)\in E_1}K\left(f(u),g(v)\right)=\frac{1}{d_1^2}\sum_{(i,j)\in V_2\times V_2}K\left(\phi^w(i),\psi^w(j)\right)\\&\le& \nonumber\frac{\gamma_+(G_2,K)}{d_1d_2}\sum_{(i,j)\in E_2}K\left(\phi^w(i),\psi^w(j)\right)\\&=&\frac{\gamma_+(G_2,K)}{d_1d_2}\sum_{(u,w)\in E_1}\sum_{(w,v)\in E_1}E_2\left(\pi_w(w,u),\pi_w(w,v)\right)K\left(f(u),g(v)\right).
\end{eqnarray}
The definition of $G_1\os G_2$ in combination with~\eqref{squared graph use} and~\eqref{eq:square sub E2} now yields the estimate
\begin{align*}
&\frac{1}{n_1^2}\sum_{(u,v)\in V_1\times V_1}K(f(u),f(v))\\&\le \frac{\gamma_+\left(G_1^2,K\right)\gamma_+(G_2,K)}{n_1d_1d_2}\sum_{w\in V_1}\sum_{(u,w)\in E_1}\sum_{(w,v)\in E_1}E_2\left(\pi_w(w,u),\pi_w(w,v)\right)K\left(f(u),g(v)\right)\\&= \frac{\gamma_+\left(G_1^2,K\right)\gamma_+(G_2,K)}{n_1d_1d_2}\sum_{(x,y)\in E\left(G_1\scriptsize{\textcircled s} G_2\right)} K(f(x),g(y)).\qedhere
\end{align*}
\end{proof}

\section{Counterexamples}\label{sec:counterexample}

\subsection{Expander families need not embed coarsely into each other}\label{two families} As was mentioned in the introduction, it is an open question whether every classical (i.e., Euclidean) expander graph family is also a super-expander. Here we rule out the most obvious approach towards
such a result: to embed coarsely any expander family in any other expander family. Formally, given two families of metric spaces $\mathscr X,\mathscr Y$, we say that $\mathscr X$ admits a coarse embedding into $\mathscr Y$ if there exist non-decreasing $\alpha,\beta:[0,\infty)\to [0,\infty)$ satisfying $\lim_{t\to \infty}\alpha(t)=\infty$ such that for every $(X,d_X)\in \mathscr X$ there exists $(Y,d_Y)\in \mathscr Y$ and a mapping $f:X\to Y$ that satisfies
$$
\forall\, x,y\in X,\quad \alpha\left(d_X(x,y)\right)\le d_Y(f(x),f(y))\le \beta\left(d_X(x,y)\right).
$$
This condition clearly implies that $\alpha(0)=0$, and for notational convenience we also assume without loss of generality that $\beta(0)=0$.

Let $\c$ denote the set of all increasing sub-additive functions $\omega:[0,\infty)\to [0,\infty)$ with $\omega(0)=0$. If $(X,d_X)$ is a metric space and $\omega\in \c$ then $(X,\omega\circ d_X)$ is also a metric space, known as the {\em metric transform} of $(X,d_X)$ by $\omega$.

In what follows, given a connected graph $G=(V,E)$, the geodesic metric induced by $G$ on $V$ will be denoted $d_G$. Recall that a sequence of graphs $\{G_n\}_{n=1}^\infty$ is called a constant degree expander sequence if there exists $d\in \N$ such that each $G_n$ is $d$-regular and $\sup_{n\in \N} \lambda(G_n)<1$. The purpose of this section is to prove the following result.

\begin{theorem}\label{thm:two families}
There exist two constant degree expander sequences $\{G_i\}_{i=1}^\infty$ and $\{H_i\}_{i=1}^\infty$ such that  $\{(V(H_i),d_{H_i})\}_{i=1}^\infty$ does not admit a coarse embedding into the family of metric spaces $\{(V(G_i),\omega\circ d_{G_i}):\ (i,\omega)\in \N\times \c\}$.
\end{theorem}

\begin{proof}
It is well known (see e.g.~\cite{LPS,Mar88}) that there exists $c\in (0,\infty)$, an integer $d\ge 3$, and a sequence of $d$-regular expanders $\{G_i\}_{i=1}^\infty$ such that if we set $n_i= |V(G_i)|$ then $\{n_i\}_{i=1}^\infty$ is strictly increasing and each $G_i$ has girth at least $4c\log n_i$. By adjusting $c$ to be a smaller constant if necessary (as we may), we assume below that
\begin{equation}\label{eq: ell small assumption}
 c\log n_i < \frac{n_i}{2(d+1)^{2c\log n_i}}.
\end{equation}
We also assume throughout the ensuing argument that $c\log n_i>7$ for all $i\in \N$.

The desired expander sequence $\{H_i\}_{i=1}^\infty$  will be constructed by modifying $\{G_i\}_{i=1}^\infty$ so as to contain sufficiently many short cycles. Specifically, fix $i\in \N$ and write $G_i=(V_i,E_i)$. We will construct $H_i=(V_i,F_i)$ with $F_i\supsetneq E_i$, i.e., $H_i$ will be a graph with the same vertices as $G_i$ but with additional edges. The construction will ensure that
\begin{equation}\label{eq:diam Hi}
\diam(H_i)\ge \frac{c}{2}\log n_i.
\end{equation}
(Here, and in what follows, diameters of graphs are always understood to be with respect to their shortest-path metric.) We will also ensure that for every integer $h\in [3, c\log n_i]$ the graph $H_i$ contains a cycle of length $h$ which is embedded isometrically into $(H_i,d_{H_i})$, i.e., there exist $x_1,\ldots,x_h\in V_i$ such that $d_{H_i}(x_a,x_b)=\min\{|a-b|,h-|a-b|\}$ for every $a,b\in \{1,\ldots,h\}$, and $\{x_1,x_2\},\{x_2,x_3\},\ldots,\{x_{h-1},x_h\},\{x_h,x_1\}\in F_i$.

Set
\begin{equation}\label{eq:def ell two seq}
\ell \eqdef \left\lfloor c\log n_i\right\rfloor.
\end{equation}
We will  define inductively sets of edges $E=F^0\subsetneq F^1\subsetneq\ldots\subsetneq F^\ell$ with $|F_j\setminus F_{j-1}|=1$ for all $j\in \{1,\ldots,\ell\}$. Fix $j\in \{0,\ldots,\ell-1\}$ and assume inductively that $F^j$ has already been defined so that the graph $$G_i^j\eqdef (V_i,F^j)$$ has maximal degree at most $d+1$. Write
$$
M_j\eqdef \left\{u\in V_i:\ \exists\, e\in F^j\setminus E,\ u\in e\right\} =\bigcup_{e\in F^j\setminus E} e.
$$
Thus $|M_j|\le 2j$. Hence, if we set
$$
D_j\eqdef\left\{u\in V_i:\ d_{G_i^j}(u,M_j)\le 2c\log n_i\right\},
$$
then
$$
|D_j|\le 2j (d+1)^{2c\log n_i}\le 2\ell (d+1)^{2c\log n_i}\stackrel{\eqref{eq: ell small assumption}\wedge \eqref{eq:def ell two seq}}{<}n_i.
$$
Therefore $V\setminus D_j\neq \emptyset$. Choose an arbitrary vertex $x\in V\setminus D_j$. Since $G_i$ has girth at least $4c\log n_i$ and $j\le \ell$, there exists $y\in V$ with $d_{G_i}(x,y)=j+2$. Define $F^{j+1}=F^j\cup\{\{x,y\}\}$. This creates a new cycle of length $j+3$.

By construction, the graph $G_i^{j+1}\eqdef (V_i,F^{j+1})$ contains a cycle $C_h$ of length $h$ for every $h\in \{3,\ldots,j+3\}$. Moreover, we claim that these cycles are embedded isometrically into the metric space $(V_i,d_{G_i^{j+1}})$. Indeed, due to the choice of $x$, if $h\in \{3,\ldots,j+2\}$ then  $$d_{G_i^{j}}(C_h,\{x,y\})> 2c\log n_i-(j+2),
$$
which is at least $h/2$ (the diameter of $C_h$) because $c\log n_i>7$. Thus the new edge $\{x,y\}$ does not change the isometric embeddability of $C_h$. The new cycle $C_{j+3}$ is isometrically embedded into $(V_i,d_{G_i})$ since the girth of $G_i$ is at least $4c\log n_i>2(j+2)$. Since $$d_{G_i^j}(M_j,C_{j+3})> 2c\log n_i-(j+2)>\frac{j+3}{2},$$
 The cycle $C_{j+3}$ remains isometrically embedded into $(V_i,d_{G_i^{j+1}})$. Note also that by construction the new edge $\{x,y\}$ is not incident to any vertex in $M_j$. Therefore the maximum degree of $(V_i, F^{j+1})$ remains $d+1$. This completes the inductive construction.

The degree of every vertex of $G_i^{\ell+1}$ is either $d$ or $d+1$. Add to every vertex of degree $d$ a self loop so as to obtain a $d+1$ regular graph $H_i=(V_i,F_i)$ without changing the induced shortest path metric. Note that~\eqref{eq:diam Hi} holds true because $D_\ell\neq V_i$.

 It follows from Lemma~\ref{lem:graphs containing} that for every kernel $K:X\times X\to [0,\infty)$,
$$
\gamma(H_i,K)\le \frac{d+1}{d}\gamma(G_i,K)\quad\mathrm{and}\quad \gamma_+(H_i,K)\le \frac{d+1}{d}\gamma_+(G_i,K).
$$
In particular, since $\{G_i\}_{i=1}^\infty$ is an expander sequence also $\{H_i\}_{i=1}^\infty$ is an expander sequence.

Assume for the sake of obtaining a contradiction that $\{(V_i,d_{H_i}\}_{i=1}^\infty$ admits a coarse embedding into $\{(V_i,\omega\circ d_{G_i}):\ (i,\omega)\in \N\times \c\}$. Then there exist $\{\omega_i\}_{i=1}^\infty\subseteq \c$ and nondecreasing moduli $\alpha,\beta:[0,\infty)\to [0,\infty)$ with
\begin{equation}\label{eq:alpha assumption}
\lim_{t\to \infty}\alpha(t)=\infty,
\end{equation}
and for every $i\in \N$ there exists $j(i)\in \N$ and $f_i:V_i\to V_{j(i)}$ satisfying
\begin{equation}\label{eq:fi contrapositive}
\forall\, u,v\in V(H_i),\quad \alpha\left(d_{H_i}(u,v)\right)\le \omega_i\left(d_{G_{j(i)}}(f_i(u),f_i(v))\right)\le \beta\left(d_{H_i}(u,v)\right).
\end{equation}
Note that only the values of $\beta$ on $\N\cup \{0\}$ matter here, and that since $\beta(\cdot)$ serves only as an upper bound in~\eqref{eq:fi contrapositive} we may assume without loss of generality that the sequence $\{\beta(n)\}_{n=0}^\infty$ is strictly increasing.

Define
\begin{equation}\label{eq:hi corrected}
h_i\eqdef \left\lfloor\frac13\min\left\{\beta^{-1}\left(\left\lfloor \omega_i\left(c\log n_{j(i)}\right)\right\rfloor\right),c\log n_i\right\}\right\rfloor.
\end{equation}
We claim that
\begin{equation}\label{eq:hi to infinity}
\lim_{i\to \infty} h_i=\infty.
\end{equation}
Indeed, since $\{G_j\}_{j=1}^\infty$ is an expander sequence,
$$
\lambda\eqdef \sup_{j\in \N}\lambda(G_j)<1.
$$
We therefore have the following bound on the diameter of $G_i$ (see~\cite{Chu89}):
\begin{equation}\label{eq:diameter of G_j cung}
\diam(G_j)\le \frac{2\log n_j}{\log(1/\lambda)}.
\end{equation}
Observe that since $G_j$ has girth at least $4c\log n_j$, it follows from~\eqref{eq:diameter of G_j cung} that
$c\log(1/\lambda)\le 1$. It now follows from~\eqref{eq:diam Hi}, \eqref{eq:fi contrapositive} and~\eqref{eq:diameter of G_j cung} that
\begin{equation}\label{eq:omegai big}
\alpha\left(\frac{c}{2}\log n_i\right)\le \omega_i\left(\frac{2\log n_{j(i)}}{\log(1/\lambda)}\right)\le \frac{4}{c\log(1/\lambda)}\omega_i\left(c\log n_{j(i)}\right),
\end{equation}
where in the rightmost inequality of~\eqref{eq:omegai big} we used the fact that $\omega_i$ is increasing and sub-additive. Due to~\eqref{eq:alpha assumption} and~\eqref{eq:hi corrected}, we indeed have~\eqref{eq:hi to infinity} as a consequence of~\eqref{eq:omegai big}.

Our construction ensures that $H_i$ contains a cycle $C\eqdef \{x_1,\ldots,x_{3h_i}\}$ of length $3h_i$ which is embedded isometrically into $(H_i,d_{H_i})$. Then
\begin{equation}\label{eq:cycle went into tree}
f_i(C)\stackrel{\eqref{eq:fi contrapositive}}{\subseteq} B_{G_{j(i)}}\left(f_i(x_1),\omega_i^{-1}\left(\beta(3h_i)\right)\right)\stackrel{\eqref{eq:hi corrected}}{\subseteq} B_{G_{j(i)}}\left(f_i(x_1),c\log n_{j(i)}\right).
\end{equation}
Since $c\log n_{j(i)}$ is smaller than half the girth of $G_{j(i)}$, the ball  $B_{G_{j(i)}}\left(f_i(x_1),c\log n_{j(i)}\right)$ is isometric to a tree. We will now proceed to show that combined with the inclusion~\eqref{eq:cycle went into tree} this leads to a contraction, using a coarse version of an argument of Rabinovich and Raz~\cite{RazR}.

Let $\overline C$ denote the one dimensional simplicial complex induced by $C$, i.e., in $\overline C$, which is isometric to the circle $\frac{3h_i}{2\pi}S^1$, all the edges of $C$ are present as intervals of length $1$. Similarly, denote by $\overline T$ the one dimensional simplicial complex induced by $B_{G_{j(i)}}\left(f_i(x_1),c\log n_{j(i)}\right)$ (thus $\overline T$ is isometric to a metric tree). Let $\overline f_i:\overline C\to \overline T$ be the linear interpolation of $f_i$, i.e., the extension of $f_i$ to $\overline C$ such that for every $u,v\in C$ with $\{u,v\}\in F_i$ the segment $[u,v]$ is mapped onto the unique geodesic $[f_i(u),f_i(v)]\subseteq \overline T$  with constant speed (see e.g. the discussion preceding Theorem 2 of~\cite{NS10}). It follows from~\eqref{eq:fi contrapositive} that $$d_{G_{j(i)}}(f_i(u),f_i(v))\le \omega_i^{-1}(\beta(1))$$ whenever $\{u,v\}$ is an edge of $H_i$. Hence $f_i$ is $\omega_i^{-1}(\beta(1))$-Lipschitz. Therefore $\overline f_i$ is also $\omega_i^{-1}(\beta(1))$-Lipschitz.

Consider the three paths $$\overline{f_i}([x_1,x_{h_{i}+1}]), \overline{f_i}([x_{h_i+1},x_{2h_i+1}]),\overline{f_i}([x_{2h_i+1},x_{1}])\subseteq \overline T.$$ Arguing as in~\cite{RazR}, since $\overline T$ is a metric tree, there must exist a common point
$$
p\in \overline{f_i}([x_1,x_{h_{i}+1}])\bigcap \overline{f_i}([x_{h_i+1},x_{2h_i+1}])\bigcap \overline{f_i}([x_{2h_i+1},x_{1}]).
$$
We can therefore find $$\left(\overline{a},\overline{b},\overline{c}\right)\in [x_1,x_{h_{i}+1}]\times [x_{h_i+1},x_{2h_i+1}]\times [x_{2h_i+1},x_{1}]$$ such that $$f_i\left(\overline{a}\right)=f_i\left(\overline{b}\right)=f_i\left(\overline{c}\right)=p.$$
By considering the closest points to $\overline{a},\overline{b},\overline{c}$ in $C$, there exist $a,b,c\in C$ such that
$$
\max\left\{d_{\overline C} \left(a,\overline a\right),d_{\overline C} \left(b,\overline b\right),d_{\overline C} \left(c,\overline c\right)\right\}\le \frac12,
$$
and
$$
\max\left\{d_{H_i}(a,b),d_{H_i}(a,c),d_{H_i}(b,c)\right\}\ge h_i.
$$
Without loss of generality we may assume that $d_{H_i}(a,b)=d_{\overline C}(a,b)\ge h_i$.

Since $\overline f_i$ is $\omega_i^{-1}(\beta(1))$-Lipschitz and $f\left(\overline{a}\right)=f\left(\overline{b}\right)$,
\begin{multline}\label{eq:to contrast alpha}
\alpha(h_i)\stackrel{\eqref{eq:fi contrapositive}}{\le} \omega_i\left(d_{G_{j(i)}}\left(f_i(a),f_i(b)\right)\right)\le \omega_i\left(d_{G_{j(i)}}\left(f\left(a\right),f\left(\overline a\right)\right)+ d_{G_{j(i)}}\left(f\left(b\right),f\left(\overline b\right)\right)\right)\\\le \omega_i\left( 2\omega_i^{-1}(\beta\left(1\right))\frac12\right)=\beta(1).
\end{multline}
The desired contradiction now follows by contrasting~\eqref{eq:alpha assumption} and~\eqref{eq:hi to infinity} with~\eqref{eq:to contrast alpha}.
\end{proof}

\subsection{A metric space failing calculus for nonlinear spectral
gaps}\label{sec:no-decay}

Let $(X,d_X)$ be a metric space and $p\in (0,\infty)$. Observe that
if $A=(a_{ij})$ is an $n\times n$ symmetric stochastic matrix then,
provided $X$ contains at least two points, the fact that
$\bpconst(A,d_X^p)<\infty$ implies that $A$ is ergodic, and
therefore
\begin{equation}\label{eq:ergodic}
\lim_{t\to \infty }\bpconst\left(A^t,d_X^p\right)=\lim_{t\to \infty
}\bpconst\left(\A_t(A),d_X^p\right)=1.
\end{equation}
Thus, we always have asymptotic decay of the Poincar\'e constants of
$A^t$ and $\A_t(A)$ as $t\to\infty$, but for the iterative
construction presented in this paper we need a quantitative variant
of~\eqref{eq:ergodic}. At the very least, we need  $(X,d_X^p)$ to
admit the following type of {\em uniform decay} of the Poincar\'e
constant.
\begin{definition}[Spaces admitting uniform decay of Poincar\'e
constants] Let $X$ be a set and $K:X\times X\to [0,\infty)$ a
kernel. Say that $(X,K)$ has the {\em uniform decay property} if for
every $M\in (1,\infty)$ there exists $t\in \N$ and $\Gamma\in
[1,\infty)$ such that for every $n\in \N$ and every $n\times n$
symmetric stochastic matrix $A$,
$$
\gamma_+(A,K)\ge \Gamma\implies \gamma_+(\A_t(A),K)\le
\frac{\gamma_+(A,K)}{M}.
$$
\end{definition}

We now show that there exists a  metric space $(X,d_X)$ such that
$(X,d_X^2)$ does not have the uniform decay property.

\begin{proposition}\label{prop:example}
There exist a metric space $(X,\rho)$ and a universal constant
$\eta\in (0,\infty)$ with the following property. For every $n\in
\N$ there is an $n$-vertex regular graph $G_n=(V_n,E_n)$ such that
$\lim_{n\to \infty}\bpconst(G_n,\rho^2)=\infty$, yet for every $t\in
\mathbb N$ there exists $n_0\in \N$ such
$$
n\ge n_0\implies  \bpconst(\A_t(G_n),\rho^2)\ge \eta\cdot
\bpconst(G_n,\rho^2).
$$
\end{proposition}

\begin{proof}
Define
$$
X\eqdef\ell_\infty \cap \mathbb Z^{\aleph_0},
$$
i.e., $X$ is the set of all integer-valued bounded sequences.
Consider the following metric $\rho:X\times X\to [0,\infty)$.
\begin{equation}\label{eq:def rho}
\rho(x,y)\eqdef \log\left(1+\|x-y\|_\infty\right).
\end{equation}
Note that $\rho$ is indeed
 a metric since the mapping $T:[0,\infty)\to [0,\infty)$ given by
 $$
 T(s)\eqdef \log(1+s)
$$
  is concave, increasing and $T(0)=0$.

Let $G_n=(V_n,E_n)$ be an arbitrary sequence of constant degree
expanders, i.e., $G_n$ is an $n$-vertex graph of degree $d$ (say
$d=4$) satisfying
$$
C\eqdef \sup_{n\in \N} \bpconst(G_n,\|\cdot\|_2^2)<\infty.
$$
We claim that
\begin{equation}\label{eq:log log goal}
\bpconst(G_n,\rho^2)\lesssim (\log(1+\log n))^2.
\end{equation}
The goal is to prove
 that every $f,g:G_n\to X$ satisfy
 \begin{equation*}
  \frac1{n^2} \sum_{(u,v)\in V_n\times V_n} \rho(f(u),g(v))^2
  \lesssim \frac{\left(\log (1+\log n)\right)^2}{nd}\sum_{(u,v)\in E_n} \rho(f(u),g(v))^2 .
  \end{equation*}
To this end write $$ S_n\eqdef f(V_n)\cup g(V_n)\subseteq\mathbb
Z^{\aleph_0}.
$$
By Bourgain's embedding theorem~\cite{Bourgain-embed}, applied to
the metric space $(S_n,\ell_\infty)$, there exists $\beta:S_n\to
\ell_2$ satisfying
 \begin{equation} \label{eq:limitations-1}
\forall\, u,v\in V_n,\quad   \|f(u)-g(v)\|_\infty \le
\|\beta(f(u))-\beta(g(v))\|_2
  \le c(1+\log n)  \|f(u)-g(v)\|_\infty,
  \end{equation}
where $c\in (1,\infty)$ is a universal constant. For every $u,v\in
V_n$ we have
\begin{multline}\label{eq:use bourgain}
\rho(f(u),g(v))\stackrel{\eqref{eq:def rho}\wedge
\eqref{eq:limitations-1}}{\le}
\log\left(1+\|\beta(f(u))-\beta(g(v))\|_2\right)\\
\stackrel{\eqref{eq:limitations-1}}{\le}\log\left(1+c(1+\log
n)\|f(u)-g(v)\|_\infty\right)\stackrel{\eqref{eq:def rho}}{\lesssim}
\log(1+\log n)\cdot\rho(f(u),g(v)),
\end{multline}
where in the last step of~\eqref{eq:use bourgain} we used the fact
that if $f(u)\neq g(v)$ then $\|f(u)-g(v)\|_\infty\ge 1$.

  As shown in~\cite[Remark~5.4]{MN-quotients}, there exists a universal
  constant $\kappa>1$ and a mapping $\phi:\ell_2\to \ell_2$ such that
\begin{equation}\label{eq:back to ell_}
\forall\, x,y\in \ell_2,\quad T\left(\|x-y\|_2\right)\le
  \|\phi(x)-\phi(y)\|_2\le \kappa T\left(\|x-y\|_2\right).
\end{equation}
A combination of~\eqref{eq:limitations-1}, \eqref{eq:use bourgain} and~\eqref{eq:back to
ell_} implies that
  the mapping $\psi=\phi\circ \beta:S_n\to \ell_2$ satisfies
  \begin{equation*}
  \forall\, u,v\in V_n \quad \rho(f(u),g(v)) \le \|\psi(f(u))-\psi(g(v))\|_2
   \lesssim \log(1+\log n) \cdot \rho(f(u),g(v)) ,
   \end{equation*}
Since $\gamma_+(G_n,\|\cdot\|_2^2)\le C$, we conclude that
  \begin{multline*}
  \frac1{n^2}  \sum_{(u,v)\in V_n\times V_n} \rho(f(u),g(v))^2
    \le
  \frac1{n^2} \sum_{(u,v)\in V_n\times V_n} \|\psi(f(u))-\psi(g(v))\|_2^2 \\
   \le
   \frac{C}{nd}\sum_{(u,v)\in E_n } \|\psi(f(u))-\psi(g(v))\|_2^2
   \lesssim \frac{(\log(1+\log n))^2}{nd}\sum_{(u,v)\in  E_n}  \rho(f(u),g(v))^2.
  \end{multline*}
This completes the proof of~\eqref{eq:log log goal}.

We will now bound $\bpconst(\A_t(G_n),\rho^2)$ from below. For this
purpose it is sufficient to examine a specific embedding of the
graph $\A_t(G_n)$ into $X$. Let $\f:V_n\to \mathbb Z^{\aleph_0}$ be
an isometric embedding of the shortest path metric on $\A_t(G_n)$
into $(\mathbb Z^{\aleph_0},\|\cdot\|_\infty)$. If $\{u,v\}\in
E(\A_t(G_n))$ then
$\rho(\f(u),\f(v))=T(\|\f(u)-\f(v)\|_\infty)=T(1)=1$. On the other
hand, since the degree of $\A_t(G)$ is $td^t$, at least half of the
pairs in $V_n\times V_n$ are at distance $\gtrsim \frac{\log n}{t
\log d}$ in the shortest path metric metric on $\A_t(G)$. Hence for
at least half of the pairs $(u,v)\in V_n\times V_n$ we have
$$
\rho(\f(u),\f(v))\ge \log\left(1+\xi \frac{\log n}{t\log d}\right),
$$
where $\xi\in (0,\infty)$ is a universal constant. If
$$
n\ge e^{(t\log d)^2}
$$
then we deduce that
\[
\bpconst(\A_t(G_n),\rho^2)\ge \frac{\frac{1}{n^2}\sum_{(u,v)\in
V_n\times V_n}\rho(\f(u),\f(v))^2}{\frac{1}{ntd^t}\sum_{(u,v)\in
E(\A_t(G_n))}\rho(\f(u),\f(v))^2}\gtrsim (\log(1+\log
n))^2\stackrel{\eqref{eq:log log goal}}{\gtrsim}
\bpconst(G_n,\rho^2),
\]
thus completing the proof of Proposition~\ref{prop:example}.
\end{proof}

\begin{remark}\label{matousek}
Using Matou\v{s}ek's $L_p$-variant of the Poincar\'e inequality for
expanders~\cite{Mat97}, the proof of Proposition~\ref{prop:example}
extends mutatis mutandis to show that $(X,d_X^p)$ fails to have the
uniform decay property for any $p\in (0,\infty)$.
\end{remark}

\begin{remark}\label{rem:banach?}
We do not know if there exists a normed space which does not have the uniform decay property, though we conjecture that such spaces do exist, and that this even holds for $\ell_\infty$. Note that despite the fact that all separable metric spaces embed into $\ell_\infty$, we cannot formally deduce from Proposition~\ref{prop:example} that $\ell_\infty$ satisfies the same conclusion since the uniform decay property of the Poincar\'e constant is not necessarily monotone when passing to subsets of metric spaces. We suspect that $(\ell_1,\|\cdot\|_1^2)$ does have the uniform decay property despite the fact that  $\ell_1$ does not admit an equivalent uniformly convex norm.
\end{remark}

\bigskip

\noindent{\bf Acknowledgments}
Michael Langberg was involved in early discussions on the analysis of the zigzag product.
Keith Ball helped in simplifying this analysis. We thank Steven Heilman, Michel Ledoux, Mikhail Ostrovskii and Gideon Schechtman for helpful suggestions. We are also grateful to two anonymous referees for their careful reading of this paper and many helpful comments. An extended abstract
     announcing parts of this work, and titled ``Towards a calculus for nonlinear spectral gaps", appeared in
     Proceedings of the Twenty-First Annual ACM-SIAM Symposium
               on Discrete Algorithms (SODA 2010). M. M. was partially supported by ISF grants 221/07 and  93/11,
BSF grants 2006009 and 2010021, and a gift from Cisco Research Center. Part of this work was completed while M.M. was a member of the Institute for Advanced Study at Princeton NJ., USA. A. N.
was partially supported by NSF grant CCF-0832795, BSF grants
2006009 and 2010021, the Packard Foundation and the Simons Foundation. Part of this work was completed while A. N. was visiting Universit\'e Pierre et Marie Curie, Paris, France.

\bibliographystyle{abbrv}
\bibliography{zigzag}

 \end{document}